\documentclass[aos]{imsart}

\RequirePackage[OT1]{fontenc}
\RequirePackage{amsthm,amsmath,amsfonts}
\RequirePackage[numbers]{natbib}

\usepackage{xr}
\RequirePackage[colorlinks,citecolor=blue,urlcolor=blue]{hyperref}

\usepackage{algorithm}
\usepackage{algorithmic}
\usepackage{mathrsfs}

\startlocaldefs
\numberwithin{equation}{section}
\theoremstyle{plain}
\newtheorem{defn}{Definition}[section]
\newtheorem{thm}{Theorem}[section]
\newtheorem{prop}{Proposition}[section]
\newtheorem{lemm}{Lemma}[section]
\newtheorem{rem}{Remark}[section]
\newtheorem{assu}{Assumption}[section]
\newtheorem{example}{Example}[section]

\newenvironment{prf}{{Proof. }}{\qed}
\endlocaldefs

\newcommand\independent{\protect\mathpalette{\protect\independenT}{\perp}}
\def\independenT#1#2{\mathrel{\rlap{$#1#2$}\mkern2mu{#1#2}}}

\newcommand{\F}{F}
\newcommand{\G}{G}
\newcommand{\D}{{\mathrm{D}}}

\newcommand{\eps}{\varepsilon}

\newcommand{\EE}{\mathbb{E}}

\newcommand{\Var}{\mbox{Var}}

\newcommand{\PX}{\mathbb{P}}
\newcommand{\PXO}{\mathbb{P}}
\newcommand{\PXn}{\mathbb{P}_n}
\newcommand{\DPX}{\mathscr{D}(\PX)}

\newcommand{\DPXhat}{\widehat{\mathscr{D}}_{n,p}}
\newcommand{\DPXO}{\mathscr{D}(\PXO)}
\newcommand{\CDO}{\mathscr{C}(D^0)}
\newcommand{\CDPX}{\mathscr{C}(\DPX)}

\newcommand{\GDPX}{G_{\DPX}}

\newcommand{\Fi}{\mathcal{F}_i}

\newcommand{\Kinit}{\mathcal{K}^{\text{init}}}

\newcommand{\nrep}{n_{\text{rep}}}
\newcommand{\pl}{p_{\text{lin}}}
\newcommand{\pc}{p_{\text{c}}}
\newcommand{\mK}{\mathcal{K}}
\newcommand{\mKt}{\widetilde{\mathcal{K}}}

\newcommand{\argmin}{\mathrm{argmin}}
\newcommand{\argminn}{\operatornamewithlimits{argmin}}
\newcommand{\pa}{\mathrm{pa}}

\newcommand\Dominik[1]{{\color{orange}Dominik: ``#1''}}

\usepackage{tikz}
\usetikzlibrary{arrows,automata,decorations.markings}
\usetikzlibrary{patterns,snakes}
\usetikzlibrary{positioning,fit}
\usetikzlibrary{calc}

\tikzstyle{vertex}=[circle, draw, fill, inner sep=0pt, minimum size=0.15cm]

\usepackage{enumerate}

\begin{document}

\begin{frontmatter}
\title{
	Causal inference in partially linear structural equation models}
\runtitle{Causal inference in PLSEMs}
 
\begin{aug}
\author{\fnms{Dominik} \snm{Rothenh\"ausler}\thanksref{t1}\ead[label=e2]{rothenhaeusler@stat.math.ethz.ch}},
\author{\fnms{Jan} \snm{Ernest}\thanksref{t1}\thanksref{t2}\ead[label=e1]{jan.ernest@alumni.ethz.ch}}
\and
\author{\fnms{Peter} \snm{B\"uhlmann}
\ead[label=e3]{buehlmann@stat.math.ethz.ch}
\ead[label=u1,url]{http://stat.ethz.ch}}
\thankstext{t1}{These authors contributed equally to this work.}
\thankstext{t2}{Supported in part by the Max Planck ETH Center for Learning Systems and by the Swiss National Science Foundation grant no. 2-77991-14}
\runauthor{D. Rothenh\"ausler, J. Ernest and P. B\"uhlmann}
\affiliation{ETH Z\"urich}
\address{Seminar for Statistics\\
ETH Z\"urich\\
R\"amistrasse 101\\
8092 Z\"urich\\
Switzerland\\
\printead{e2} \\
\phantom{E-mail:\ }\printead*{e1} \\
\phantom{E-mail:\ }\printead*{e3} \\
%
\printead{u1}}
\end{aug}

\begin{abstract}
We consider 
 identifiability 
of partially linear additive structural equation models with Gaussian noise (PLSEMs) and estimation of distributionally equivalent models to a given PLSEM. Thereby, we also include robustness results for errors in the neighborhood of Gaussian distributions.
Existing identifiability results in the framework 
of additive SEMs with Gaussian noise
are limited to linear and nonlinear SEMs, 
which can be considered as special cases of PLSEMs with vanishing nonparametric or parametric part, respectively.
We close the wide gap between these two special cases by providing a comprehensive theory of the identifiability of PLSEMs by means of (A) a graphical, (B) a transformational, (C) a functional and (D) a causal ordering characterization of 
PLSEMs that generate a given distribution $\PX$.  
In particular, the characterizations (C) and (D) answer the fundamental question to which extent 
nonlinear functions 
in additive SEMs with Gaussian noise
restrict the set of potential causal models and hence influence the identifiability.

On the basis of the transformational characterization (B) we provide a score-based estimation procedure that
outputs
the graphical representation (A) of the distribution equivalence class of a given PLSEM. We derive its (high-dimensional) consistency and demonstrate its performance on simulated datasets. 
\end{abstract}

\end{frontmatter}

\section{Introduction} \label{M-sec:intro}
Causal inference is fundamental in many scientific disciplines. Examples include the identification of causal molecular mechanisms in genomics~\cite{statnikov2012, steketal12}, the investigation of causal relations among activity in brain regions from fMRI data~\cite{ramsey2010} or the search for causal associations in public health \cite{glass2013}. 

A major research topic in causal inference aims at establishing causal dependencies based on purely observational data. The notion ``observational'' commonly refers to the fact that one obtains the data from the system of variables under consideration without subjecting it to external manipulations. Typically, one then assumes that the observed data has been generated by an underlying causal model and tries to draw conclusions about its structure. 

Two main research tasks in this setting are identifiability and estimation of the underlying causal model. 
We consider  identifiability of partially linear additive structural equation models with Gaussian noise (PLSEMs) and estimation of distributionally equivalent models to a given PLSEM. Thereby, we also include robustness results for errors in the neighborhood of Gaussian distributions. 

So far, there exists a wide ``identifiability gap'' for PLSEMs, as
their identifiability has only been characterized for the two special cases where all the functions are linear or all the functions are nonlinear. We close this ``identifiability gap'' by providing comprehensive characterizations of the identifiability 
of the general class of PLSEMs from various perspectives.

Unlike in regression where partially linear models are mainly studied because of efficiency gains in estimation, the use of partially linear models has a deeper meaning in causal inference. In fact, as we will show, it is closely connected to identifiability.  
The functional form of an additive component directly influences the identifiability of the corresponding (and also other) causal relations. 
For this reason we strongly believe that the 
understanding of the identifiability 
of PLSEMs is important.
First and foremost, it raises the awareness of 
potentially limited (or increased) identifiability in the presence of linear (or nonlinear) relations in the data. Second, by not restricting the functions to be either all linear or all nonlinear, PLSEMs allow for a flexible modeling approach.


We start by reviewing and introducing important concepts in Section~\ref{M-ssec:probdesc}. We then provide a brief overview of related work in Section~\ref{M-ssec:relwork} and explicitly state the main contributions of this paper in Section~\ref{M-ssec:contrib}.

\subsection{Problem description and important concepts} \label{M-ssec:probdesc}

We consider $p$ random variables $X=(X_1,...,X_p)$ with joint distribution $\PX$, which is assumed to be Markov with respect to an underlying directed acyclic graph (DAG). 
A DAG $D=(V,E)$ is an ordered pair consisting of a set of vertices $V=\{1,...,p\}$ associated with the variables $\{X_1,...,X_p\}$, and a set of directed edges $E \subset V^2$ such that there are no directed cycles. A directed edge between the nodes $i$ and $j$ in $D$ is denoted by $i \rightarrow j$. 
Node $i$ is called a \emph{parent} of node $j$ and $j$ is called a \emph{child} of $i$. Moreover, the edge is said to be oriented \emph{out of} $i$ and \emph{into} $j$. If $i \rightarrow j$ or $i \leftarrow j$, $i$ and $j$ are called \emph{adjacent} and the edge is \emph{incident} to $i$ and $j$. The \emph{degree} of a node $i$, denoted by $\deg_D(i)$, counts the number of edges incident to node $i$ in DAG $D$. A node $k$ that can be reached from $i$ by following directed edges is called \emph{descendant} of~$i$. We use the convention that any node is a descendant of itself. The set $\pa_D(j) = \{i \ | \ i \rightarrow j \text{ in } D \}$ consists of all parents of node $j$. 
The multi-index notation $X_{\pa_D(j)}$ denotes the set of 
variables 
$\{X_i\}_{i \in \pa_D(j)}$.
An edge $i \rightarrow j$ is said to be \emph{covered} in $D$, if $\pa_D(i) = \pa_D(j) \setminus \{i\}$. In that case, $\pa_D(i)$ is a \emph{cover} for edge $i \rightarrow j$. The process of changing the orientation of a covered edge from $i \rightarrow j$ to $i \leftarrow j$ is referred to as a \emph{covered edge reversal}. A triple $(i,j,k)$ is called a \emph{$v$-structure}, if $\{i,j\} \subseteq \pa_D(k)$ and $i$ and~$j$ are not adjacent.
The graph obtained by replacing all directed edges $i \rightarrow j$ by undirected edges $i \text{ --- } j$ is called \emph{skeleton} of $D$. The \emph{pattern} of a DAG $D$ is the graph 
with the same skeleton as $D$ and $i \rightarrow j$ is directed if and only if it is part of a $v$-structure in $D$. 
A permutation $\sigma: V \rightarrow V$ is a \emph{causal ordering} of~$D$ if $\sigma(i) < \sigma(j)$ for all $i \rightarrow j$ in~$D$.
DAGs 
may be used as underlying structures for structural equation models (SEMs). 
A SEM relates the distribution of every random variable $\{X_1,...,X_p\}$ to the distribution of its direct causes (the parents in the corresponding DAG~$D$) and random noise. In its most general form, 
\begin{equation} \label{M-eq:SEMgeneral}
	X_j = f_j(X_{\pa_D(j)}, \eps_j), \qquad j=1,...,p ,
\end{equation}
where $\{f_j\}_{j=1,...,p}$ are functions from $\mathbb{R}^{|\pa_D(j)|+1} \rightarrow \mathbb{R}$ and $\{\eps_j\}_{j=1,...,p}$ are mutually independent noise variables. 
Lastly, for a function $\F:\mathbb{R}^{p} \rightarrow \mathbb{R}^{p}$, we write $\mathrm{D} \F$ for the Jacobian of $\F$.

\subsubsection{Main focus: PLSEMs} \label{M-ssec:probdescAddGaussSEMandDEC}

In this paper we study the restriction of the general SEM in equation~\eqref{M-eq:SEMgeneral} 
to \emph{partially linear additive SEMs with Gaussian noise (PLSEMs)} of the form: 
\begin{align} \label{M-eq:PLSEM}
	X_j & = \mu_j + \sum\limits_{i \in \pa_D(j)} f_{j,i}(X_i) + \eps_j, 
\end{align}
where $\mu_j \in \mathbb{R}$, $f_{j,i} \in C^2(\mathbb{R})$, $f_{j,i} \not\equiv 0$, with $\EE[f_{j,i}(X_i)]=0$, and $\eps_j \sim \mathcal{N}(0, \sigma_j^2)$ with $\sigma_j^2 > 0$ for $j=1,...,p$. Likewise, we may write
\begin{align*} 
	X_j & = \mu_j + \sum\limits_{i \in \pa^{\text{L}}_D(j)} \alpha_{j,i} X_i + \sum\limits_{i \in \pa^{\text{NL}}_D(j)} f_{j,i}(X_i) + \eps_j, 
\end{align*}
with $\alpha_{j,i} \in \mathbb{R}\setminus \{0\}$, $\mu_j$, $f_{j,i}$, $\eps_j$ as above, $\pa^{\text{L}}_D(j) \cup \pa^{\text{NL}}_D(j) = \pa_D(j)$ and $\pa^{\text{L}}_D(j) \cap \pa^{\text{NL}}_D(j) = \emptyset$. Note that we do not \emph{a priori} fix the sets $\pa^{\text{L}}_D(j)$ and $\pa^{\text{NL}}_D(j)$. 
For~$\PX$ generated by a PLSEM with DAG $D$, the PLSEM corresponding to $D$ is unique (cf. Lemma~\ref{S-le:uniquePLSEM} in the supplement).
Therefrom, we call an edge $i \rightarrow j$ in $D$ a \mbox{\emph{(non-)linear edge}}, if $f_{j,i}$ in the  PLSEM corresponding to $D$ is (non-)linear. 
Note that the concept of \mbox{(non-)linearity} of an edge is defined with respect to a specific DAG $D$. 
Depending on the orientations of other edges, the status of an edge $i \rightarrow j$ may change from linear to nonlinear. An example is given in Figure~\ref{M-fig:edge-status-change}.
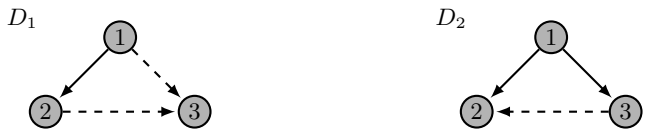
\begin{figure}[h] 
\begin{tikzpicture}[->,>=latex,shorten >=1pt,auto,node distance=1.4cm,
                    thick]
  \tikzstyle{every state}=[fill=black!30,draw=black,text=black, inner sep=0.4pt, minimum size=12pt]
  
  \node[state] (X1a) {$1$};
  \node[state] (X2a) [below left of=X1a] {$2$};
  \node[state] (X3a) [below right of=X1a] {$3$};
  
  \path (X1a) edge [dashed] (X3a)
  		(X1a) edge  (X2a)
  		(X2a) edge [dashed]  (X3a);

  \node[state] (X2b) [right=3.3cm of X3a] {$2$};
  \node[state] (X1b) [above right of=X2b] {$1$};
  \node[state] (X3b) [below right of=X1b] {$3$};

  \path (X1b) edge (X2b)
  		(X1b) edge (X3b)
  		(X3b) edge [dashed] (X2b);
   
      \coordinate [label=above:$D_1$] (C) at (-1.3,0.0);	
      \coordinate [label=above:$D_2$] (C) at (4.4,0);	
  
\end{tikzpicture}
\caption{Two DAGs $D_1$ and $D_2$ with linear edges (dashed) and nonlinear edges (solid). Let us give a brief outlook: let $\PX$ be generated by a PLSEM with DAG $D_1$. In this paper we prove that there exists a PLSEM with DAG $D_2$ that generates the same distribution~$\PX$. Moreover, 
we show that 
$D_1$ and $D_2$ are the only two DAGs with a corresponding PLSEM that generates~$\PX$. For now, simply note that $1 \rightarrow 3$ is linear in $D_1$, but nonlinear in $D_2$.}
\label{M-fig:edge-status-change}
\end{figure}

The restriction to additive SEMs is interesting from both a statistical and computational perspective 
as the estimation of additive functions is well understood and one largely avoids the curse of dimensionality. 
The assumption of Gaussian noise is necessary for our characterization results in Section~\ref{M-sec:char}. 
In fact, identifiability properties may deteriorate in partially linear models with arbitrary noise distributions, see Section~\ref{M-ssec:relworkOther}. We therefore consider PLSEMs to be among the most general SEMs with reasonable estimation properties. For an extension to error distributions in the neighborhood of the Gaussian distribution,  see Section~\ref{M-sec:misspec}.


\subsubsection{Main task: characterization of all PLSEMs that generate $\PX$}
\label{M-ssec:main-task}

The main task of this paper is the systematic characterization of all PLSEMs that generate a given distribution $\PX$ under very general assumptions. In particular: how do edge functions in different PLSEMs relate to each other? How does changing a single linear edge to a nonlinear edge affect the set of potential underlying PLSEMs? Do causal orderings of different DAGs corresponding to PLSEMs that generate $\PX$ share certain properties? 



Under faithfulness, it may be natural to characterize all PLSEMs that generate~$\PX$ by their corresponding DAGs as they are restricted to a subset of the Markov equivalence class (see Section~\ref{M-ssec:relworkGeneral}). 
For a distribution~$\PX$ that has been generated by a faithful PLSEM, we call the set of DAGs
\begin{align*} 
\begin{split}
	\DPX := \left\{ D \ \begin{array}{|l}  \PX \text{ is faithful to } D \text{ and there exists a} \\ \text{PLSEM with DAG } D \text{ that generates } \PX  \end{array} \right\} 
\end{split}
\end{align*}
the \emph{(PLSEM) distribution equivalence class}. 
Can we build on characterizations of the Markov equivalence class to characterize $\DPX$? For example, can $\DPX$ 
also 
be graphically represented by a 
single 
PDAG? Is it possible to efficiently estimate~$\DPX$? Before we explain our approaches to answer these questions in Section~\ref{M-ssec:contrib}, let us briefly summarize related work.

\subsection{Related work} \label{M-ssec:relwork}

First, in Section~\ref{M-ssec:relworkGeneral}, we discuss the identifiability of general SEMs. 
We then motivate why our theoretical results close a relevant ``gap'' by reviewing existing identifiability results for two special cases of PLSEMs where either all the functions $f_{j,i}$ are linear (Section~\ref{M-ssec:relworkLinGauss}) or 
nonlinear (Section~\ref{M-ssec:relworkNonlinGauss}). Finally, we briefly comment on the assumption of Gaussian noise in Section~\ref{M-ssec:relworkOther}.
 

\subsubsection{Identifiability of general SEMs} \label{M-ssec:relworkGeneral}

In the general SEM as defined in equation~\eqref{M-eq:SEMgeneral} one cannot draw any conclusions about $D$ given $\PX$ without making further assumptions.
One such assumption commonly made is faithfulness (cf. Section~\ref{M-ssec:char_faith}). 
Under faithfulness, one can identify the Markov equivalence class of $D$ (a set of DAGs that all entail the same conditional independences), see, for example,~\citep{pearl09}. Markov equivalence classes 
 are well-characterized. In fact, the Markov equivalence class of a DAG $D$ consists of all DAGs with the same skeleton and v-structures as~$D$~\citep{vermapearl90} and can be graphically represented by a single partially directed graph (cf.~Section~\ref{M-ssec:char_faith}). Moreover, any two Markov equivalent DAGs can be transformed into each other by a sequence of distinct covered edge reversals~\citep{chickering95}. 

The estimation of the general SEM is difficult due to the curse of dimensionality in fully nonparametric estimation. 
In combination with the unidentifiability,
this motivates
the use of restricted SEMs, 
which have better estimation properties and for which it is possible to achieve 
(partial) identifiability of the SEM (even without assuming faithfulness), see Section~\ref{M-ssec:char_unfaith} or~\citep{petersetal13} for an overview.

\subsubsection{Special case of PLSEM: Linear Gaussian SEM} \label{M-ssec:relworkLinGauss}

A widespread specification of PLSEMs are linear Gaussian SEMs,
which have the same identifiability properties as the general SEMs: without additional assumptions they are unidentifiable, whereas under faithfulness, 
their distribution equivalence class
equals the Markov equivalence class, 
see, for example,~\citep{spirtes2016}.

The estimation of the Markov equivalence class of linear Gaussian SEMs in the low-dimensional case has been addressed in 
e.g.~\citep{sgs00, chick02}, whereas the high-dimensional scenario (requiring sparsity of the true underlying DAG) is discussed in e.g.~\citep{kabu07, sarpet12, pbcausal13, nandy15}. 

An exception of identifiability of linear Gaussian SEMs occurs if all $\eps_j$ have equal variances $\sigma_j^2 = \sigma^2 > 0, \forall j$. Under this assumption, the true underlying DAG $D$ is identifiable~\cite{petbu13}. Yet, the assumption of equal noise variances seems to be overly restrictive in many scenarios. 
In general, the linearity assumption may be rather restrictive if not implausible in some cases.

\subsubsection{Special case of PLSEM: Nonlinear additive SEM with Gaussian noise}  \label{M-ssec:relworkNonlinGauss}
Interestingly, the assumption of exclusively nonlinear functions $f_{j,i}$ in equation~\eqref{M-eq:PLSEM} greatly improves the identifiability properties, see~\citep{hoy09} for the bivariate case and~\citep{petersetal13} for a general treatment. In fact, if 
%
all $f_{j,i}$ 
are nonlinear and three times differentiable, $\DPX$ only consists of the single true underlying DAG $D$~\citep[Corollary 31 (ii)]{petersetal13}. 
The nonlinearity assumption is crucial, though. The authors provide an example where two DAGs are distribution equivalent if one of the nonlinear functions is replaced by a linear function~\citep[Example 26]{petersetal13}. 

Various estimation methods have been introduced for additive nonlinear SEMs to infer the underlying DAG~\citep{petersetal13,vdg13,nowzopb13}. In particular, a restricted maximum likelihood estimation method called CAM, which is consistent in the low- and high-dimensional setting (assuming a sparse underlying DAG), has been proposed specifically for nonlinear additive SEMs with Gaussian noise~\citep{pbjoja13}. 


\subsubsection{Identifiability of PLSEMs with non-Gaussian or arbitrary noise} \label{M-ssec:relworkOther}
The identifiability
properties
of linear SEMs 
generally improve if one allows for non-Gaussian noise distributions. In fact, if
all but one $\eps_j$ are assumed to be non-Gaussian (commonly referred to as LiNGAM setting), 
the underlying DAG~$D$ is identifiable~\citep{shim06}. 
A general theory for linear SEMs with arbitrary noise distributions is presented in~\citep{hoy08PCLingam}. Both papers also propose estimation procedures for the respective model classes. 
 
Unfortunately, the situation is different 
for PLSEMs:
identifiability can be lost if one considers PLSEMs with non-Gaussian (or arbitrary) noise distributions. This can be seen from a specific example of a bivariate linear SEM with Gumbel-distributed noise, which is identifiable in the LiNGAM framework, but for which there exists a nonlinear additive backward model~\citep{hoy09}. Still, this example seems to be rather particular. In fact, for bivariate additive SEMs, all unidentifiable cases of additive models can be classified into five categories, see~\citep{zhang09, petersetal13}. Based on bivariate identifiability, it has been shown that one can conclude multivariate identifiability under an additional assumption referred to as IFMOC assumption~\citep{peters11}. For instance, this approach allows to conclude identifiability of the multivariate LiNGAM and CAM settings and as such covers settings with both, Gaussian or non-Gaussian noise and all  linear or all nonlinear functions. However, it is less explicit than the results presented in Section~\ref{M-sec:char}. In particular, it does not allow for a characterization of the distribution equivalence class of a PLSEM with Gaussian noise where some of the edge functions are linear and some are nonlinear.


\subsection{Our contribution} \label{M-ssec:contrib}

As discussed in Section~\ref{M-ssec:relwork}, there exists a wide ``identifiability gap'' for PLSEMs.
Their identifiability has only been studied for the two special cases of linear SEMs and entirely nonlinear additive SEMs. Moreover, to the best of our knowledge, it has not yet been understood to what extent (single) nonlinear functions in additive SEMs with Gaussian noise 
restrict the 
underlying 
causal model.
We close the ``identifiability gap'' for PLSEMs and answer the questions raised in Section~\ref{M-ssec:main-task} with the following theoretical results:
\begin{enumerate}
\item[(A)] A graphical representation of $\DPX$ with a single partially directed graph $\GDPX$ in Section~\ref{M-ssec:char_graphical} (analogous to the use of CPDAGs to represent Markov equivalence classes).
\item[(B)] A transformational characterization of $\DPX$ through sequences of covered \mbox{\emph{linear}} edge reversals in Section~\ref{M-ssec:char_trafo} (analogous to the 
characterization of Markov equivalence classes via sequences of covered edge reversals in~\citep{chickering95}).
\item[(C)] A functional characterization of PLSEMs in Section~\ref{M-ssec:funcchar}: 
all PLSEMs that generate the same distribution $\PX$ are constant rotations of each other.
\item[(D)] A causal orderings characterization of PLSEMs  
in Section~\ref{M-ssec:orderchar}. In particular, it precisely specifies to what extent nonlinear functions in PLSEMs restrict the set of potential causal orderings.
\end{enumerate}
The first two characterizations hold only under faithfulness, the third and fourth are general. We will give details on the precise interplay between nonlinearity and faithfulness in Section~\ref{M-ssec:interplay}.
Building on the transformational characterization result in (B) we provide an efficient score-based estimation procedure that outputs the graphical representation $\GDPX$ in (A) given $\PX$ and one DAG $D \in \DPX$. The proposed algorithm only relies on sequences of local transformations and score computations and hence is feasible for large graphs with numbers of variables in the thousands (assuming reasonable sparsity). We demonstrate its performance on simulated data. Moreover, we provide some robustness results for identifiability in the neighborhood of Gaussian noise and we derive (high-dimensional) consistency based on the consistency proof of the CAM methodology in~\citep{pbjoja13}.

\section{Comprehensive characterization of PLSEMs} \label{M-sec:char}
In this section we present our main theoretical results. They consist of 
characterizations of PLSEMs that generate a given distribution $\PX$ from various perspectives.
In Section~\ref{M-ssec:char_faith} we assume that $\PX$ is faithful to the underlying causal model and demonstrate that this leads to a transformational characterization and a graphical representation of $\DPX$ very similar to the well-known counterparts characterizing a Markov equivalence class. Our main theoretical contributions, which hold under very general assumptions and, in particular, do not rely on the faithfulness assumption, are presented in Section~\ref{M-ssec:char_unfaith}. 
They fully characterize all PLSEMs that generate a given distribution~$\PX$ on a functional level. Moreover, they explain how nonlinear functions impose very specific restrictions on the set of potential causal orderings.
%
Section~\ref{M-ssec:interplay} brings together the two previous sections by discussing the precise interplay of nonlinearity and faithfulness.

\subsection{Characterizations of $\DPX$ under faithfulness} \label{M-ssec:char_faith}

Let $\PX$ be generated by a PLSEM with DAG $D \in \DPX$. The goal of this section is to characterize $\DPX$. Recall that $\DPX$ is the set of all DAGs~$D$ such that $\PX$ is faithful to $D$ and there exists a PLSEM with DAG $D$ that generates~$\PX$. 
In words, faithfulness means that no conditional independence relations other than those entailed by the Markov property hold, see e.g.~\citep{sgs00}. 
In particular, it implies that 
$\DPX$ is a subset of the Markov equivalence class and all DAGs in $\DPX$ have the same skeleton and $v$-structures~\citep{vermapearl90}. Markov equivalence classes can be graphically represented with single graphs, known as CPDAGs (also referred to as essential graphs, maximally oriented graphs or completed patterns)~\citep{chickering95, andersson1997, meek1995, vermapearl90}, where an edge is directed if and only if it is oriented the same way in all the DAGs in the Markov equivalence class, else, it is undirected. 
 The Markov equivalence class then equals the set of all DAGs 
 that can be obtained from the CPDAG by orienting the undirected edges without creating new $v$-structures. 
 In Section~\ref{M-ssec:char_graphical} we derive an analogous graphical representation of $\DPX$.

Another useful (transformational) characterization result says that any two Markov equivalent DAGs can be transformed into each other by a sequence of distinct covered edge reversals~\citep{chickering95}. We will demonstrate in Section~\ref{M-ssec:char_trafo} that a very similar principle 
transfers to $\DPX$.


\subsubsection{Graphical representation of $\DPX$} \label{M-ssec:char_graphical}

The distribution equivalence class 
$\DPX$ can be graphically represented by a single partially directed acyclic graph (PDAG). A PDAG is a graph with directed and undirected edges that does not contain any directed cycles. A \emph{consistent DAG extension} of a PDAG is a DAG with the same skeleton, the same edge orientations on the directed subgraph of the PDAG, and no additional $v$-structures. 


\begin{defn} \label{M-def:PDAGrepr}
Let $\mathcal{E}$ be a set of Markov equivalent DAGs.
We denote by $G_{\mathcal{E}}$ the PDAG that has the same skeleton as the DAGs in $\mathcal{E}$ and $i\rightarrow j$ in $G_{\mathcal{E}}$ if and only if $i \rightarrow j$ in all the DAGs in $\mathcal{E}$, else, $i \text{ --- }
 j$. We say that $G_{\mathcal{E}}$ is maximally oriented with respect to $\mathcal{E}$.
\end{defn}

For a given distribution equivalence class $\DPX$, the corresponding 
PDAG $G_{\DPX}$ is uniquely defined by Definition~\ref{M-def:PDAGrepr}. 
Moreover, $G_{\DPX}$ is a graphical representation of $\DPX$ in the following sense:
\begin{thm} \label{M-thm:pdag}
$\DPX$ equals the set of 
all 
consistent DAG extensions of~$G_{\DPX}$.
\end{thm}

A proof can be found in Section~\ref{S-sec:prf-graphical} in the supplement. Theorem~\ref{M-thm:pdag} states that one can represent $\DPX$ with a single PDAG $G_{\DPX}$ without loss of information, as $\DPX$ can be reconstructed from $G_{\DPX}$ by listing all consistent DAG extensions. An example is given in Figure~\ref{M-fig:exampleGraphChar}.
Note that $\GDPX$ can be interpreted as a maximally oriented graph with respect to some background knowledge as defined in~\citep{meek1995}. For details, we refer to Section~\ref{M-ssec:estim_GDPX}.

\begin{figure}[h]
\centering
\begin{tikzpicture}[->,>=latex,shorten >=1pt,auto,node distance=1.0cm,
                    thick]
  \tikzstyle{every state}=[fill=black!30,draw=black,text=black, inner sep=0.4pt, minimum size=12pt]
  
  \node[state] (X1a) {$1$};
  \node[state] (X2a) [below right of=X1a] {$2$};
  \node[state] (X3a) [above right of=X2a] {$3$};
  \node[state] (X4a) [below of=X2a] {$4$};
  
  \path (X1a) edge (X2a) 
  		(X2a) edge (X4a);
  \path[-] 
  		(X1a) edge (X3a)
		(X2a) edge (X3a);

  \coordinate [label=above:$G_{\DPX}$] (L1) at (0.75,-2.75);	

  \node[state] (X1b) [right=1.5cm of X3a] {$1$};
  \node[state] (X2b) [below right of=X1b] {$2$};
  \node[state] (X3b) [above right of=X2b] {$3$};
  \node[state] (X4b) [below of=X2b] {$4$};
    
  \path (X1b) edge (X2b)
  		(X2b) edge [dashed] (X4b)
  		(X3b) edge [dashed] (X1b)
  		(X3b) edge [dashed] (X2b);

  \node[state] (X1c) [right of=X3b] {$1$};
  \node[state] (X2c) [below right of=X1c] {$2$};
  \node[state] (X3c) [above right of=X2c] {$3$};
  \node[state] (X4c) [below of=X2c] {$4$};
    
  \path (X1c) edge (X2c)
  		(X2c) edge [dashed] (X4c)
  		(X1c) edge [dashed] (X3c)
  		(X3c) edge [dashed] (X2c);

  \node[state] (X1d) [right of=X3c] {$1$};
  \node[state] (X2d) [below right of=X1d] {$2$};
  \node[state] (X3d) [above right of=X2d] {$3$};
  \node[state] (X4d) [below of=X2d] {$4$};
    
  \path (X1d) edge (X2d)
  		(X2d) edge [dashed] (X4d)
  		(X1d) edge (X3d)
  		(X2d) edge [dashed] (X3d);

  \coordinate [label=above:$\DPX$] (L1) at (6.5,-2.75);	
\end{tikzpicture}
\caption{Graphical representation of $\DPX$ with the single PDAG $G_{\DPX}$. $\DPX$ equals the set of all consistent DAG extensions of $G_{\DPX}$. The graph with $2 \rightarrow 3 \rightarrow 1$ is not a consistent DAG extension of $G_{\DPX}$ as it contains a cycle. Linear edges are dashed, nonlinear edges are solid.}
\label{M-fig:exampleGraphChar}
\end{figure}
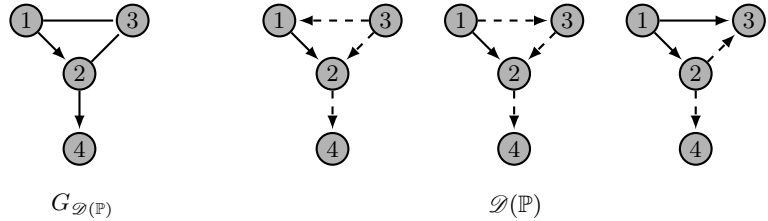

Conceptually, this is analogous to the use of CPDAGs to represent Markov equivalence classes. There are important differences, though: first of all, necessary and sufficient conditions have been derived for a graph to be a CPDAG of a Markov equivalence class~\citep[Theorem 4.1]{andersson1997}. These properties do not all transfer to $G_{\DPX}$. For example, $G_{\DPX}$ typically is not a chain graph, see Figure~\ref{M-fig:exampleGraphChar}. 
Secondly, given a DAG $D$, the CPDAG (and hence a full characterization of the Markov equivalence class) can be obtained by an iterative application of three purely graphical orientation rules (R1-R3 in Figure~\ref{M-fig:orientationRules}) applied to the pattern of~$D$~\citep{meek1995}. This is not true for $G_{\DPX}$ and $\DPX$. It is still feasible to obtain $G_{\DPX}$ from a DAG $D \in \DPX$, but it is crucial to know which of the functions in the (unique) corresponding PLSEM (cf. Lemma~\ref{S-le:uniquePLSEM} in the supplement) are linear and which are nonlinear. 
We will show in Section~\ref{M-sec:estim} that 
the transformational characterization in Theorem~\ref{M-thm:covered-reversals} gives rise to a
consistent and efficient score-based procedure to estimate $G_{\DPX}$ 
based on $D \in \DPX$ and samples of $\PX$.

\subsubsection{Transformational characterization of $\DPX$} \label{M-ssec:char_trafo}

Given $D \in \DPX$, the distribution equivalence class $\DPX$ can be comprehensively characterized via sequences of local transformations of DAGs. 

\begin{thm}\label{M-thm:covered-reversals}
  Assume that $\PX$ has been generated by a PLSEM and that it is faithful to the underlying DAG. Then, the following two results hold: 
\begin{enumerate}  
	\item[(a)]    Let $D \in \DPX$, $i \rightarrow j$ covered in $D$, and $D'$ be the DAG that differs from~$D$ only by the reversal of $i \rightarrow j$. Then, $D' \in \DPX$ if and only if $i \rightarrow j$ is linear in $D$. 
	Furthermore, if $i \rightarrow j$ is 
	covered and 
	nonlinear in $D$, then 
	$i \rightarrow j$ in all DAGs in $\DPX$.
  	\item[(b)] Let $D, D' \in \DPX$. Then there exists a sequence of distinct covered linear edge reversals that transforms $D$ to $D'$.
  \end{enumerate}
\end{thm}

A proof can be found in Section~\ref{S-sec:prf-covered} in the supplement and an 
illustration is provided in Figure~\ref{M-fig:exampleTrafoChar}. 
Note that the interesting part of this result is that $\DPX$ is connected with respect to covered linear edge reversals. It will be of particular importance in the design of score-based estimation procedures for $\DPX$ and $G_{\DPX}$ in Section~\ref{M-sec:estim}.

Theorem~\ref{M-thm:covered-reversals} covers the two special cases discussed in Section~\ref{M-ssec:relwork}:
if all the functions $f_{j,i}$ in equation~\eqref{M-eq:PLSEM} are linear, 
$\DPX$ (which, in this setting, is equal to the Markov equivalence class) can be fully characterized by sequences of covered edge reversals of $D$ (as all the edges are linear). If, on the contrary, all the functions $f_{j,i}$ in equation~\eqref{M-eq:PLSEM} are nonlinear, $\DPX$ only consists of the DAG~$D$ as there is no covered linear edge in $D$. 

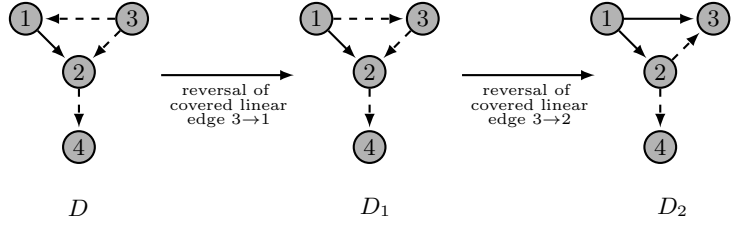
\begin{figure}[h]
\centering
\begin{tikzpicture}[->,>=latex,shorten >=1pt,auto,node distance=1.0cm,
                    thick]
  \tikzstyle{every state}=[fill=black!30,draw=black,text=black, inner sep=0.4pt, minimum size=12pt]
  
  \node[state] (X1a) {$1$};
  \node[state] (X2a) [below right of=X1a] {$2$};
  \node[state] (X3a) [above right of=X2a] {$3$};
  \node[state] (X4a) [below of=X2a] {$4$};
  
  \path (X1a) edge (X2a) 
  		(X2a) edge [dashed] (X4a) 
  		(X3a) edge [dashed] (X1a)
		(X3a) edge [dashed] (X2a);

  \coordinate [label=above:$D$] (L1) at (0.7,-2.75);	

  \node[state] (X1b) [right=2cm of X3a] {$1$};
  \node[state] (X2b) [below right of=X1b] {$2$};
  \node[state] (X3b) [above right of=X2b] {$3$};
  \node[state] (X4b) [below of=X2b] {$4$};
    
  \path (X1b) edge (X2b)
  		(X2b) edge [dashed] (X4b)
  		(X1b) edge [dashed] (X3b)
  		(X3b) edge [dashed] (X2b);

  \coordinate [label=above:$D_1$] (L1) at (4.65,-2.75);

  \node[state] (X1c) [right=2cm of X3b] {$1$};
  \node[state] (X2c) [below right of=X1c] {$2$};
  \node[state] (X3c) [above right of=X2c] {$3$};
  \node[state] (X4c) [below of=X2c] {$4$};
    
  \path (X1c) edge (X2c)
  		(X2c) edge [dashed] (X4c)
  		(X1c) edge (X3c)
  		(X2c) edge [dashed] (X3c);

  \coordinate [label=above:$D_2$] (L1) at (8.6,-2.75);	

  \draw (1.8,-0.75) -- node[below] {$\substack{\text{reversal of } \\ \text{covered linear} \\ \text{edge } 3 \rightarrow 1}$} ++(1.8,0);
  \draw (5.8,-0.75) -- node[midway,below] {$\substack{\text{reversal of } \\ \text{covered linear} \\ \text{edge } 3 \rightarrow 2}$} ++(1.8,0);

\end{tikzpicture}
\caption{Transformational characterization of $\DPX$ from Figure~\ref{M-fig:exampleGraphChar}. Let $1 \rightarrow 2$ in $D$ be nonlinear (solid) and all other edges in $D$ be linear (dashed). Then, $D_1$ and $D_2$ 
can be reached from $D$ by the displayed sequence of covered linear edge reversals. Note that in $D$ and $D_2$, $1 \rightarrow 2$ is covered but nonlinear and hence cannot be reversed by Theorem~\ref{M-thm:covered-reversals}~(a). Moreover, $2 \rightarrow 4$ is not covered in any of $D,D_1$ and $D_2$ and hence cannot be reversed.}
\label{M-fig:exampleTrafoChar}
\end{figure}


\subsection{General characterizations not assuming faithfulness} \label{M-ssec:char_unfaith}
In this section we give general characterizations of 
PLSEMs that generate the same distribution $\PX$, both, from the perspective of causal orderings and from a functional viewpoint. The functional characterization in Section~\ref{M-ssec:funcchar} describes how the $f_{j,i}$ of different PLSEMs relate to each other. The characterization via causal orderings in Section~\ref{M-ssec:orderchar} describes the set of causal orderings, such that there exists a corresponding PLSEM that generates the given distribution $\PX$.
 It will show that nonlinear functions impose a very specific structure on the model, 
 which (perhaps surprisingly) is compatible with some of the previous theory on graphical models, as described in Section~\ref{M-ssec:relwork}. Furthermore it will help us understand in the general case how nonlinear functions restrict the set of PLSEMs that generate $\PX$. Section~\ref{M-ssec:intuition} gives some intuition on the functional characterization in Section~\ref{M-ssec:funcchar}. 
 Throughout this section, we assume that $\PX$ is generated by a PLSEM as defined in equation~\eqref{M-eq:PLSEM}. 

\subsubsection{Functional characterization}\label{M-ssec:funcchar}

Let us first characterize the result on the level of SEMs. Consider a PLSEM that generates $\PX$,
\begin{align*} 
  X_{j} &= \mu_{j} + \sum\limits_{i \in \pa_D(j)} f_{j,i}(X_i) + \eps_j ,
\end{align*}
where $f_{j,i},  D, \eps_{j}, \mu_{j}, \sigma_{j}^{2} = \text{Var}(\eps_{j})$ satisfy the assumptions 
from Section~\ref{M-ssec:probdescAddGaussSEMandDEC}. 

Let us define the function $\F : \mathbb{R}^p \rightarrow \mathbb{R}^p$ by
\begin{equation}\label{M-eq:4}
	\F(x)_j := \frac{1}{\sigma_{j}} \left(x_j - \mu_{j} -\sum\limits_{i \in \pa_D(j)} f_{j,i}(x_i) \right).
\end{equation}
It turns out to be convenient to work with this function $\F$.
Notably, we do not lose any information by working with $\F$ instead of $f_{j,i}$, $\pa_{D}(j), \mu_{j}$ and~$\sigma_{j}$ as these quantities can 
be recovered from $\F$. Specifically, we can easily obtain the distribution of the errors from the function $\F$
as
\begin{align}\label{M-eq:23}
  \sigma_{j} :=  1/\partial_{j} \F_{j}.
\end{align}
By definition, 
$\F(X) \sim \mathcal{N}(0, \text{Id}_{p})$. Note that $\F$ maps the observed random variable $X \in \mathbb{R}^{p}$ to the scaled residuals $\frac{\eps_{j}}{\sigma_{j}}$. As for every $\eps \in \mathbb{R}^{p}$ there exists exactly one $X \in \mathbb{R}^{p}$ that satisfies equation~\eqref{M-eq:PLSEM}, $\F$ is invertible. Hence, if $Z \sim \mathcal{N}(0, \text{Id}_{p})$, it holds that $\F^{-1}(Z) \sim X$. Using this, we obtain $\mu_{j} = \mathbb{E}_{Z}[\F^{-1}(Z)_{j}]$ and we can recover the functions $f_{j,i}$ from the function $\F$ using the equations
\begin{align} \label{M-eq:23b}
  f_{j,i}' = -  \sigma_{j} \partial_{i} \F_{j} \qquad \mbox{and} \qquad \mathbb{E}_{Z} f_{j,i}( \F^{-1}(Z)_{i})= 0.
\end{align}
Note that the equation on the left hand side determines $f_{j,i}$ up to a constant, whereas the equation on the right hand side determines the constant using only quantities that can be calculated from $\F$. In the same spirit, $\pa_{D}(j)$ can be recovered from $\F$ via
\begin{align}\label{M-eq:29}
  \pa_{D}(j) = \{ i \neq j : \partial_{i} \F_{j} \not \equiv 0  \}.
\end{align}
In this sense, instead of describing the 
PLSEM by $f_{j,i} , \pa_{D}(j), \mu_j$ and  $ \sigma_{j} $ it can simply be described by the function $\F : \mathbb{R}^{p} \rightarrow \mathbb{R}^{p}$. Now let us define
\begin{align*}
  \mathcal{F}(\PX) := \left\{ \F : \mathbb{R}^{p} \mapsto \mathbb{R}^{p} : \F \mbox{ suffices \eqref{M-eq:4} for a PLSEM that generates $\PX$} \right\}.
\end{align*}
We call the functions in this set \emph{PLSEM-functions}. Let us define the set of orthonormal matrices $\mathcal{O}_{n}(\mathbb{R}) = \{ O \in \mathbb{R}^{ p \times p} : O O^{t} = \mbox{Id} \}$. 
The following theorem follows from Lemma~\ref{S-thm:func-char} in the supplement. See also Remark~\ref{S-rem:proof-func} in the supplement for details. It states that we can construct all PLSEMs that generate $\PX$ by essentially \textit{rotating} $\F$. 
\begin{thm}[Characterization of potential PLSEMs]\label{M-thm:funct-char}
  For a given $\F \in \mathcal{F}(\PX)$ there exists a set of (constant) rotations $\mathcal{O}_{\mathcal{F}(\PX)} \subset \mathcal{O}_{n}(\mathbb{R})$ such that
\begin{align*}
  \mathcal{F}(\PX) = \left\{ O \cdot \F : O \in \mathcal{O}_{\mathcal{F}(\PX)} \right\}.
\end{align*}
A description and explicit formulae for each $ O \in \mathcal{O}_{\mathcal{F}(\PX)}$ can be found in Remark~\ref{S-rem:proof-func} in the supplement.
\end{thm}
 Astonishingly, in this sense, all PLSEMs that generate $\PX$ are rotations of each other. The importance of this result lies in its simplicity: There are very simple linear relationships between the $f_{j,i}$ in one PLSEM and the $\tilde f_{j,i}$ in another PLSEM. The formulae in Section~\ref{S-sec:prf-functional} in the supplement permit to fully characterize these matrices $\mathcal{O}_{\mathcal{F}(\PX)}$. In fact, the characterization in Lemma~\ref{S-thm:func-char} in the supplement is the first step towards all other characterizations.

\subsubsection{Characterization via causal orderings}\label{M-ssec:orderchar}

This section discusses a characterization of all potential causal orderings of a given PLSEM. 
Let us define the set of \textit{potential causal orderings} as
\begin{align*}
\begin{split}
	\mathcal{S}(\PX) \! := \! \left\{ \begin{array}{l} \hspace{-0.15cm}  \sigma \mbox{ permutation on } \{1,...,p\} \!\! : \mbox{there is a PLSEM  with DAG $D$} \\ \hspace{-0.15cm} \text{that generates $\PX$ such that } \sigma(i) < \sigma(j) \mbox{ for all } i \rightarrow j \text{ in } D \end{array} \hspace{-0.15cm} \right\}. 
\end{split}
\end{align*}
 Without assuming faithfulness, if all $f_{j,i}$ are linear, all permutations of $\{1,...,p\}$ are a causal ordering of a DAG corresponding to a PLSEM that generates $\PX$. That is, $\mathcal{S}(\PX)$ is equal to the set of all permutations of $\{1,\ldots,p\}$. Roughly, the more nonlinear functions in the PLSEM, the smaller the resulting set $\mathcal{S}(\PX)$. The interesting point is that nonlinear edges restrict $\mathcal{S}(\PX)$ in a very specific way. Before we state the theorem, consider a PLSEM that generates $\PX$, define the function $\F : \mathbb{R}^{p} \rightarrow \mathbb{R}^{p}$ as in equation~\eqref{M-eq:4} and define the set
\begin{align}\label{M-eq:defV}
  \mathcal{V} := \{ (i,j) \in \{1,...,p \}^{2} :  e_{j}^{t} (\D \F)^{-1} \partial_{i}^{2} \F \not \equiv 0 \},
\end{align}
where $e_{j}$, $j=1,\ldots,p$ is the standard basis of $\mathbb{R}^{p}$, $t$ stands for the transpose and $\mathrm{D} \F$ denotes the Jacobian of $\F$. We will discuss the interpretation of the set $\mathcal{V}$ and the expression $e_{j}^{t} (\D \F)^{-1} \partial_{i}^{2} \F$
in more detail later. 
For now, the potential causal orderings can be characterized as follows:
\begin{thm}[Characterization of potential causal orderings]\label{M-thm:char-via-order}
\begin{align*}
  \mathcal{S}(\PX) = \left\{ \sigma \mbox{ permutation on } \{ 1,\ldots,p\} :  \sigma(i) < \sigma(j) \mbox{ for all } (i,j) \in \mathcal{V} \right\} .
\end{align*}
\end{thm}
The proof of this theorem can be found in Section~\ref{S-sec:prf-ordering} in the supplement. In words, all permutations of the indices that do not swap any of the tuples in~$\mathcal{V}$ are a causal ordering of a DAG corresponding to a PLSEM that generates~$\PX$. And for all permutations of indices for which one of the tuples in $\mathcal{V}$ is switched, there exists \emph{no} PLSEM with this causal ordering that generates~$\PX$. Moreover, by Lemma~\ref{S-thm:nonlin-desc}~(b) in the supplement, if $(i,j) \in \mathcal{V}$, then $j$ is a descendant of $i$ in every PLSEM that generates $\PX$. 

Now, let us give some intuition on 
the set $\mathcal{V}$. 
For $e_{j}^{t} (\D \F)^{-1} \partial_{i}^{2} \F$ to be non-zero it is necessary that there is a directed path from node $i$ to node $j$ that begins with a nonlinear edge. However, the existence of such a path is not sufficient, due to potential cancellations. 
An example is given in Figure~\ref{M-fig:nonlin-cancel} where the causal ordering of nodes $1$ and $3$ is not fixed even though $\partial_1^2 F_3 \not\equiv 0$. In particular, the requirement that the direct effect of $i$ on $j$ (the function $f_{j,i}$ in the PLSEM) is nonlinear, that is, the requirement that $\partial_i^2 \F_j \not\equiv 0$, is \emph{not} sufficient to fix the causal ordering between $i$ and $j$.  Also, it is not sufficient to require that the total effect of variable $i$ on variable $j$ is nonlinear. This is shown in part (a) of the following example.  

\begin{example}
	Consider the DAG $1 \rightarrow 2 \rightarrow 3$ and $\PX$ that has been generated by a PLSEM of the form $X_1 = \eps_1, X_2 = f_{2,1}(X_1) + \eps_2, X_3 = f_{3,2}(X_2) + \eps_3$ with $\eps \sim \mathcal{N}(0, \mathrm{Id}_3)$.
\begin{enumerate}
\item[(a)]  Let $f_{2,1}(x) = 0.5 x$ be linear, $f_{3,2}(x) = x^{3}$ be nonlinear. The corresponding PLSEM-function is $\F(x) = (x_1, x_2 - 0.5 x_1, x_3 - x_{2}^{3})^t$. Using elementary calculations it can be seen that $e_{j}^{t} (\D \F)^{-1} \partial_{i}^{2} \F \not \equiv 0$ only for $(i,j)=(2,3)$. Hence, $\mathcal{V} = \{(2,3)\}$ and all permutations $\sigma$ respecting $\sigma(2) < \sigma(3)$ are a causal ordering of a DAG corresponding to a PLSEM that generates $\PX$. For example, for the causal ordering $\sigma(2) < \sigma(3) < \sigma(1)$, there exists a (unique) PLSEM with DAG $1 \leftarrow 2 \rightarrow 3$  that generates $\PX$. In particular, the causal ordering of variables $1$ and $3$ is not fixed even though there is a nonlinear total effect of variable $1$ on variable $3$.
\item[(b)] Let $f_{2,1}(x) = x^{3}$ be nonlinear, $f_{3,2}(x) = 0.5 x$ be linear.  The corresponding PLSEM-function is $\F(x) = (x_1, x_2 - x_1^{3}, x_3 - 0.5 x_{2})^t$. We obtain $\mathcal{V} = \{(1,2), (1,3)\}$ and all permutations $\sigma$ with $\sigma(1) < \sigma(2)$ and $\sigma(1) < \sigma(3)$ are a causal ordering of a DAG corresponding to a PLSEM that generates $\PX$. In particular, for $\sigma(1) < \sigma(3) < \sigma(2)$ we obtain that the PLSEM corresponding to the (unfaithful) DAG $1 \rightarrow 3 \rightarrow 2$ with $1 \rightarrow 2$ generates~$\PX$. 
\end{enumerate}
\end{example}
Let us make several concluding remarks: in (a), the causal ordering between  nodes $1$ and $3$ is not fixed, whereas in (b), 
it is fixed. Hence, the set $\mathcal{V}$ sometimes also fixes the causal ordering between two nodes that are not adjacent in the DAG corresponding to $\F$. 
Secondly, in both examples, the causal ordering of nodes incident to nonlinear edges is fixed. This raises the question whether it is true in general that nonlinear edges cannot be reversed. The answer is no (see Figure~\ref{M-fig:nonlin-cancel}), but in some sense, the models with ``reversible nonlinear edges'' are rather particular. Finally, if we make additional mild assumptions, stronger statements can be made about the index tuples in $\mathcal{V}$. 
We will discuss these issues further in Section~\ref{M-ssec:interplay}.

\subsubsection{Intuition on the functional characterization}\label{M-ssec:intuition}


This section motivates Theorem~\ref{M-thm:funct-char}. Consider two functions $\F, \G \in\mathcal{F}(\PX)$ that correspond to two different PLSEMs. 
By Proposition~\ref{S-prop:PLSEM2} in the supplement,  
\begin{align}\label{M-eq:8}
  \F(X)\sim \mathcal{N}(0,\mbox{Id}) \text{   and   } \G(X) \sim \mathcal{N}(0,\mbox{Id}).
\end{align}
Moreover, it follows from the definition of PLSEMs that $\F$ is invertible.
Let $Z \sim  \mathcal{N}(0,\mbox{Id}_{p})$. 
Using equation~\eqref{M-eq:8} twice, 
$$\F^{-1}(Z) \sim X \text{   and   }  \G (\F^{-1}(Z)) \sim \mathcal{N}(0, \mbox{Id}). $$ 
Hence the function $J: \mathbb{R}^{p} \rightarrow \mathbb{R}^{p}$, $J:=  \G (\F^{-1})$ suffices $J(Z) \sim Z \sim \mathcal{N}(0,\mbox{Id})$. Furthermore, it can be shown that $|\det \D J| =1 $. Then, using the transformation formula for probability densities, we obtain
\begin{align*}
  \frac{1}{(2 \pi )^{p/2}} \exp\left(- \frac{\|J(x)\|_{2}^{2}}{2}\right) = \frac{1}{(2 \pi )^{p/2}} \exp\left(- \frac{\| x \|_{2}^{2}}{2}\right) \mbox{ for all } x \in \mathbb{R}^{p}.
\end{align*}
By rearranging,
\begin{align*}
  \| J(x) \|_{2} =  \| x \|_{2} \mbox{ for all } x \in \mathbb{R}^{p}.
\end{align*}
If we admit that $J$ must be a linear function (which requires some work), this formula gives us $J \in \mathcal{O}_{n}(\mathbb{R}) := \{ O \in \mathbb{R}^{p \times p} : O O^{t} = \mbox{Id}\}$ and it immediately follows that $ \G = J \F$. This reasoning shows that the main work in proving Theorem~\ref{M-thm:funct-char} lies in showing that $J$ is a linear function.

\subsection{Understanding the interplay of nonlinearity and faithfulness}\label{M-ssec:interplay}

As indicated in Section~\ref{M-ssec:orderchar}, without further assumptions, some nonlinear edges can be reversed. An example is given in Figure~\ref{M-fig:nonlin-cancel}. There, the edge $1 \rightarrow 3$ can be reversed even though $f_{3,1}$ is a nonlinear function in the PLSEM corresponding to~$D_1$. 
The issue here arises because the nonlinear effect from $X_{1}$ to $X_{3}$ in $D_1$ cancels out over two paths. 
If we write $X_{3}$ as a function of $\eps_{1},\eps_{2},\eps_{3}$, that function is linear. 
The setting of $D_1$ in Figure~\ref{M-fig:nonlin-cancel} is rather particular as $ \partial_{1}^{2} f_{2,1}$ and $ \partial_{1}^{2} f_{3,1}$ are linearly dependent. As the function space $\mathcal{C}^{2}(\mathbb{R})$ is infinite dimensional, this is arguably a degenerate scenario. Note that faithfulness does not save us from this cancellation effect as 
$\PX$ is faithful to both, $D_{1}$ and $D_{2}$. 

\begin{figure}[h] 
\begin{tikzpicture}[->,>=latex,shorten >=1pt,auto,node distance=1.4cm,
                    thick]
  \tikzstyle{every state}=[fill=black!30,draw=black,text=black, inner sep=0.4pt, minimum size=12pt]
  
  \node[state] (X1a) {$1$};
  \node[state] (X2a) [below left of=X1a] {$2$};
  \node[state] (X3a) [below right of=X1a] {$3$};

  \path (X1a) edge (X3a)
  		(X1a) edge  (X2a)
  		(X2a) edge [dashed]  (X3a);

  \node[state] (X2b) [right=3.3cm of X3a] {$2$};
  \node[state] (X1b) [above right of=X2b] {$1$};
  \node[state] (X3b) [below right of=X1b] {$3$};

  \path (X1b) edge (X2b)
  		(X3b) edge [dashed] (X1b)
  		(X3b) edge [dashed] (X2b);
   
        \coordinate [label=above:$D_1$] (C) at (-1.3,0.0);	
      \coordinate [label=above:$D_2$] (C) at (4.4,0);	
      
\end{tikzpicture}
\caption{Nonlinear edges can be reversed if nonlinear effects cancel out. $X_{1} = \eps_{1}$, $X_{2} = X_{1}^{2}+ X_{1} + \eps_{2}$, $X_{3} =  X_{2} - X_{1}^{2} + \eps_{3}$ with $\eps \sim \mathcal{N}(0,\mathrm{Id}_{3})$ generates the same joint distribution of $(X_1,X_2,X_3)$ as $X_{3} = \tilde \eps_{3}$, $X_{1} = X_{3}/3 + \tilde \eps_{1}$, $X_{2} = X_{1}/2+X_{1}^{2}+X_{3}/2+ \tilde \eps_{2}$ 
with $\tilde{\eps}_3 \sim \mathcal{N}(0,3), \tilde{\eps}_1 \sim \mathcal{N}(0,2/3), \tilde{\eps}_2 \sim \mathcal{N}(0,1/2)$ independent.
This stems from the fact that the nonlinear parts of the functions $f_{2,1}(x) $ and $f_{3,1}(x) $ cancel out, i.e. $f_{2,1}''+f_{3,1}'' = 0$. 
Note that this example does not contradict the previous theoretical results. It holds that $ e_{3}^{t} (\D \F)^{-1} \partial_{1}^{2} \F \equiv 0$ for the PLSEM-function $\F$ corresponding to $D_1$. 
Hence the causal ordering of $D_2$ does not contradict Theorem~\ref{M-thm:char-via-order}.}
\label{M-fig:nonlin-cancel}
\end{figure}
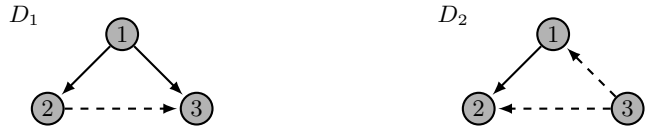

Nevertheless, we can rely on a different, rather weak assumption: 
consider a node $i$ in a DAG $D$ and assume that the corresponding functions in the set 
$$ \{ \partial_{i}^{2} f_{j',i} :  j' \text{ is a child of } i \text{ in $D$ and } f_{j',i} \text{ is nonlinear} \}$$ are linearly independent. 
In other words: assume that the ``nonlinear effects'' from $X_{i}$ on its children are linearly independent functions. Then these nonlinear edges cannot be reversed.

The following theorem is a direct implication of Lemma~\ref{S-thm:nonlin-desc} (a) and (b) in the supplement.
\begin{thm}\label{M-sec:nonl-faithf}
Consider a PLSEM and the corresponding distribution~$\PX$. Let $j$ be a child of $i$ in $D$ and let $f_{j,i}$ be a nonlinear function. If the functions in the set $ \{ \partial_{i}^{2} f_{j',i} : j' \text{ is a child of } i \text{ in $D$ and } f_{j',i} \text{ is nonlinear}  \}$ are linearly independent, then $j$ is a descendant of $i$ in any other DAG $D'$ of a PLSEM that generates $\PX$.
\end{thm}

Intuitively, this should not be the end of the story: if an edge $i \rightarrow j$ is nonlinear, then usually there should also be a nonlinear relationship between~$i$ and the descendants of $j$. Hence it should be possible to infer some statements about the causal ordering of $i$ and the descendants of $j$. In general, this is not true as demonstrated in Figure~\ref{M-fig:unfaith-desc}. 
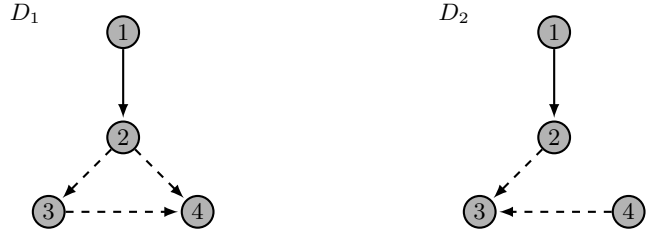
\begin{figure}[h]
\begin{tikzpicture}[->,>=latex,shorten >=1pt,auto,node distance=1.4cm,
                    thick]
  \tikzstyle{every state}=[fill=black!30,draw=black,text=black, inner sep=0.4pt, minimum size=12pt]
  
  \node[state] (X1a) {$1$};
  \node[state] (X2a) [below of=X1a] {$2$};
  \node[state] (X3a) [below left of=X2a] {$3$};
  \node[state] (X4a) [below right of=X2a] {$4$};

  \path 
  		(X1a) edge  (X2a)
  		(X2a) edge  [dashed] (X3a)
  		(X3a) edge  [dashed] (X4a)
  		(X2a) edge  [dashed] (X4a);

  \node[state] (X3b) [right=3.3cm of X4a] {$3$};
  \node[state] (X2b) [above right of=X3b] {$2$};
  \node[state] (X4b) [below right of=X2b] {$4$};
  \node[state] (X1b) [above of=X2b] {$1$};

  \path 
  		(X1b) edge  (X2b)
  		(X2b) edge [dashed] (X3b)
  		(X4b) edge [dashed] (X3b);

  
      \coordinate [label=above:$D_1$] (C) at (-1.3,0.0);	
      \coordinate [label=above:$D_2$] (C) at (4.4,0);	  
  
\end{tikzpicture}
\caption{If $\PX$ is not faithful to $D$, descendants are not fixed. Node~$4$ is a descendant of node~$1$ in $D_{1}$ but not in $D_{2}$. On the left hand side, $X_{1} = \eps_{1}$, $X_{2} = X_{1}^{2}+ \eps_{2}$, $X_{3} = X_{2} + \eps_{3}$, $X_{4} =  X_{3} - X_{2} + \eps_{4}$, with $\eps \sim \mathcal{N}(0,\mathrm{Id}_{4})$. On the right hand side, $X_{1} = \tilde \eps_{1}$, $X_{2} = X_{1}^{2}+ \tilde \eps_{2}$, $X_{3} = X_{2} + 1/2 \cdot X_{4} + \tilde \eps_{3}$, $X_{4} = \tilde \eps_{4}$, where $\tilde \eps_{1} \sim \mathcal{N}(0,1), \tilde \eps_{2} \sim \mathcal{N}(0,1), \tilde \eps_{3} \sim \mathcal{N}(0,1/2)$ and $\tilde \eps_{4} \sim \mathcal{N}(0,2)$. 
Both PLSEMs generate the same distribution. 
Note that in this case, additional assumptions on the nonlinear function $f_{2,1}$ would not resolve the issue. 
}
\label{M-fig:unfaith-desc}
\end{figure}

 Under the assumption of faithfulness, additional statements can be made about descendants of $j$. In some sense the nonlinear effect from $i$ on the descendants of $j$, mediated through some of the descendants of $j$, cannot ``cancel out''. Hence, all descendants of $j$ are fixed. The following theorem is a direct implication of Lemma~\ref{S-thm:nonlin-desc} (c) and (d) in the supplement.

\begin{thm}\label{M-sec:nonl-faithf-2}
Let the assumptions of Theorem~\ref{M-sec:nonl-faithf} be true. In addition, let $\PX$ be faithful to the DAG $D$. Fix $k  \neq i$. Then $k$ is a descendant of $i$ in each DAG $D'$ of a PLSEM that generates $\PX$ if and only if $k$ is a descendant of a nonlinear child of $i$ in $D$.
\end{thm}
Note that we use the convention that a node is a descendant of itself. Theorem~\ref{M-sec:nonl-faithf-2} guarantees that certain descendants of $i$ are descendants of~$i$ in all DAGs~$D'$ of PLSEMs that generate $\PX$. In that sense, it provides a simple criterion that tells us whether or not $k$ is descendant of $i$ in all of these DAGs. It is crucial to be precise: we do not assume that $\PX$ is faithful to $D'$, that means, we search over all PLSEMs that generate $\PX$. 
If we search over the smaller space $\DPX$, that is, additionally assume that $\PX$ is faithful to~$D'$, the set of potential PLSEMs usually gets smaller.
In many cases, there are some edges that are not fixed if we search over all PLSEMs, but fixed if we only search over PLSEMs with DAGs in $\DPX$.

As discussed in Section~\ref{M-ssec:char_graphical}, 
$\DPX$ can be represented by a single PDAG $\GDPX$. In the following, we will 
discuss 
the estimation of $\DPX$ and $\GDPX$.





\section{Score-based estimation of $\DPX$ and $G_{\DPX}$} \label{M-sec:estim}
Consider $\PX$ that has been generated by a PLSEM and assume that $\PX$ is faithful to the underlying DAG. We denote by $\{X^{(i)}\}_{i=1,...,n}$ i.i.d. copies of $X \in \mathbb{R}^p$ and by $\PXn$ their empirical distribution.
The goal of this section is to derive a consistent score-based estimation procedure for the distribution equivalence class $\DPX$ based on $\PXn$ and one (true) DAG $D^0 \in \DPX$. We first describe a ``naive'' recursive solution that lists all members of $\DPX$ and motivate the score-based approach in Section~\ref{M-ssec:estim_recursive}. We then present a more efficient procedure that directly estimates the graphical representation $G_{\DPX}$ as defined in Section~\ref{M-ssec:estim_GDPX}. Both methods rely on the transformational characterization result in Theorem~\ref{M-thm:covered-reversals}. 

In practice, we may replace the true $D^0$ by an estimate, e.g., from the CAM methodology~\citep{pbjoja13}. If the estimate is consistent for a DAG in $\DPX$ we obtain consistency of our method for the entire distribution equivalence class~$\DPX$. 

\subsection{Estimation of $\DPX$} \label{M-ssec:estim_recursive}

Theorem~\ref{M-thm:covered-reversals} provides a straightforward way to list all members of $\DPX$. Starting from the DAG $D^0$, one can search over all sequences of distinct covered linear edges reversals. By Theorem~\ref{M-thm:covered-reversals} (a), all DAGs that are traversed are in $\DPX$ and by Theorem~\ref{M-thm:covered-reversals} (b), $\DPX$ is connected with respect to sequences of distinct covered linear edge reversals. Moreover, by Theorem~\ref{M-thm:covered-reversals}~(a), an edge that is nonlinear and covered in a DAG in $\DPX$ has the same orientation in all the members of $\DPX$. These simple observations immediately lead to a recursive estimation procedure. Its population version is described in Algorithm~\ref{M-alg:listAllDAGsPLSEM}. The inputs are $D^0$ (with all its edges marked as ``unfixed'') and an oracle that answers the question if a specific edge in a DAG in $\DPX$ is linear or nonlinear.

\begin{algorithm}[!htb]
\begin{algorithmic}[1]
\IF{there is no covered edge in DAG $D^0$ that is marked as unfixed}
\STATE Add $D^0$ to the distribution equivalence class $\DPX$ and terminate.
\ENDIF
\STATE Choose a covered edge $i \rightarrow j$ in DAG $D^0$ that is marked as unfixed. 
\IF{the edge $i \rightarrow j$ is linear in $D^0$}
\STATE Define a DAG $D^0_1 := D^0$ with edge $i \rightarrow j$ in $D^0_1$ marked as fixed and a DAG $D^0_2$ equal to $D^0$ except for a reversed edge $i \leftarrow j$ marked as fixed in $D_2^0$.
\STATE Call the function \texttt{listAllDAGsPLSEM} recursively for both DAGs $D^0_1$ and $D^0_2$.
\ELSE
\STATE Mark the edge $i \rightarrow j$ in $D^0$ as fixed and call \texttt{listAllDAGsPLSEM} for DAG $D^0$.
\ENDIF
\end{algorithmic}
\caption{listAllDAGsPLSEM (population version)}\label{M-alg:listAllDAGsPLSEM}
\end{algorithm}

Unfortunately, the (true) information whether a selected covered edge $i \rightarrow j$ in a DAG $D \in \DPX$ is linear or not is generally not available. Also, it cannot simply be deduced from the starting DAG $D^0$ as the status of the edge may have changed in $D$. For an example, see Figure~\ref{M-fig:edge-status-change}: edge $1 \rightarrow 3$ is not covered and linear in $D_1$ but nonlinear and covered (and hence irreversible) in $D_2 \in \DPX$. 

To check the status of a covered edge in a given DAG $D \in \DPX$, one could either test (non-)linearity of the functional component in the (unique) PLSEM corresponding to $D$ or rely on a score-based approach. In the following we are going to elaborate on the latter.
We closely follow the approach presented in~\citep{pbjoja13}. 

We assume that the functions $f_{j,i}$ in equation~\eqref{M-eq:PLSEM} are from a class of smooth functions $\mathcal{F}_i \subseteq \{f \in C^2(\mathbb{R}), \EE[f(X_i)]=0 \}$, 
which is closed with respect to the $L_2(\mathbb{P}_{X_i})$-norm and closed under linear transformations. For a set of given basis functions, we denote by $\mathcal{F}_{n,i} \subseteq \mathcal{F}_i$ the finite-dimensional approximation space which typically increases as $n$ increases. The spaces of additive functions with components in $\mathcal{F}_i$ and $\mathcal{F}_{n,i}$, respectively, are closed assuming an analogue of a minimal eigenvalue condition. All details are given in~\citep{pbjoja13}. Without loss of generality, we assume $\mu_j = 0$ as in the original paper. For $D^0 \in \DPX$, let $\theta^{D^0} \!\! := \! (\{f^{D^0}_{j,i}\}_{j=1,...,p, i\in \pa_{D^0}(j)}, \{\sigma^{D^0}_j\}_{j=1,...,p})$ be the infinite-dimensional parameter of the corresponding PLSEM. The expected negative log-likelihood reads
\begin{equation*} 
	\EE\left[- \log p_{\theta^{D^0}}(X) \right] = \sum\limits_{j=1}^p \log(\sigma^{D^0}_j) + \text{C}, \qquad C = \frac{p}{2} \log(2\pi) + \frac{p}{2} .
\end{equation*}



All $D^0 \in \DPXO$ 
lead to the minimal expected negative log-likelihood, as 
by definition, the corresponding PLSEM generates the true distribution~$\PXO$. For a misspecified model with wrong 
DAG $D \not\in \DPXO$ 
we obtain the projected parameter $\theta^{D} = \left( \{ f^{D}_{j,i} \}_{j=1,...,p, i\in \pa_{D}(j)}, \{ \sigma^{D}_j \}_{j=1,...,p} \right)$ as
\begin{align*}
	\{f^{D}_{j,i}\}_{i \in \pa_D(j)} &= \argminn\limits_{g_{j,i} \in \Fi} \EE [ (X_j - \sum\limits_{i \in \pa_D(j)} g_{j,i}(X_i) )^2 ] \\
	(\sigma^{D}_j)^2 &= \EE [ (X_j - \sum\limits_{i \in \pa_D(j)} f^{D}_{j,i}(X_i) )^2 ] 
\end{align*}
with expected negative log-likelihood
$$
	\EE \left[ -\log\left(p^{D}_{\theta^{D}}(X)\right) \right] = \sum\limits_{j=1}^p \log(\sigma_j^{D}) + C, \qquad C = \frac{p}{2} \log(2\pi) + \frac{p}{2} ,
$$
where all expectations are taken with respect to the true distribution $\PX$. We refer to $\EE \left[ -\log\left(p^{D}_{\theta^{D}}(X)\right) \right]$ as the \emph{score of $D$} and to $\log(\sigma_j^D)$ as \emph{score} of node~$j$ in $D$.
For a DAG $D^0 \in \DPXO$, let 
$$
	\CDO = \left\{D \mid D \text{ and } D^0 \text{ differ by one covered nonlinear edge reversal}\right\}.
$$
Then, for $D^0 \in \DPXO$ and $D \in \CDO$
that (without loss of generality) only differ by the orientation of the covered edge between the nodes $i$ and $j$, the difference in expected negative log-likelihood is given as
\begin{align} \label{M-eq:diffScoreCovRev}
\begin{split}
	& \EE\left[ -\log\left(p^{D}_{\theta^{D}}(X)\right) \right] - \EE \left[ -\log\left(p_{\theta^{D^0}}(X)\right) \right] \\
	&= \log(\sigma^{D}_{i}) + \log(\sigma^{D}_{j})  - \log(\sigma^{D^0}_{i}) - \log(\sigma^{D^0}_{j}).
\end{split}
\end{align}
Since the score is decomposable over the nodes, the reversal of a covered edge only affects the scores locally at the two nodes $i$ and $j$ incident to the covered edge.
We denote by
\begin{align} \label{M-eq:xip}
\begin{split}
	\xi_p : &= \min\limits_{\substack{D^0 \in \DPXO \\ D \in \CDO}} \left( \EE\left[-\log\left(p^{D}_{\theta^{D}}(X)\right) \right] - \EE \left[-\log\left(p_{\theta^{D^0}}(X)\right) \right] \right) \\ 
\end{split}
\end{align}
the \emph{degree of separation} of true models in $\DPXO$ and misspecified models in $\CDPX$
that can be reached by the reversal of one covered nonlinear edge in any DAG $D^0 \in \DPXO$. From the transformational characterization in Theorem~\ref{M-thm:covered-reversals} it follows that $\xi_p > 0$. 
%
Combining equations~\eqref{M-eq:diffScoreCovRev} and~\eqref{M-eq:xip} motivates the estimation procedure in Algorithm~\ref{M-alg:listAllDAGsPLSEMscorebased} that takes as inputs $n$ samples
$X^{(1)},...,X^{(n)}$ and a DAG $D^0 \in \DPX$ (with all its edges marked as ``unfixed'') and
outputs a score-based estimate $\DPXhat$ of $\DPX$. To make the algorithm more robust with respect to misspecifications of the noise distributions (cf. Section~\ref{M-sec:misspec}) we only perform one-sided tests in line $8$ of Algorithm~\ref{M-alg:listAllDAGsPLSEMscorebased}.

\begin{algorithm}[!htb]
\begin{algorithmic}[1]
\IF{there is no covered edge in DAG $D^0$ that is marked as unfixed}
\STATE Add $D^0$ to $\DPXhat$ and terminate.
\ENDIF
\STATE Choose a covered edge $i \rightarrow j$ in DAG $D^0$ that is marked as unfixed. Denote by $D'$ the DAG that equals $D^0$ except for a reversed edge $i \leftarrow j$.
\STATE Additively regress $X_i$ on $X_{\pa_{D^0}(i)}$, $X_j$ on $X_{\pa_{D^0}(j)}$, $X_i$ on $X_{\pa_{D^0}(i) \cup \{j\}}$, $X_j$ on $X_{\pa_{D^0}(i)}$ 
\vspace{-0.1cm}
\STATE Compute the standard deviations of the residuals to obtain $\hat{\sigma}^{D^0}_i, \hat{\sigma}^{D^0}_j, \hat{\sigma}^{D'}_i$ and $\hat{\sigma}^{D'}_j$.
\STATE Compute the score difference $\Delta := \log(\hat{\sigma}^{D'}_i) + \log(\hat{\sigma}^{D'}_j)  - \log(\hat{\sigma}^{D^0}_i) - \log(\hat{\sigma}^{D^0}_j)$
\IF{$\Delta < \alpha$}
\STATE Set $D^0_1 := D^0$ with $i \rightarrow j$
 marked as fixed, $D^0_2 := D'$ with $i \leftarrow j$ marked as fixed, $\alpha_1 := \alpha$ and $\alpha_2 := \alpha - \Delta$.
\STATE Call the function \texttt{listAllDAGsPLSEM} recursively for both, DAG $D^0_1$ with parameter $\alpha = \alpha_1$ and DAG $D^0_2$ with parameter $\alpha = \alpha_2$.
\ELSE
\STATE Mark the edge $i \rightarrow j$ in $D^0$ as fixed and call \texttt{listAllDAGsPLSEM} for DAG $D^0$ with parameter $\alpha = \alpha_1$.
\ENDIF
\end{algorithmic}
\caption{listAllDAGsPLSEM}\label{M-alg:listAllDAGsPLSEMscorebased}
\end{algorithm}

To prove the (high-dimensional) consistency of the score-based estimation procedure, we make the following assumptions. For a function $h: \mathbb{R} \rightarrow \mathbb{R}$, we write $P(h) = \EE[h(X)]$ and $P_n(h) = \frac{1}{n} \sum\limits_{i=1}^{n} h(X^{(i)})$.
\begin{assu} \label{M-assu:consistency}
\hfill
\begin{enumerate}
\item[(i)]	Uniform upper bound on node degrees:
			$$ \max_{\substack{D \in \DPXO \cup \CDPX \\ j = 1,...,p}} \deg_{D}(j) \leq M \text{ for some positive constant } M < \infty.$$
\item[(ii)] 
			Uniform lower bound on error variances: 
			$$\min\limits_{\substack{D \in \DPXO \cup \CDPX \\ j=1,...,p }} (\sigma^{D}_j)^2 \geq L > 0. $$
\item[(iii)] Empirical process bound:
			$$\max\limits_{\substack{D \in \DPXO \cup \CDPX \\ j=1,...,p }} \Delta^D_{n,j} = o_P(1),$$
			where $\Delta^D_{n,j}  = \sup\limits_{g_{j,i} \in \Fi} |(P_n - P)((X_j - \sum\limits_{i \in \pa_D(j)} g_{j,i}(X_i))^2)|$. \item[(iv)] Control of approximation error:
			 $$\max\limits_{\substack{D^0 \in \DPXO \\ j=1,...,p}} |\gamma_{n,j}^{D^0}| = o(1),$$
			 where 
			 $$
				 \gamma_{n,j}^{D^0} = \EE[(X_j- \sum\limits_{i \in \pa_{D^0}(j)} f_{n;j,i}^{D^0}(X_i))^2] - \EE[(X_j- \sum\limits_{i 	
				 \in \pa_{D^0}(j)} f_{j,i}^{D^0}(X_i))^2]
			 $$
			 with
			 $$f^{D^0}_{n; j,i} = \argminn\limits_{g_{j,i} \in \mathcal{F}_{n,i}} \EE [(X_j - \sum\limits_{i \in \pa_{D^0}(j)} 
			 g_{j,i}(X_i) )^2 ]$$
			 and $\mathcal{F}_{n,i}$ are the approximation spaces as introduced before. 
\end{enumerate}
\end{assu}

Assumption~\ref{M-assu:consistency} (i) is satisfied if $D^0$ has bounded node degrees, as all DAGs under consideration are restricted to the same skeleton and hence all have equal node degrees. In the low-dimensional setting, Assumption~\ref{M-assu:consistency} (iii) is justified by~\citep[Lemma~5]{pbjoja13} under the assumptions mentioned there. 
These assumptions entail smoothness conditions on the functions in $\mathcal{F}_i$ and tail and moment conditions on $X$. In the high-dimensional setting, it follows from~\citep[Lemma~6]{pbjoja13} and $\sqrt{\log(p)/n} = o(1)$ together with Assumption~\ref{M-assu:consistency} (i) and the assumptions mentioned in the original paper.
Assumption~\ref{M-assu:consistency} (iv) can be ensured by requiring a smoothness condition on the coefficients of the basis expansion for the true functions~\citep[Section 4.2]{pbjoja13}. A proof of Theorem~\ref{M-thm:consistencyRecursiveAlgorithm} can be found in Section~\ref{S-ssec:proofConsistencyRecursiveAlgorithm} in the supplement.

\begin{thm} \label{M-thm:consistencyRecursiveAlgorithm}
	Under Assumption~\ref{M-assu:consistency} and $\xi_p \geq \xi_0 > 0$, for any constant $\alpha \in (0, \xi_0)$,
	$$
		\mathbb{P}[ \DPXhat = \DPXO ] \rightarrow 1 \qquad (n \rightarrow \infty). 
	$$
In case of a high-dimensional setting, for which the uniformity in Assumption~\ref{M-assu:consistency} is required, the convergence should be understood as poth $p \rightarrow \infty$ and $n \rightarrow \infty$.
\end{thm}

\begin{rem} \label{M-rem:gapcondition}
The assumption on the gap between log-likelihoods of true and wrong models in~\citep{pbjoja13} is stricter and would imply
the uniform bound $\xi_p/ p \geq \xi_0 > 0$, whereas here we only require $\xi_p \geq \xi_0 > 0$. 
As we are given a true DAG $D^0 \in \DPX$, we solely 
perform local transformations of DAGs
thanks to the transformational characterization result in Theorem~\ref{M-thm:covered-reversals}. This only affects the scores of two nodes 
and allows us to rely on this much weaker gap condition.
\end{rem}
\vspace{0.1cm}

\subsection{Estimation of $G_{\DPX}$} \label{M-ssec:estim_GDPX}

The estimation of all DAGs in $\DPX$ is feasible but may be computationally intractable in the presence of many linear edges. For example, if $D^0$ is a fully connected DAG with $p$ nodes and all its edges are linear, the number of DAGs in $\DPX$ corresponds to the number of causal orderings of $p$ nodes which is~$p!$. It therefore would be desirable to have a procedure that works without enumerating all DAGs in $\DPX$. In this section we are going to describe 
such a procedure that directly estimates the maximally oriented PDAG $G_{\DPX}$ defined in Section~\ref{M-ssec:char_graphical}. Recall that by Theorem~\ref{M-thm:pdag}, this fully characterizes $\DPX$, as $\DPX$ can be recovered from $G_{\DPX}$ by listing all consistent DAG extensions. 

The main idea is the following: instead of traversing the space of DAGs, we traverse the space of maximally oriented PDAGs that represent sets of distribution equivalent DAGs. As an example, let $D^0 \in \DPX$ and $i \rightarrow j$ be covered and linear in $D^0$. By Theorem~\ref{M-thm:covered-reversals}~(a), the DAG $D'$ that only differs from $D^0$ by the reversal of $i \rightarrow j$ is in $\DPX$. 
Instead of memorizing both, $D^0$ and $D'$, and recursively searching over sequences of covered linear edge reversals from both of these DAGs as in Algorithms~\ref{M-alg:listAllDAGsPLSEM} and~\ref{M-alg:listAllDAGsPLSEMscorebased}, we represent 
$D^0$ and $D'$ by the PDAG $G$ that is maximally oriented with respect to the set of DAGs $\{D^0, D' \}$. By Definition~\ref{M-def:PDAGrepr}, $G$ equals $D^0$ but for an undirected edge $i \text{ --- } j$.  
To construct $G_{\DPX}$, the idea is now to iteratively modify $G$ by either fixing or removing orientations of directed edges if they are nonlinear or linear in one of the consistent DAG extensions of $G$ in which they are covered. 
For that to work based on $G$ only, that is, without listing all consistent DAG extensions of $G$, the two key questions are the following: 
\begin{enumerate}
\item[(Q1)] For $i \rightarrow j$ in a maximally oriented PDAG $G$, can we decide based on $G$ only if there is a consistent DAG extension of $G$ in which $i \rightarrow j$ is covered?
\item[(Q2)] If $i \rightarrow j$ is known to be covered in a consistent DAG extension of $G$: can we derive a score-based check if $i \rightarrow j$ is linear or nonlinear in this extension based on $G$?
\end{enumerate}
Interestingly, the answer to both questions is yes (cf. Lemma~\ref{M-le:DAGextension}) 
and can be derived from a related theory on how background knowledge on specific edge orientations restricts the Markov equivalence class. 
It was shown in \citep[Theorems~2~\&~4]{meek1995} that for a
pattern~$P$ of a DAG, consistent background knowledge $\mathcal{K}$ (in our case: 
additional knowledge on edge orientations due to nonlinear functions in the PLSEM) can be incorporated by simply orienting these edges in~$P$ and closing orientations under a set of four sound and complete graphical orientation rules R1-R4, which are depicted in Figure~\ref{M-fig:orientationRules}. 
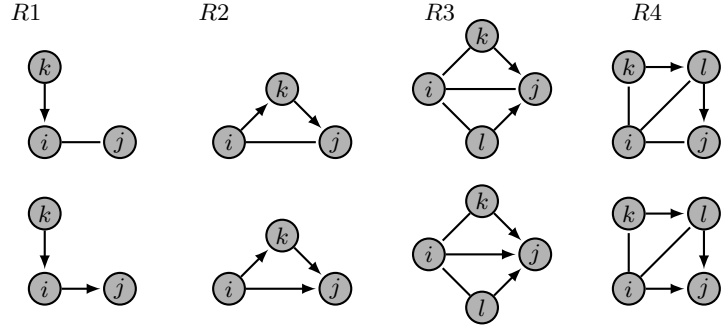
\begin{figure}
\centering
\begin{tikzpicture}[->,>=latex,shorten >=1pt,auto,node distance=1.0cm,
                    thick]
  \tikzstyle{every state}=[fill=black!30,draw=black,text=black, inner sep=0.4pt, minimum size=12pt]

  \node[state] (Xia) {$i$};
  \node[state] (Xka) [above of=Xia] {$k$};
  \node[state] (Xja) [right of=Xia] {$j$};
  
  \path (Xka) edge (Xia); 
  \path[-] 
  		(Xia) edge (Xja);

  \coordinate [label=above:$R1$] (L1) at (-0.25,1.5);

  \node[state] (Xib) [below=1.5cm of Xia]{$i$};
  \node[state] (Xkb) [above of=Xib] {$k$};
  \node[state] (Xjb) [right of=Xib] {$j$};
  
  \path (Xkb) edge (Xib)
  		(Xib) edge (Xjb);

  \node[state] (Xic) [right=1cm of Xja] {$i$};
  \node[state] (Xkc) [above right of=Xic] {$k$};
  \node[state] (Xjc) [below right of=Xkc] {$j$};
  
  \path (Xic) edge (Xkc)
  		(Xkc) edge (Xjc); 
  \path[-] 
  		(Xic) edge (Xjc);

  \coordinate [label=above:$R2$] (L2) at (2.25,1.5);

  \node[state] (Xid) [below=1.5cm of Xic]{$i$};
  \node[state] (Xkd) [above right of=Xid] {$k$};
  \node[state] (Xjd) [below right of=Xkd] {$j$};
  
  \path (Xid) edge (Xkd)
  		(Xkd) edge (Xjd)
  		(Xid) edge (Xjd);

  \node[state] (Xie) [right=1.5cm of Xkc] {$i$};
  \node[state] (Xke) [above right of=Xie] {$k$};
  \node[state] (Xle) [below right of=Xie] {$l$};
  \node[state] (Xje) [below right of=Xke] {$j$};
  
  \path (Xke) edge (Xje)
  		(Xle) edge (Xje); 
  \path[-] 
  		(Xie) edge (Xje)
  		(Xke) edge (Xie)
  		(Xle) edge (Xie);

  \coordinate [label=above:$R3$] (L3) at (5.25,1.5);	

  \node[state] (Xif) [below=1.75cm of Xie] {$i$};
  \node[state] (Xkf) [above right of=Xif] {$k$};
  \node[state] (Xlf) [below right of=Xif] {$l$};
  \node[state] (Xjf) [below right of=Xkf] {$j$};
  
  \path (Xkf) edge (Xjf)
  		(Xlf) edge (Xjf)  		
  		(Xif) edge (Xjf); 
  \path[-] 
  		(Xkf) edge (Xif)
  		(Xlf) edge (Xif);

  \node[state] (Xig) [right=1.5cm of Xle] {$i$};
  \node[state] (Xjg) [right of=Xig] {$j$};
  \node[state] (Xkg) [above of=Xig] {$k$};
  \node[state] (Xlg) [right of=Xkg] {$l$};
  
  \path (Xkg) edge (Xlg)
  		(Xlg) edge (Xjg);
  \path[-]
  		(Xig) edge (Xkg)
  		(Xig) edge (Xjg)
  		(Xig) edge (Xlg);
  		
  \coordinate [label=above:$R4$] (L4) at (8,1.5);

  \node[state] (Xih) [below=1.5cm of Xig] {$i$};
  \node[state] (Xjh) [right of=Xih] {$j$};
  \node[state] (Xkh) [above of=Xih] {$k$};
  \node[state] (Xlh) [right of=Xkh] {$l$};
  
  \path (Xkh) edge (Xlh)
		(Xih) edge (Xjh)
  		(Xlh) edge (Xjh);
  \path[-]
  		(Xih) edge (Xkh)
   		(Xih) edge (Xlh);

\end{tikzpicture}
\caption{Orientation rules R1-R4 for Markov equivalence classes with background knowledge from~\citep{meek1995}. If there is an edge constellation as in the top row, $i \text{ --- } j$ is oriented as $i \rightarrow j$ when closing orientations under R1-R4.}
\label{M-fig:orientationRules}
\end{figure}
The resulting PDAG, which we denote by $G_{P, \mathcal{K}}$, is maximally oriented with respect to the 
set
of all 
Markov equivalent 
DAGs with edge orientations that comply with the background knowledge. 

It is important to note that we generally do not obtain $G_{\DPX}$ if we simply add all nonlinear edges in $D^0$ as background knowledge $\mathcal{K}$ and close orientations under R1-R4. The resulting maximally oriented PDAG $G_{P,\mathcal{K}}$ is typically not equal to $G_{\DPX}$. For an example, consider~$D_1$ in Figure~\ref{M-fig:edge-status-change} and denote by $P_1$ its pattern. For $\mathcal{K} = \{1 \rightarrow 2\}$ we obtain 
the PDAG $G_{P_1,\mathcal{K}}$ with undirected edge $1 \text{ --- } 3$. But $1 \rightarrow 3$ in $G_{\DPX}$ by Definition~\ref{M-def:PDAGrepr} as $\DPX = \{D_1,D_2\}$. This illustrates that we have to add all edges to $\mathcal{K}$ that are nonlinear in a DAG in $\DPX$ in which they are covered ($1 \rightarrow 3$ is covered and nonlinear in~$D_2$). 




\begin{lemm} \label{M-le:DAGextension}
	Let $P$ be the 
	pattern of a DAG and $\mathcal{K}$ a consistent set of background knowledge (not containing 
	directed edges of $P$). Let $G_{P, \mathcal{K}}$ denote the maximally oriented graph with respect to $P$ and $\mathcal{K}$ with orientations closed under R1-R4. 
\begin{enumerate}
	\item[(a)] 
	Edge $i \rightarrow j$ in $\mathcal{K}$ is not covered in any of the consistent DAG extensions of $G_{P, \mathcal{K}}$ if and only if $G_{P, \mathcal{K}} = G_{P, \mathcal{K} \setminus \{i \rightarrow j\}}$.
	\item[(b)] 
	If $G_{P, \mathcal{K}} \neq G_{P, \mK \setminus \{i \rightarrow j\}}$, there exists a consistent DAG extension of $G_{P, \mK}$ in which $\pa_{G_{P, \mK}}(j) \setminus \{i\}$ is a cover for $i \rightarrow j$.
	\end{enumerate}
\end{lemm} 
A proof is given in Section~\ref{S-ssec:proofLemAlgorithm} in the supplement. By construction, $G_{P, \mathcal{K}} = G_{P, \mathcal{K} \setminus \{i \rightarrow j\}}$ if and only if the orientation of $i \rightarrow j$ in $G_{P, \mathcal{K} \setminus \{i \rightarrow j\}}$ is implied by one of R1-R4 applied to $G_{P, \mathcal{K}}$ with undirected edge $i \text{ --- } j$. Hence, Lemma~\ref{M-le:DAGextension}~(a) answers~(Q1) as it provides a simple graphical criterion to check whether $i \rightarrow j$ in $G_{P, \mathcal{K}}$ is covered in one of the consistent DAG extensions of $G_{P, \mathcal{K}}$ based on $G_{P, \mathcal{K}}$ only. 
Note that part (a) is closely related to~\citep[Section~5]{andersson1997}, where the authors construct the CPDAG (representing the Markov equivalence class) from a given DAG by removing edge orientations that are not implied by a set of graphical orientation rules, which contain R1-R3 in Figure~\ref{M-fig:orientationRules}.
Lemma~\ref{M-le:DAGextension}~(b) answers (Q2): it allows us to implement a score-based check whether $i \rightarrow j$ is linear or nonlinear in a DAG extension of $G_{P, \mathcal{K}}$ in which it is covered by simply reading off the parents of $j$ in $G_{P, \mathcal{K}}$ and use them as a cover for $i \rightarrow j$. Details are given in Remark~\ref{M-rem:score}.

We now propose the following iterative estimation procedure for $G_{\DPX}$:
let $D^0 \in \DPX$ be given, $P$ denote its pattern and 
define $\mathcal{K}_1 := \mathcal{K}^{\text{init}}_{1} \cup \mathcal{K}^{\text{nonl}}_{1}$, where $\mathcal{K}^{\text{init}}_{1}$ contains all directed edges in $D^0$ that are undirected in $P$ and $\mathcal{K}^{\text{nonl}}_{1} := \emptyset$. By construction, 
$G_{P,\mathcal{K}_1} = D^0$. 
For $k \geq 1$, in each iteration $k$ to $k+1$, we apply Lemma~\ref{M-le:DAGextension}~(a) and use R1-R4 to select $\{i \rightarrow j\} \in \mathcal{K}^{\text{init}}_k$ ($i \rightarrow j$ in $G_{P,\mathcal{K}_k}$) that is covered in a consistent DAG extension of $G_{P,\mathcal{K}_k}$ (that is, not implied by any of R1-R4). If $\mathcal{K}^{\text{init}}_k=\emptyset$ or no such edge exists, we stop and output $G_{P,\mathcal{K}_k}$. Else, we check whether $i \rightarrow j$ is linear or nonlinear in a consistent DAG extension in which it is covered and construct a new set of background knowledge $\mathcal{K}_{k+1} := \mathcal{K}^{\text{init}}_{k+1} \cup \mathcal{K}^{\text{nonl}}_{k+1} \subseteq \mathcal{K}_{k}$ according to the following rules: 
\\
\underline{Case 1}: If $i \rightarrow j$ is linear, $\mathcal{K}^{\text{nonl}}_{k+1} = \mathcal{K}^{\text{nonl}}_{k}$ and  $\mathcal{K}^{\text{init}}_{k+1} = \mathcal{K}^{\text{init}}_{k} \setminus \{i \rightarrow j\}$.  \\
\underline{Case 2}: If $i \rightarrow j$ is nonlinear, $\mathcal{K}^{\text{nonl}}_{k+1} = \mathcal{K}^{\text{nonl}}_{k} \cup \{i \rightarrow j\}$;  $\mathcal{K}^{\text{init}}_{k+1} = \mathcal{K}^{\text{init}}_{k} \setminus \{i \rightarrow j\}$. 

In particular, by construction, \underline{Case 1} implies that $i \text{ --- } j$ in all $G_{P,\mathcal{K}_{l}}$ for $l > k$, whereas \underline{Case 2} fixes the orientation $i \rightarrow j$ in all $G_{P,\mathcal{K}_{l}}$ for $l > k$. 

\begin{lemm} \label{M-le:correct-computeGDPX}
	Let $\{\mathcal{K}_k\}_k$ be constructed as above. Then, the corresponding sequence of maximally oriented PDAGs $\{G_{P,\mathcal{K}_k}\}_k$ converges to~$\GDPX$.
\end{lemm}
A proof is given in Section~\ref{S-ssec:prf-le-correct-computeGDPX} in the supplement and an illustration is provided in Figure~\ref{M-fig:exampleGDPXalgorithm}. As in both cases, $|\mathcal{K}^{\text{init}}_{k+1}| = |\mathcal{K}^{\text{init}}_{k}| - 1$, $\{G_{P,\mathcal{K}_k}\}_k$ converges to $\GDPX$
after at most $|\mathcal{K}^{\text{init}}_{1}|$ iterations, where $|\mathcal{K}^{\text{init}}_{1}|$ is the number of undirected edges in~$P$.

\begin{rem} \label{M-rem:score}
Let $\{i \rightarrow j\} \in \mathcal{K}^{\mathrm{init}}_{k}$ be the edge chosen in iteration $k$ to $k+1$. By Lemma~\ref{M-le:DAGextension}~(b), $S:=\pa_{G_{P,\mathcal{K}_k}}(j) \setminus \{i\}$ is a cover of $i \rightarrow j$ in one of the consistent DAG extensions of $G_{P,\mathcal{K}_k}$. From that, we easily obtain a score-based version:
we simply regress $X_i$ on $X_{S}$ and $X_j$ on $X_{S \cup \{i\}}$ to obtain the estimates $\hat{\sigma}_i, \hat{\sigma}_j$ of the standard deviations of the residuals at nodes $i$ and $j$ for $i \rightarrow j$. Similarly, we regress $X_i$ on $X_{S \cup \{j\}}$ and $X_j$ on $X_{S}$ to get $\hat{\sigma}_i', \hat{\sigma}_j'$ for $i \leftarrow j$. If the estimated score difference $|\log(\hat{\sigma}_i') + \log(\hat{\sigma}_j') - \log(\hat{\sigma}_i) - \log(\hat{\sigma}_j)|$ is smaller than $\alpha$, we conclude that $i \rightarrow j$ is linear, else, nonlinear. The pseudo-code of the score-based procedure is provided in Algorithm~\ref{M-alg:computeGDPX}. It outputs an estimate $\widehat{G}_{n,p}$ of $G_{\DPX}$ based on $n$ samples $X^{(1)},...,X^{(n)}$ and $D^0 \in \DPX$.
\end{rem}

\begin{algorithm}[!htb]
\begin{algorithmic}[1]
\STATE Initialize $\widehat{G}_{n,p} \leftarrow D^0$, $k \leftarrow 1$, $\mathcal{K}^{\text{init}}_{1} \leftarrow \emptyset$ and $\mathcal{K}^{\text{nonl}}_{1} \leftarrow \emptyset$.
\STATE Construct the pattern $P$ of $D^0$.
\STATE Add directed edges in $D^0$ that are undirected in $P$ to $\mathcal{K}^{\text{init}}_{1}$. 
\WHILE{There is $i \rightarrow j$ in $\mathcal{K}^{\text{init}}_{k}$, such that its orientation is not implied by applying rules R1, R2, R3 or R4 to $\widehat{G}_{n,p}$ with undirected edge $i \text{ --- } j$}
\STATE Use $\pa_{\widehat{G}_{n,p}}(j) \setminus \{i\}$ to cover $i \rightarrow j$ and estimate the standard deviations $\hat{\sigma}_i, \hat{\sigma}_j, \hat{\sigma}_i'$ and $\hat{\sigma}_j'$ of the residuals as described in Remark~\ref{M-rem:score}.
\IF{$|\log(\hat{\sigma}_i') + \log(\hat{\sigma}_j')  - \log(\hat{\sigma}_i) - \log(\hat{\sigma}_j)| < \alpha$}
\STATE Set $\mathcal{K}^{\text{init}}_{k+1} \leftarrow \mathcal{K}^{\text{init}}_{k} \setminus \{i \rightarrow j\}$ and replace $i \rightarrow j$ by $i \text{ --- } j$ in $\widehat{G}_{n,p}$.
\ELSE
\STATE Set $\mathcal{K}^{\text{init}}_{k+1} \leftarrow \mathcal{K}^{\text{init}}_{k} \setminus \{i \rightarrow j\}$ and keep $i \rightarrow j$ in $\widehat{G}_{n,p}$.
\ENDIF
\STATE $k \leftarrow k + 1$.
\ENDWHILE
\RETURN Estimated PDAG $\widehat{G}_{n,p}$ representing $\DPX$.
\end{algorithmic}
\caption{computeGDPX}\label{M-alg:computeGDPX}
\end{algorithm}

A major advantage of Algorithm~\ref{M-alg:computeGDPX} is that it can be implemented based on one adjacency matrix only that is updated in every iteration.

\begin{thm} \label{M-thm:consistencyGDPXAlgorithm}
	Under Assumption~\ref{M-assu:consistency} and $\xi_p \geq \xi_0 > 0$, for any constant $\alpha \in (0, \xi_0)$,
	$$
		\mathbb{P} \left[ \widehat{G}_{n,p} = G_{\DPX} \right] \rightarrow 1 \qquad (n \rightarrow \infty) 
	$$
\end{thm}

\begin{prf}
The correctness of Algorithm~\ref{M-alg:computeGDPX} is proved in Lemma~\ref{M-le:correct-computeGDPX}. The consistency of the score-based estimation follows from the proof of Theorem~\ref{M-thm:consistencyRecursiveAlgorithm}.
\end{prf}

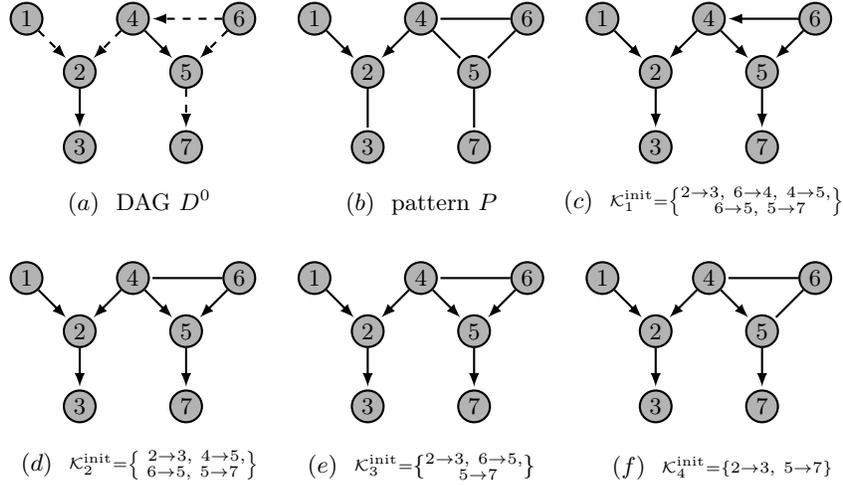
\begin{figure}[h]
\centering
\begin{tikzpicture}[->,>=latex,shorten >=1pt,auto,node distance=1.0cm,
                    thick]
  \tikzstyle{every state}=[fill=black!30,draw=black,text=black, inner sep=0.4pt, minimum size=12pt]
  
  \node[state] (X1a) {$1$};
  \node[state] (X2a) [below right of=X1a] {$2$};
  \node[state] (X3a) [below of=X2a] {$3$};
  \node[state] (X4a) [above right of=X2a] {$4$};
  \node[state] (X5a) [below right of=X4a] {$5$};
  \node[state] (X6a) [above right of=X5a] {$6$};
  \node[state] (X7a) [below of=X5a] {$7$};

  \path (X1a) edge [dashed] (X2a) 
  		(X2a) edge (X3a)
  		(X4a) edge [dashed] (X2a)
  		(X4a) edge (X5a)
  		(X6a) edge [dashed](X4a)
		(X6a) edge [dashed] (X5a)
		(X5a) edge [dashed] (X7a);

  \node[state] (X1b) [right of=X6a]{$1$};
  \node[state] (X2b) [below right of=X1b] {$2$};
  \node[state] (X3b) [below of=X2b] {$3$};
  \node[state] (X4b) [above right of=X2b] {$4$};
  \node[state] (X5b) [below right of=X4b] {$5$};
  \node[state] (X6b) [above right of=X5b] {$6$};
  \node[state] (X7b) [below of=X5b] {$7$};
    
  \path (X1b) edge (X2b)
  		(X4b) edge (X2b);
  \path[-] (X4b) edge (X5b)
  		(X2b) edge (X3b)
  		(X6b) edge (X4b)
		(X6b) edge (X5b)
		(X5b) edge (X7b);

  \node[state] (X1c) [right of=X6b]{$1$};
  \node[state] (X2c) [below right of=X1c] {$2$};
  \node[state] (X3c) [below of=X2c] {$3$};
  \node[state] (X4c) [above right of=X2c] {$4$};
  \node[state] (X5c) [below right of=X4c] {$5$};
  \node[state] (X6c) [above right of=X5c] {$6$};
  \node[state] (X7c) [below of=X5c] {$7$};
    
  \path (X1c) edge (X2c) 
  		(X2c) edge (X3c)
  		(X4c) edge (X2c)
  		(X4c) edge (X5c)
  		(X6c) edge (X4c)
		(X6c) edge (X5c)
		(X5c) edge (X7c);
  
  \coordinate [label=above:$(a) \ \ \text{DAG } D^0$] (L1) at (1.5,-2.75);	
  \coordinate [label=above:$(b) \ \ \text{pattern } P$] (L1) at (5.25,-2.75);	
  \coordinate [label=above:$(c) \ \ \substack{\mathcal{K}^{\text{init}}_{1} = \left\{ \substack{ 2 \rightarrow 3, \ 6 \rightarrow 4, \ 4 \rightarrow 5, \\ \ 6 \rightarrow 5, \ 5 \rightarrow 7} \right\}}$] (L1) at (9,-2.75);

  \node[state] (X1d) [below=3cm of X1a] {$1$};
  \node[state] (X2d) [below right of=X1d] {$2$};
  \node[state] (X3d) [below of=X2d] {$3$};
  \node[state] (X4d) [above right of=X2d] {$4$};
  \node[state] (X5d) [below right of=X4d] {$5$};
  \node[state] (X6d) [above right of=X5d] {$6$};
  \node[state] (X7d) [below of=X5d] {$7$};

  \path (X1d) edge (X2d) 
  		(X2d) edge (X3d)
  		(X4d) edge (X2d)
  		(X4d) edge (X5d)
		(X6d) edge (X5d)
		(X5d) edge (X7d);
  \path[-]
        (X6d) edge (X4d);

  \node[state] (X1e) [right of=X6d]{$1$};
  \node[state] (X2e) [below right of=X1e] {$2$};
  \node[state] (X3e) [below of=X2e] {$3$};
  \node[state] (X4e) [above right of=X2e] {$4$};
  \node[state] (X5e) [below right of=X4e] {$5$};
  \node[state] (X6e) [above right of=X5e] {$6$};
  \node[state] (X7e) [below of=X5e] {$7$};
    
  \path (X1e) edge (X2e)
  		(X2e) edge (X3e)
  		(X4e) edge (X2e)
		(X6e) edge (X5e)
		(X4e) edge (X5e)		
		(X5e) edge (X7e);
  \path[-]   	
		(X6e) edge (X4e);

  \node[state] (X1f) [right of=X6e]{$1$};
  \node[state] (X2f) [below right of=X1f] {$2$};
  \node[state] (X3f) [below of=X2f] {$3$};
  \node[state] (X4f) [above right of=X2f] {$4$};
  \node[state] (X5f) [below right of=X4f] {$5$};
  \node[state] (X6f) [above right of=X5f] {$6$};
  \node[state] (X7f) [below of=X5f] {$7$};
    
  \path (X1f) edge (X2f) 
  		(X2f) edge (X3f)
  		(X4f) edge (X2f)
  		(X4f) edge (X5f)
  		(X5f) edge (X7f);
  \path[-] (X6f) edge (X4f)
  		   (X6f) edge (X5f);


  	\coordinate [label=above:$(d) \ \ \substack{\mathcal{K}^{\text{init}}_{2} = \left\{ \substack{ \ 2 \rightarrow 3, \ 4 \rightarrow 5, \\ 6\rightarrow 5, \ 5 \rightarrow 7} \right\}}$] (L1) at (1.5,-6.25);
	\coordinate [label=above:$(e) \ \ \substack{\mathcal{K}^{\text{init}}_{3} = \left\{\substack{2 \rightarrow 3, \ 6\rightarrow 5, \\ \ 5 \rightarrow 7} \right\} }$] (L1) at (5.25,-6.25);
	\coordinate [label=above:$(f) \ \ \substack{\mathcal{K}^{\text{init}}_{4} = \{2 \rightarrow 3, \ 5 \rightarrow 7\} }$] (L1) at (9.25,-6.25);

\end{tikzpicture}
\caption{
Illustration of Algorithm~\ref{M-alg:computeGDPX}. (a) DAG $D^0$ with linear edges (dashed) and nonlinear edges (solid). (b) step 2: pattern $P$ of $D^0$. (c) step 3: directed edges in $D^0$ that are undirected in $P$ are added 
to $\Kinit_1$. By construction, $\widehat{G}_{n,p} = D^0$. (c)-(f) steps 4-12: $4 \leftarrow 6$ is covered and linear in (c), hence, orientation is removed in $\widehat{G}_{n,p}$ in (d). $4 \rightarrow 5$ is covered and nonlinear in (c), hence, orientation is fixed in $\widehat{G}_{n,p}$ in (e). $6 \rightarrow 5$ is covered and linear in a consistent DAG extension of (e), hence, orientation is removed in $\widehat{G}_{n,p}$ in (f). As both edges in $\Kinit_4$ are implied by R1 in (f), they are not covered in any of the consistent DAG extensions of $\widehat{G}_{n,p}$ in (f). Concludingly, $\widehat{G}_{n,p}=\GDPX$ in (f).
}
\label{M-fig:exampleGDPXalgorithm}
\end{figure}

\section{Model misspecification} \label{M-sec:misspec}

In this section we will discuss how small deviations from a Gaussian error distribution affect the distribution equivalence class and how the output of the algorithm \texttt{listAllDAGsPLSEM} should be interpreted in this case. We define a \emph{generalized PLSEM} by essentially dropping the assumption of Gaussianity of the noise variables from the definition of a PLSEM:
\begin{defn}[generalized PLSEM]\label{M-defn:gPLSEM}
A \emph{generalized PLSEM} with DAG $D$ is a partially linear additive SEM of the form:
\begin{align} \label{M-eq:gPLSEM}
	\mathring X_j & = \mathring \mu_j + \sum\limits_{i \in \pa_D(j)} \mathring f_{j,i}(\mathring X_i) +\mathring \eps_j, 
\end{align}
where $\mathring \mu_j \in \mathbb{R}$, $\mathring f_{j,i} \in C^2(\mathbb{R})$, $\mathring f_{j,i} \not\equiv 0$, with $\EE[ \mathring f_{j,i}(X_i)]=0$, and the noise variables $\mathring \eps_j $ are  centered with variance $\mathring \sigma_{j}^{2} >0 $, have positive density on $\mathbb{R}$, and are jointly independent  for $j=1,...,p$.
\end{defn}
 In analogy to before, without loss of generality, we assume $\mathring \mu_{j}= 0, j=1,\ldots,p$ and define projected parameters. Furthermore, for a DAG $D$, we will define the projected density $\mathring p_{\mathring \theta^{D}}^{D}$. Consider $\mathring X \sim \mathring{\PX}$ generated by a generalized PLSEM with DAG $D^{0}$.  For each DAG $D$ that is Markov equivalent to $D^{0}$, define
\begin{align*}
  \{ \mathring f_{j,i}^{D}\}_{i \in \pa_{D}(i)} &= \argmin_{g_{j,i} \in \mathcal{F}_{i}} \mathbb{E}[( \mathring X_{j} - \sum_{i \in \pa_{D}(i)} g_{j,i}(\mathring X_{i}))^{2}]  \\
(\mathring \sigma_{j}^{D})^{2} &=  \mathbb{E}[(\mathring X_{j} - \sum_{i \in \pa_{D}(j)} \mathring  f_{j,i}(\mathring X_{i}))^{2}]\\
\mathring p_{\mathring \theta^{D}}^{D}(x) &= \prod_{j=1}^{p} \mathring q_{j} (  x_{j} - \sum_{i \in \pa_{D}(j)} \mathring f_{j,i}( x_{i})),
\end{align*}
where $\mathring q_{j}$ denotes the density of $\mathring X_{j} - \sum_{i \in \pa_{D}(j)} \mathring f_{j,i}(\mathring X_{i})$ for $j=1,\ldots,p$.  Analogously define the projected density $p_{\theta^{D}}^{D}$ of $X \sim \mathbb{P}$ generated by a (Gaussian) PLSEM. Note that here $p_{\theta^{D}}^{D}$ denotes the projected density of $X$ with respect to generalized PLSEMs, in contrast to Section~ \ref{M-sec:estim} where it denotes the projected density of $X$ with respect to (Gaussian) PLSEMs. \\
The first question we answer is: How does the algorithm \texttt{listAllDAGsPLSEM} behave when the error distributions are non-Gaussian (and the algorithm wrongly assumes Gaussianity)?  The following result answers this question in the population case. Let $X \sim  \PX$ be generated by a  PLSEM with the same DAG, same edge functions, same error variances as the generalized PLSEM that generates $\mathring X$,  but with Gaussian errors.  It turns out that the DAGs in the distribution equivalence class $\mathscr{D}(\mathbb{P})$ have an interesting property. For all $D \in \mathscr{D}(\mathbb{P})$, the computed scores are lower than the score for $D^{0}$.  The proof of Theorem~\ref{M-thm:lower-loglik} can be found in Section~\ref{S-sec:proofs-model-missp} in the supplement.
\begin{thm}\label{M-thm:lower-loglik}
For all $D \in \mathscr{D}( \mathbb{P})$,
\begin{equation*}
  \sum_{j=1}^{p} \log \mathring \sigma_{j}^{D} \le \sum_{j=1}^{p} \log  \mathring \sigma_{j}^{D^{0}}.
\end{equation*}
\end{thm}
Hence if the algorithm \texttt{listAllDAGsPLSEM} starts at $D^{0}$ with $\alpha \ge 0$, it will never reject any $D \in \mathscr{D}(\mathbb{P})$. The output of the algorithm will hence be a superset of $\mathscr{D}(\mathbb{P})$. \\
From a theoretical perspective, the other question might be more interesting: what statements can be made about the distribution equivalence class of $\mathring X \sim \mathring{\mathbb{P}}$? To be more precise, for a distribution~$\mathring \PX$ that has been generated by a faithful generalized PLSEM, we call the set of DAGs
\begin{align*} 
\begin{split}
	\mathring{\mathscr{D}}(\mathring{\mathbb{P}}) := \left\{ D \ \begin{array}{|l}  \mathring{\PX} \text{ is faithful to } D \text{ and there exists a} \\ \text{generalized PLSEM with DAG } D \text{ that generates } \mathring{\PX}  \end{array} \right\} 
\end{split}
\end{align*}
the \emph{(generalized PLSEM) distribution equivalence class}. 
How do small violations of Gaussianity affect the distribution equivalence class? Intuitively, identification of certain edges should get easier, in the sense that previously identified edges stay identified. This intuition turns out to be correct. The following theorem tells us that small deviations from the Gaussian error distribution can only make the distribution equivalence class smaller. 
\begin{thm}\label{M-thm:close-loglik}
 Let
\begin{equation*}
  | \mathbb{E}[ \log p_{\theta^{D}}^{D} (X)] -  \mathbb{E}[ \log \mathring p_{\mathring \theta^{D}}^{D}(\mathring X)] | < \zeta 
\end{equation*}
for all DAGs $D \sim D^{0}$ (all DAGs $D$ that are Markov equivalent to $D^{0}$) for $\zeta >0$ sufficiently small. Then we have
\begin{equation*}
  \mathring{ \mathscr{D}}( \mathring{\mathbb{P}}) \subseteq  \mathscr{D}(  \mathbb{P}).
\end{equation*}
\end{thm}
The proof of Theorem~\ref{M-thm:close-loglik} and the definition of a feasible $\zeta >0$ can be found in Section~\ref{S-sec:proofs-model-missp}  in the supplement. In words, the assumption requires that the projected log-likelihoods of $X$ and $\mathring X$ do not differ too much for all DAGs $D \sim D^{0}$. If the error distribution of $\mathring \eps$ is close to Gaussian, then the distributions of $X$ and $\mathring X$ are close and the assumption is fulfilled. 

We now collect the implications of these theorems for the population case. By Theorem~\ref{M-thm:lower-loglik}, the output of the algorithm \texttt{listAllDAGsPLSEM} is a superset of $\mathscr{D}(  \mathbb{P}) $. Furthermore, under the assumptions of Theorem~\ref{M-thm:close-loglik}, $\mathscr{D}(  \mathbb{P}) $ is a superset of $\mathring{\mathscr{D}}(\mathring{\mathbb{P}})$. Hence, the algorithm is conservative in the sense that it will return a superset of the true underlying distribution equivalence class $\mathring{ \mathscr{D}}( \mathring{\mathbb{P}})$. In particular, it will not draw any wrong causal conclusions as it will not return incorrectly oriented edges.\\
Does  \texttt{listAllDAGsPLSEM} sometimes return a proper superset  of $\mathring{\mathscr{D}}(\mathring{\mathbb{P}})$? Intuitively, the algorithm only orients edges that are identified due to nonlinear edge functions. However, edges in generalized PLSEMs can sometimes be identified due to non-Gaussianity of certain error distributions. The algorithm \texttt{listAllDAGsPLSEM} does not take the latter into account. In such a case, under the assumptions of Theorem~\ref{M-thm:close-loglik}, the algorithm will usually output a proper superset of the distribution equivalence class. An example can be found below. To compute the distribution equivalence class $\mathring{\mathscr{D}}(\mathring{\mathbb{P}})$, we recommend to compute the log-likelihoods of the output of \texttt{listAllDAGsPLSEM} with a nonparametric log-likelihood estimator (e.g., in the spirit of \cite{nowzopb13}) and keep the DAGs with the largest corresponding log-likelihoods. Under the assumptions of Theorem~\ref{M-thm:close-loglik}, this would return the exact generalized PLSEM distribution equivalence class $\mathring{\mathscr{D}}(\mathring{\mathbb{P}})$. In this case, the main benefit of \texttt{listAllDAGsPLSEM} is to reduce the computational burden compared to more naive approaches, such as computing nonparametric log-likelihood estimates of all DAGs in the Markov equivalence class.

\begin{example} Consider the generalized PLSEM  $X_{1} \leftarrow \epsilon_{1}$,$X_{2} \leftarrow \frac{1}{\sqrt{2}} X_{1} + \epsilon_{2}$,  where $\eps_{1} $ and $\eps_{2}$ both follow a scaled $t$-distribution, with $\text{Var}(\eps_{1}) =1$ and $\text{Var}(\eps_{2}) = \frac{1}{2}$. From \cite{hoy09} it follows that there exists no additive backward model, i.e. there exists no generalized PLSEM with $X_{2} \rightarrow X_{1}$ that generates the given distribution of $(X_{1},X_{2})$. However, the ``residuals'' $r_{1} := X_{2}$ and $r_{2} := X_{1} - \frac{1}{\sqrt{2}} X_{2}$ satisfy $\text{Var}(r_{1})=1$ and $\text{Var}(r_{2})=\frac{1}{2}$. Hence the projected (Gaussian) log-likelihoods of these two models match. In this case, \texttt{listAllDAGsPLSEM} would return the two DAGs $ X_{1} \rightarrow X_{2}$  and $X_{2}\rightarrow X_{1}$, which is a strict superset of  $\mathring{\mathscr{D}}(\mathring{\mathbb{P}}) =\{X_{1} \rightarrow X_{2} \}$.
\end{example}


\section{Simulations} \label{M-sec:simulations}
In this section we empirically analyze the performance of \texttt{computeGDPX} (Algorithm~\ref{M-alg:computeGDPX}) in various settings. Consider $\PX$ that has been generated by a faithful PLSEM with known DAG $D^0$. 
The goal is to estimate the corresponding distribution equivalence class $\DPX$ based on 
$D^0$ and samples of 
$\PX$.
In Section~\ref{M-ssec:sim-setting}, we start 
with a description of 
the simulation setting. 
We then briefly comment on a population version of Algorithm~\ref{M-alg:computeGDPX} in Section~\ref{M-ssec:sim-reference}, which is used to obtain the underlying true distribution equivalence class $\DPX$. In the subsequent sections we examine the role of the tuning parameter $\alpha$ (Section~\ref{M-ssec:sim-alpha}), the performance in low- and high-dimensional settings (Section~\ref{M-ssec:sim-p}) and the computation time (Section~\ref{M-ssec:sim-time}).

\subsection{Simulation setting and implementation details} \label{M-ssec:sim-setting}
Throughout the section, let $p$ denote the number of variables, $n$ the number of samples, $\nrep$ the number of repetitions of an experiment, $\pc$ the probability to connect two nodes by an edge and $\pl$ the probability that an edge is linear. For each experiment we generate $\nrep$ random true DAGs $D^0$ with the function \texttt{randomDAG} in the R-package \texttt{pcalg}~\citep{kaetal11_2} with parameters \texttt{n =} $p$ and \texttt{prob = }$\pc$. For each of the random DAGs, we generate $n$ samples of $\PX$ from a PLSEM with edge functions chosen as follows: 
with probability $\pl$, $f_{j,i}(x) = \alpha_{j,i} \cdot x$ is linear with $\alpha_{j,i}$ randomly drawn from $[-1.5,-0.5] \cup [0.5,1.5]$. Otherwise, $f_{j,i}(x)$ is nonlinear and randomly drawn from the set $\{c_0 \cdot \cos(c_1 \cdot (x - c_2)), c_0 \cdot \tanh(c_1 \cdot (x - c_2))\}$ to have a mix of monotone and non-monotone functions in the PLSEM. In order to be able to empirically support our theoretical findings we choose the parameters $c_0 \sim \text{Unif}([-2,-1]\cup[1,2]), c_1\sim \text{Unif}([1,2])$ and $c_2 \sim \text{Unif}([-\pi/3, \pi/3])$ such that the nonlinear functions are ``sufficiently nonlinear'' and not too close to linear functions. Exemplary randomly generated nonlinear functions are shown in Figure~\ref{M-fig:nonlinfcts}. The noise variables satisfy $\eps_j \sim \mathcal{N}(0, \sigma_j^2)$ with $\sigma_j^2 \sim \text{Unif}([1,2])$ for source nodes (nodes with empty parental set) and $\sigma_j^2 \sim \text{Unif}([1/4,1/2])$ otherwise. 
\begin{figure}[h]
\centering
\includegraphics[width=0.9\textwidth]{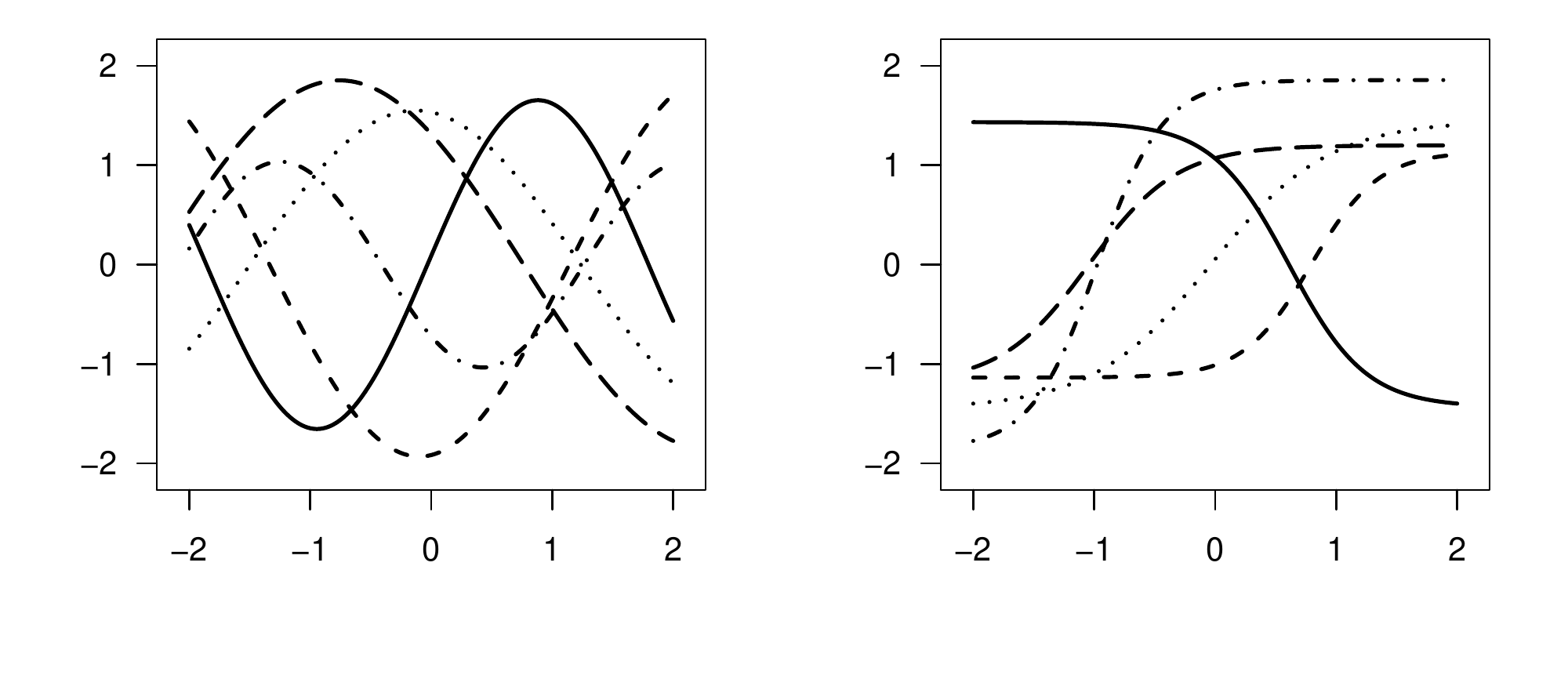}
\vspace{-1cm} 
\caption{Exemplary nonlinear functions used in simulated PLSEMs.}
\label{M-fig:nonlinfcts}
\end{figure}

In order to estimate the residuals in step~5 of \texttt{computeGDPX}, we use additive model fitting based on the \texttt{R}-package \texttt{mgcv} with default settings~\citep{wood03,wood06}. The basis dimension for each smooth term is set to 6.

There exists no state-of-the-art method that we can compare our algorithm with. In principle, given $D^0$, we can estimate the corresponding PLSEMs for all DAGs in the Markov equivalence class of $D^0$ and compute their scores. This also gives us an estimate for $\DPX$, but as explained in Section~\ref{M-ssec:estim_GDPX}, is less efficient than \texttt{computeGDPX}. We therefore only evaluate how accurately \texttt{computeGDPX} estimates $\GDPX$. For that,
let $G_{\DPX}$ and $\hat{G}$ denote the true and estimated graphical representations of $\DPX$, respectively. 
We count (i) the number of edges that are undirected in $G_{\DPX}$ but directed in $\hat{G}$ (``falsely kept orientations'') and (ii) the number of edges that are directed in $G_{\DPX}$ but undirected in $\hat{G}$ (``falsely removed orientations''). Note that as we assume faithfulness, all DAGs in $\DPX$ have the same CPDAG. By construction, \texttt{computeGDPX} does not falsely remove orientations on the directed part of the CPDAG as all these edges are not covered in any of the consistent DAG extensions. To obtain the percentages shown in Figures~\ref{M-fig:effectalpha_sparse} to~\ref{M-fig:effectp} we therefore only divide by the number of undirected edges in the CPDAG. The percentages then reflect a measure for the fraction of ``correct score-based decisions''. 

\subsection{Reference method for true distribution equivalence class $\DPX$} \label{M-ssec:sim-reference}
To be able to characterize the true distribution equivalence class based on $D^0$ and the corresponding PLSEM we 
assume that for each $i \in \{1,...,p\}$, the functions in the set $\{   \partial_{i}^{2} f_{j,i} :j \text{ is a child of } i \text{ in } D^0 \text{ and } f_{j,i} \text{ is nonlinear} \}_i$ are linearly independent for 
the 
PLSEM with DAG $D^0$
that generates $\PX$. As all functions in our simulations are randomly drawn (cf. Section~\ref{M-ssec:sim-setting}), the assumption is satisfied with probability one for $D^0$ and the corresponding edge functions. 

This additional assumption rules out cases where nonlinear effects in $D^0$ exactly cancel out over different paths and hence excludes cases as in Figure~\ref{M-fig:nonlin-cancel} where nonlinear edges may be reversed. In particular, it allows us 
to obtain
$\GDPX$ only based on $D^0$ and knowledge of the functions in the corresponding PLSEM: 
%
first, we use Theorem~\ref{S-thm:nonlin-desc}~(c) in the supplement to construct the set $\mathcal{V}$. For all nodes $i$ in $D^0$, corresponding sets of nonlinear children $C_i$ 
(as defined in Section~\ref{S-sec:prf-interplay} in the supplement) and $k \neq i$, we add $(i,k)$ to $\mathcal{V}$ if $k$ is a descendant of a node in $C_i$. 
In principle, we now apply Algorithm~\ref{M-alg:computeGDPX}, but instead of the score-based decision in steps~6-9, we use the set $\mathcal{V}$ to decide about edge orientations. Let $i \rightarrow j$ be the edge chosen in step~4 and $D$ one of the consistent DAG extensions in which $i \rightarrow j$ is covered. If $(i,j) \in \mathcal{V}$, by Theorem~\ref{S-thm:nonlin-desc}~(d) and Remark~\ref{S-rem:uniqueness-V} in the supplement, $i \rightarrow j$ in all DAGs of a PLSEM that generates $\PX$. Hence, in particular, $i \rightarrow j$ in all DAGs in $\DPX$ and by definition, $i \rightarrow j$ in $G_{\DPX}$. If $(i,j) \not \in \mathcal{V}$, by Lemma~\ref{S-lem:trans-char-via-V} in the supplement,
the DAG~$D'$ that differs from $D$ only by reversing $i \rightarrow j$ is in $\DPX$. Hence, by definition, $i \text{ --- } j$ in $G_{\DPX}$. 


\subsection{The role of $\alpha$ for varying sample size} \label{M-ssec:sim-alpha}
In \texttt{computeGDPX}, the score-based decision 
whether a selected covered edge is linear or nonlinear is based on a comparison of the absolute difference of the expected negative log-likelihood scores of two models with a parameter $\alpha$. Optimally, one would choose 
$\alpha$ close to $\xi_p$,
see equation~\eqref{M-eq:xip}, but $\xi_p$ depends on the setting (number of variables, sparsity of the DAG, degree of nonlinearity of the nonlinear functions,~etc.) and is unknown. 
In practice, the parameter $\alpha$ reflects a measure of how conservative the estimate $\hat{G}$ of $\GDPX$ is (in the sense of how many causal statements can be made). 
For example, choosing $\alpha$ large results in a conservative estimate 
$\hat{G}$ with many undirected edges (a large set $\DPX$ of equivalent DAGs). In Figures~\ref{M-fig:effectalpha_sparse} and~\ref{M-fig:effectalpha_dense}, we empirically analyze the dependence of $\hat{G}$ on $\alpha$ for different sample sizes for sparse and dense graphs, respectively.
\begin{figure}[h]
\centering
\includegraphics[width=0.75\textwidth]{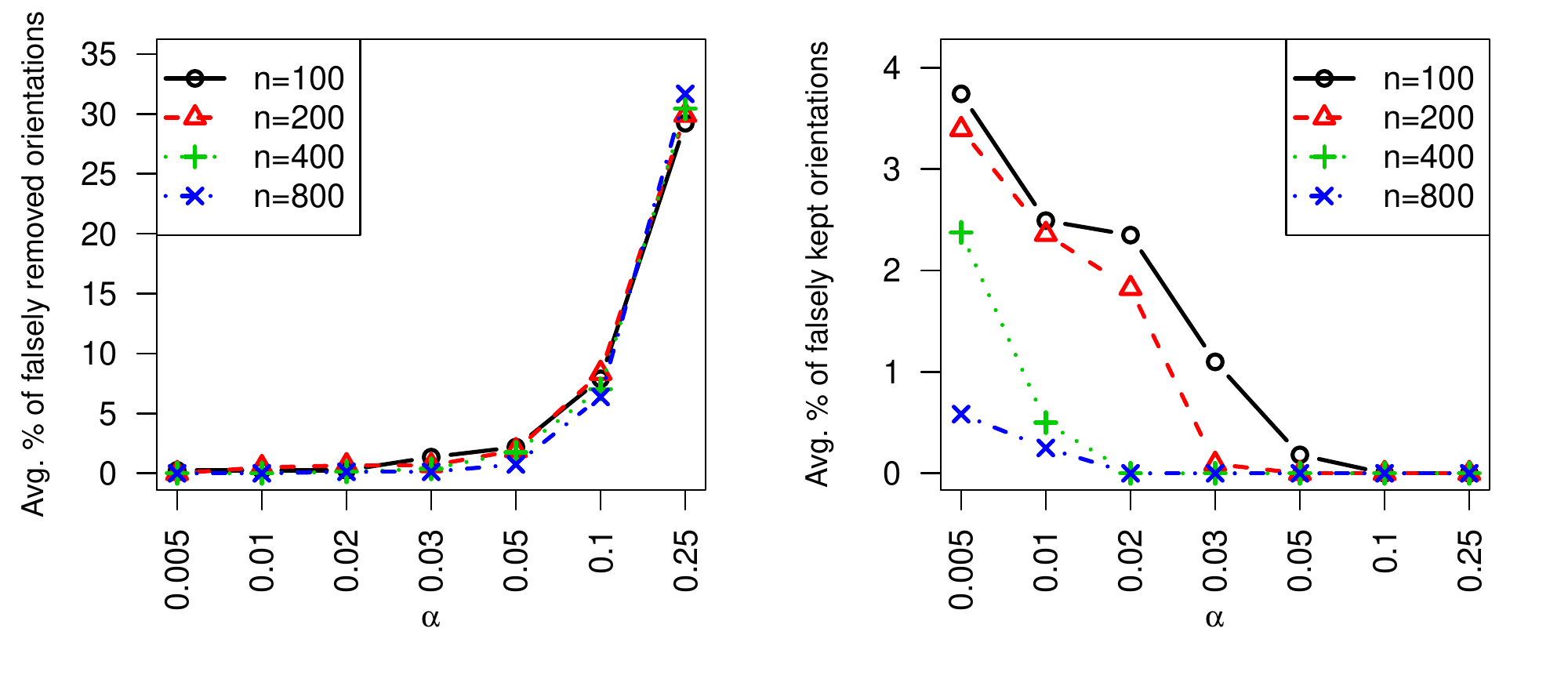} \\
\vspace{-0.4cm}
\includegraphics[width=0.75\textwidth]{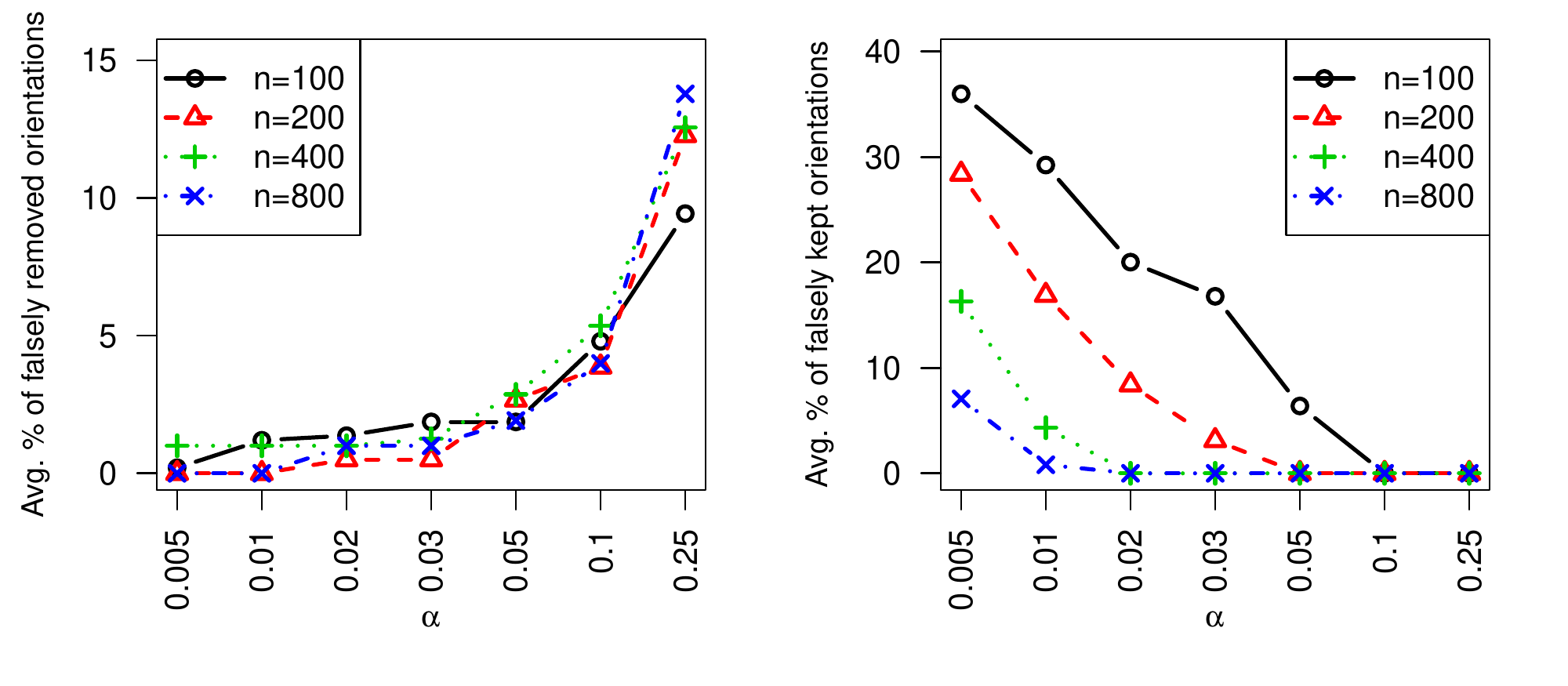} \\
\vspace{-0.4cm}
\caption{Performance of \texttt{computeGDPX} for varying sample sizes and values of $\alpha$ (x-axis) in sparse DAGs with $\pl=0.2$ (top) and $\pl=0.8$ (bottom). Parameters: $p=10$, $\nrep=100$ and $\pc=2/9$ (expected number of edges: $10$).}
\label{M-fig:effectalpha_sparse}
\end{figure}

\texttt{computeGDPX} exhibits a good performance for a wide range of values of $\alpha$. In particular, as the sample size increases, choosing $\alpha$ small results in 
very 
accurate estimates 
$\hat{G}$ 
of $G_{\DPX}$. The sparsity of the DAG does not strongly influence the results.
\begin{figure}[h]
\centering
\includegraphics[width=0.75\textwidth]{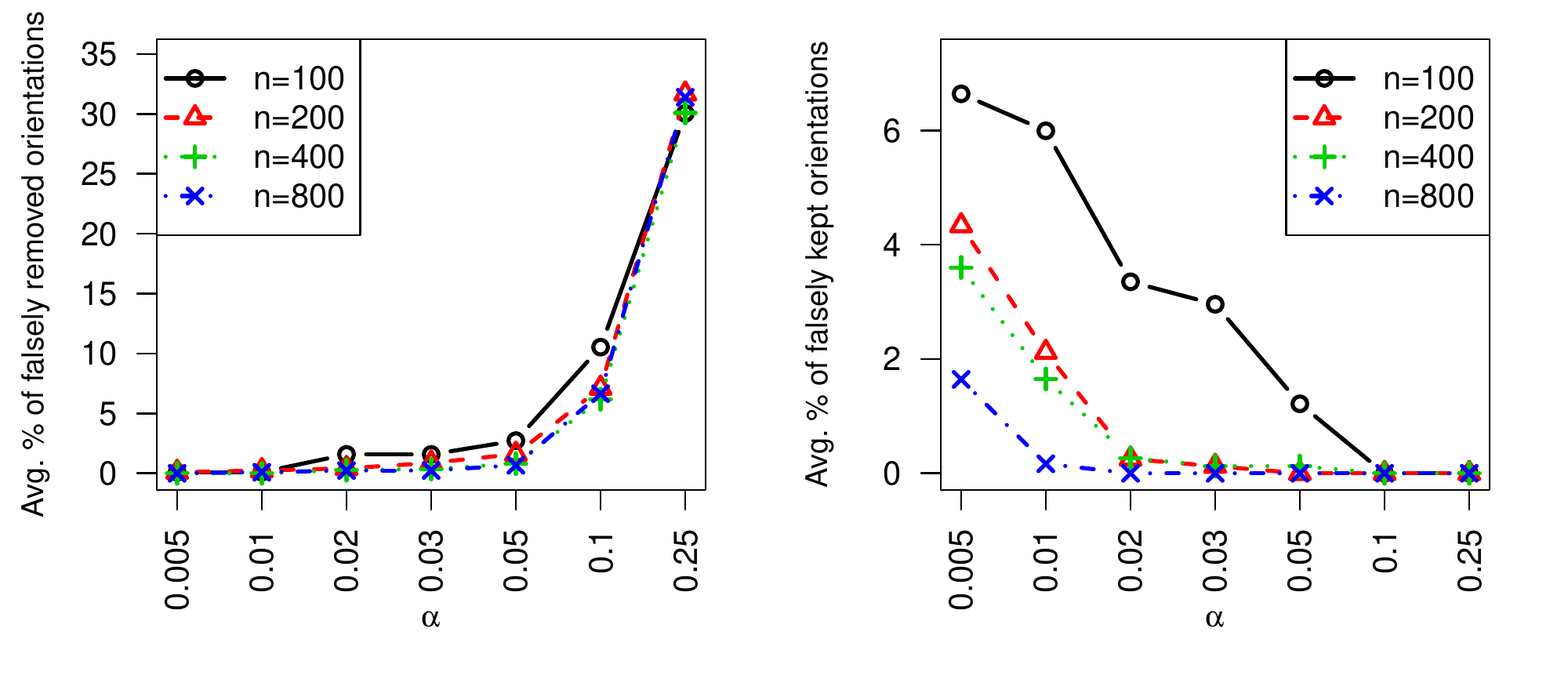} \\
\vspace{-0.4cm}
\includegraphics[width=0.75\textwidth]{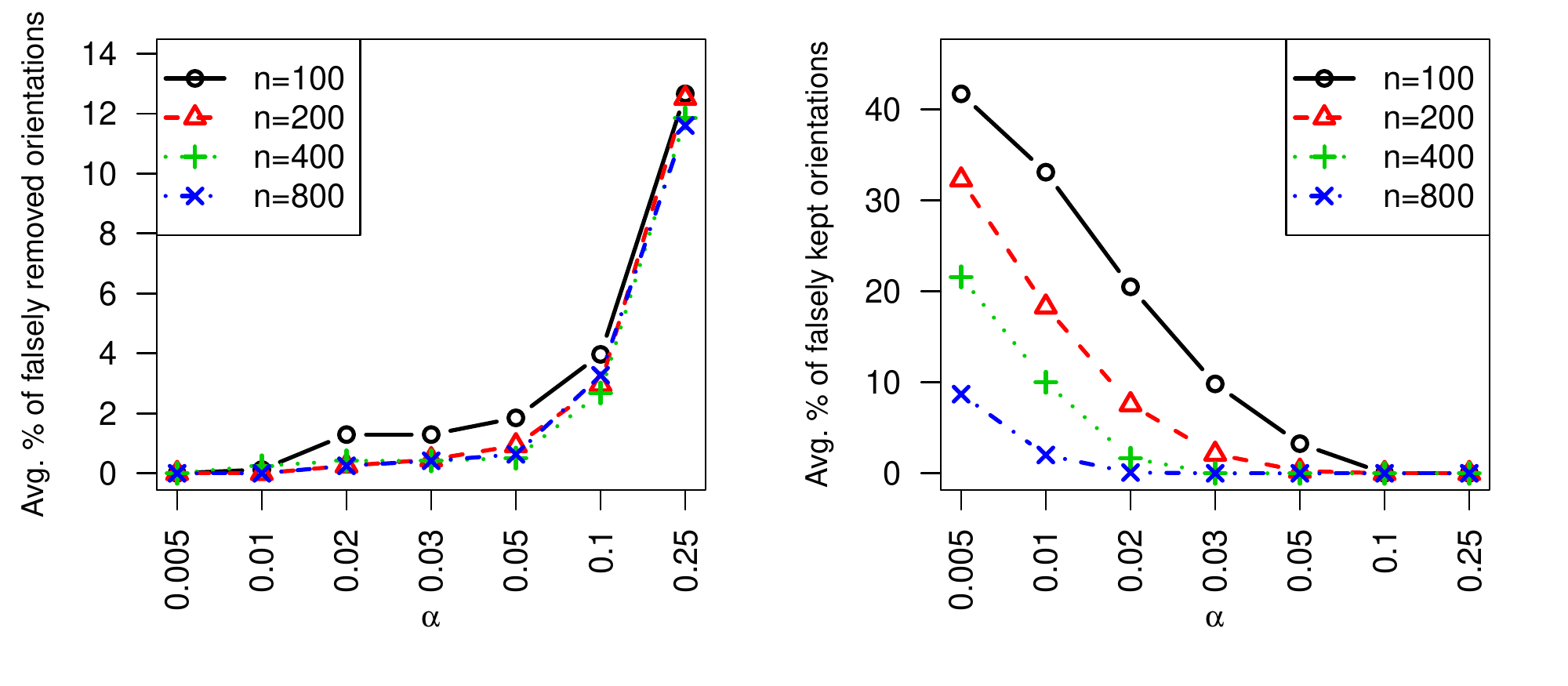}
\vspace{-0.4cm}
\caption{Performance of \texttt{computeGDPX} for varying sample sizes and values of $\alpha$ (x-axis) in dense DAGs for $\pl=0.2$ (top) and $\pl=0.8$ (bottom). Parameters: $p=10$, $\nrep=100$ and $\pc=6/9$ (expected number of edges: $30$).}
\label{M-fig:effectalpha_dense}
\end{figure}

\subsection{The dependence on $p$: low- and high-dimensional setting} \label{M-ssec:sim-p}
From the fact that \texttt{computeGDPX} only relies on local score computations, 
we expect that its performance does not strongly depend on the number of variables~$p$ as long as the neighborhood sizes in the DAGs (the node degrees) are similar for different values of $p$. We simulate $\nrep = 100$ random DAGs with $p=10$, $p=100$ and $p=1000$ nodes, respectively. Moreover, we set $\pc = 2/(p-1)$ which results in an expected number of $p$ edges and an expected node degree of $2$ for all settings. As demonstrated in Figure~\ref{M-fig:effectp}, the accuracy of \texttt{computeGDPX} with respect to varying values of $\alpha$ is barely affected by the number of variables $p$. In particular, \texttt{computeGDPX} exhibits a good performance even in high-dimensional settings with $p=1000$ and sample sizes in the hundreds. The same conclusions hold for $\pc=6/(p-1)$ with an expected node degree of $6$ (not shown).
\begin{figure}[h]
\centering
\includegraphics[width=0.75\textwidth]{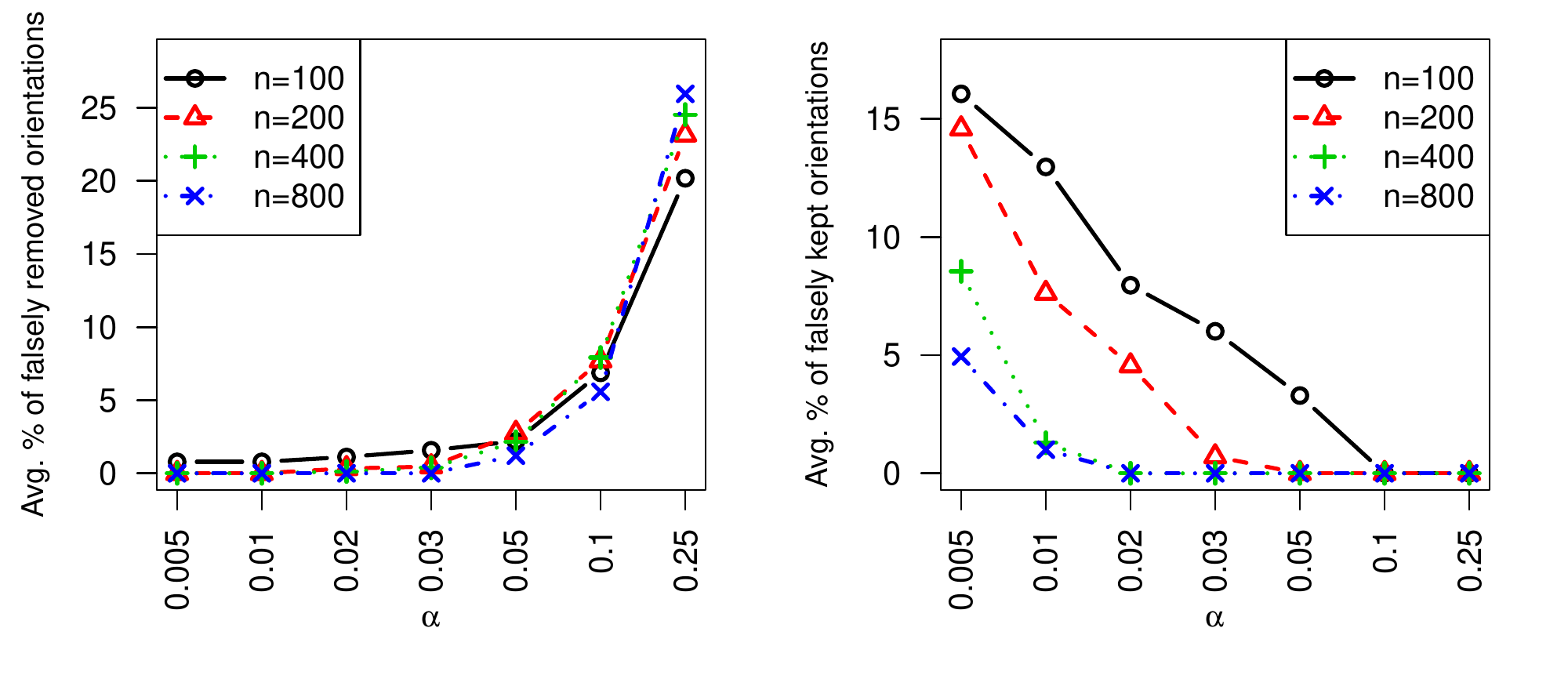} \\
\vspace{-0.4cm}
\includegraphics[width=0.75\textwidth]{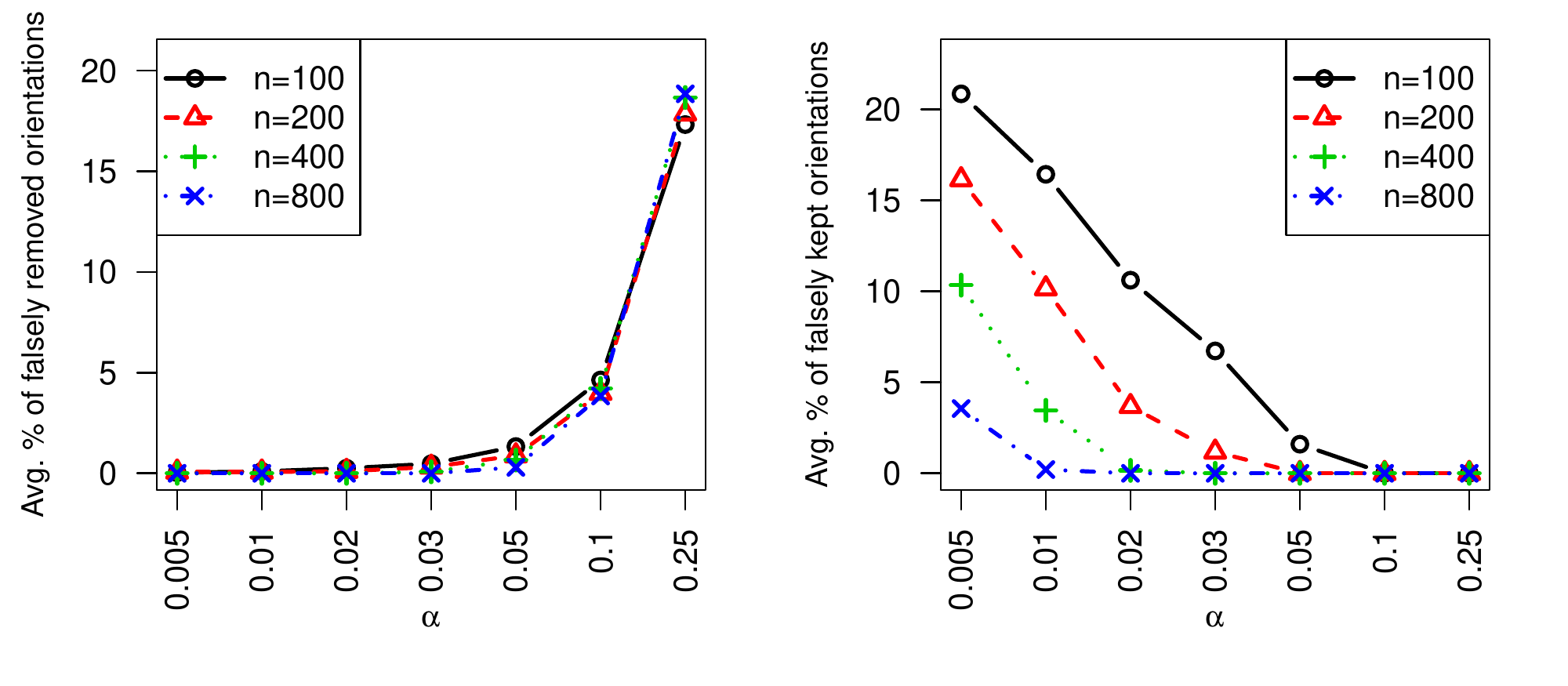} \\
\vspace{-0.4cm}
\includegraphics[width=0.75\textwidth]{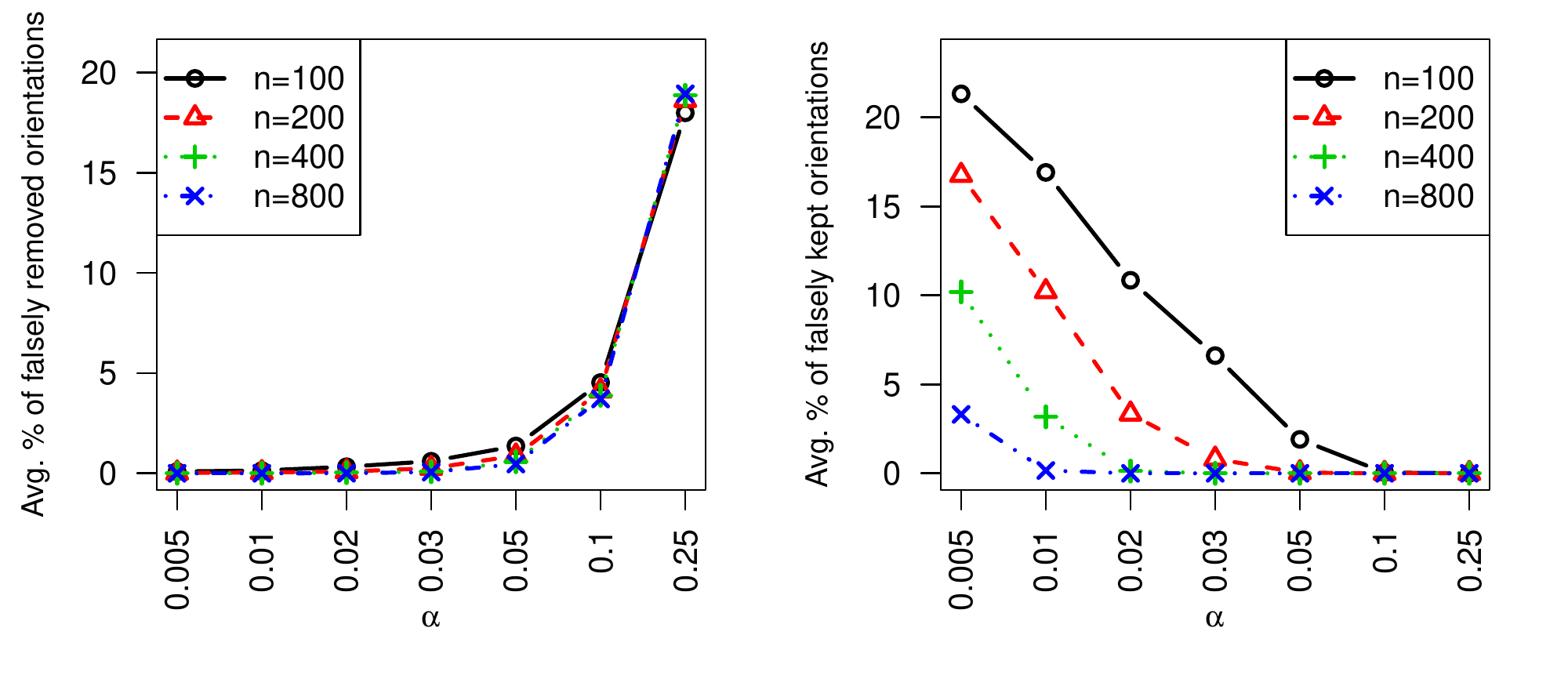}
\caption{Performance of \texttt{computeGDPX} for varying sample sizes and values of $\alpha$ (x-axis) for $p=10$ (top), $p=100$ (middle) and $p=1000$ (bottom). Parameters: $\pl=0.5$, $\nrep=100$ and $\pc=2/(p-1)$ (expected number of edges: $p$).}
\label{M-fig:effectp}
\end{figure}

\subsection{Computation time} \label{M-ssec:sim-time}
Finally, we analyze the computation time of \texttt{computeGDPX} depending on the number of variables $p$ and sparsity $\pc$. We examine two scenarios: (i) most of the functions in the PLSEM are nonlinear ($\pl=0.2$) and (ii) the worst-case scenario (w.r.t. computation time) where all the functions in the PLSEM are linear ($\pl = 1$) and $\DPX$ is equal to the Markov equivalence class ($G_{\DPX}$ equals the CPDAG). For all combinations of $p \in \{10,20,50,100,250,500,1000,2000,5000\}$ and $\pc \in \{2/(p-1), 8/(p-1)\}$ and for both scenarios (i) and (ii), we measure the time consumption of \texttt{computeGDPX} for $n=400$ and $\alpha = 0.05$.
In the scenario where all the functions are linear, we additionally compare it to \texttt{dag2cpdag} in the R-package \texttt{pcalg}, which constructs the CPDAG based on iterative application of R1-R3 in Figure~\ref{M-fig:orientationRules}. The median CPU times are shown in Table~\ref{M-tab:CPUtimes}. 
\texttt{computeGDPX} is able to estimate $G_{\DPX}$ in less than a minute even if the number of variables is in the thousands. In general, the speed of our implementation heavily depends on the sparsity of the DAGs. This can be seen from the case with $p=10$ and expected number of edges $40$. In this setting the DAGs are almost fully connected. This in turn implies that not many of the edges are fixed due to $v$-structures and a lot of score-based tests have to be performed. On the other hand, if the underlying DAGs are sparse, we observe that \texttt{computeGDPX} even outperforms \texttt{dag2cpdag} with respect to computation time if the number of variables is large. Note that this only holds for sparse DAGs. In general, \texttt{dag2cpdag} is much faster than our implementation (not shown).

\begin{table}[h]
\centering
\caption{Median CPU times [s] for \texttt{computeGDPX} and for \texttt{dag2cpdag} that iteratively applies R1 to R3 in Figure~\ref{M-fig:orientationRules}. $\nrep = 100$ repetitions for $\pl=0.2$ and $\nrep=20$ repetitions for $\pl=1$.}
\label{M-tab:CPUtimes}
\begin{tabular}{l|l|l|l|l|l|l|}
\cline{2-7}
& \multicolumn{2}{c|}{$\pl=0.2$} & \multicolumn{4}{c|}{$\pl=1$} \\ \cline{2-7} 
& \multicolumn{2}{c|}{\texttt{computeGDPX}} & \multicolumn{2}{c|}{\texttt{computeGDPX}} & \multicolumn{2}{c|}{\texttt{dag2cpdag}} \\ \hline 
\multicolumn{1}{|l|}{$\EE[|\text{edges}|]$} & $p$ & $4p$ & $p$ & $4p$ & $p$ & $4p$ \\ \hline
\multicolumn{1}{|l|}{$p=10$} & 0.092 & 0.785 & 0.157 & 1.101 & 0.007 & 0.005  \\ \hline
\multicolumn{1}{|l|}{$p=20$} & 0.150 & 0.105 & 0.174 & 0.162 & 0.006 & 0.006  \\ \hline
\multicolumn{1}{|l|}{$p=50$} & 0.300 & 0.164 & 0.332 & 0.223 & 0.008 & 0.009  \\ \hline
\multicolumn{1}{|l|}{$p=100$} & 0.604 & 0.281 & 0.665 & 0.325 & 0.014 & 0.016 \\ \hline
\multicolumn{1}{|l|}{$p=250$} & 1.446 & 0.630 & 1.740 & 0.717 & 0.072 & 0.087 \\ \hline
\multicolumn{1}{|l|}{$p=500$} & 2.705 & 1.253 & 3.486 & 1.523 & 0.395 & 0.599 \\ \hline
\multicolumn{1}{|l|}{$p=1000$} & 5.616 & 2.513 & 6.603 & 2.974 & 3.464 & 4.231 \\ \hline
\multicolumn{1}{|l|}{$p=2000$} & 11.504 & 5.380 & 13.493 & 6.331 & 25.463 & 31.591 \\ \hline
\multicolumn{1}{|l|}{$p=5000$} & 29.226 & 16.276 & 35.094 & 18.462 & 400.324 & 591.574 \\ \hline
\end{tabular}
\end{table}

\section{Conclusion} \label{M-sec:conclus}
We comprehensively characterized the identifiability of partially linear structural equation models with Gaussian noise (PLSEMs) from various perspectives. First, we proved that under faithfulness we obtain graphical and transformational characterizations of distribution equivalent DAGs similar to well-known characterizations of Markov equivalence classes of DAGs. More generally, we demonstrated that reinterpreting PLSEMs as PLSEM-functions leads to an interesting geometric characterization of all PLSEMs that generate the same distribution $\PX$, as they can all be expressed as constant rotations of each other. Therefrom we derived a precise condition how PLSEM-functions (and hence also how single nonlinear additive components in PLSEMs) restrict the set of potential causal orderings of the variables and showed how it can be leveraged to conclude about the causal relations of specific pairs of variables under mild additional assumptions.  We also provided some robustness results when the noise terms are in the neighborhood of Gaussian distributions.
The theoretical results were complemented with an efficient algorithm that finds all equivalent DAGs to a given DAG or PLSEM. We proved its high-dimensional consistency and evaluated its performance on simulated data.

From an application perspective, the algorithms  \texttt{listAllDAGsPLSEM}  and \texttt{computeGDPX} can serve two purposes. First, they can be used in conjunction with any causal structure learning procedure in the DAG space. This has been proposed in~\citep{castelo2003} and it can also be used in the context of PLSEMs. In comparison to the Markov equivalence class, the algorithms can potentially identify additional directed edges. 
In addition, the proposed methods can play an important role for the output of the CAM algorithm~\citep{pbjoja13} (with pruning).  In particular if some of the edge functions are close to linear or the sample size is low, the CAM algorithm will output one DAG even though there might be many DAGs with similar scores. In that scenario, the proposed algorithms provide a simple and important criterion to assess the reliability of  oriented edges.

More broadly speaking, our characterizations of PLSEMs (and corresponding DAGs) that generate the same distribution $\PX$ are crucial for further algorithmic developments in structure learning. For example, as mentioned before, in the spirit of~\citep{castelo2003}, or also for Monte Carlo sampling in Bayesian settings, see a related discussion in~\citep[Section~1]{andersson1997}. 







\section*{Acknowledgements}
The authors thank Emilija Perkovi\'c and  Jonas Peters  for fruitful discussions. They also thank some anonymous reviewers, an Associate Editor and the Editor for constructive comments.

\section{Appendix}

This supplement contains detailed specifications and proofs of our main theorems. The order of the presentation matches the one in the main paper. 
Figure~\ref{S-figproof} gives an overview of the dependency structure of the different theorems.

\begin{figure}[htb!]
    \centering
      \begin{tikzpicture}[>=stealth',shorten >=1pt,node distance=2.5cm, main node/.style={minimum size=0.4cm}]
      [>=stealth',shorten >=1pt,node distance=3cm,initial/.style    ={}]

    \node[main node] (L54)  {Theorem~\ref{M-thm:covered-reversals}};
    \node[main node,yshift=1cm] (L55) at (L54)  {Lemma~\ref{S-thm:nonlin-desc}};
    \node[main node] (L63) [right=0.5cm of L54]  {Theorem~\ref{M-thm:pdag}}; 
\node[main node,yshift=1.75cm] (L70) at (L63)  {Theorem~\ref{M-sec:nonl-faithf}};
\node[main node,yshift=1cm] (L71) at (L63)  {Theorem~\ref{M-sec:nonl-faithf-2}};

    \node[main node]            (L56)  [left of= L55]   {Theorem~\ref{M-thm:char-via-order}}; 
    \node[main node]            (L59)  [left of=L54]   {Lemma~\ref{S-lem:trans-char-via-V}};
    \node[main node,yshift=-2cm]            (L60)  at  (L56)  {Lemma~\ref{S-le:uniquePLSEM}};


 \node[main node,yshift=1.75cm,xshift=0.5cm] (L68) at (L56) {Theorem~\ref{M-thm:lower-loglik}};
 \node[main node,yshift=2.75cm,xshift=0.5cm] (L64) at (L56) {Theorem~\ref{M-thm:close-loglik}};

    \node[main node,yshift=1cm]            (L57)  [left=0.5cm of L56]   {Lemma~\ref{S-thm:func-char}}; 


\node[main node,yshift=1cm]            (L61)  [left of=L57]   {Lemma~\ref{S-le:mainlemma}};
\node[main node,yshift=-1cm]            (L62)  at (L57)   {Lemma~\ref{S-lem:constant-equivalence}};
\node[main node,yshift=-1cm]  (L65)  at (L61)   {Proposition~\ref{S-prop:PLSEM2}};
\node[main node,yshift=-2cm]  (L66)  at (L61)   {Lemma~\ref{S-lemma:decomposability}};
\node[main node, yshift=1cm] (L67) at (L57)  {Theorem~\ref{M-thm:funct-char}};

\draw[->]  (L57) edge (L68);
\draw[->]  (L57) edge (L64);

 \draw[->] (L61) edge    (L57);
 \draw[->] (L62) edge    (L57);

 \draw[->] (L57) edge    (L56);
 \draw[->] (L57) edge    (L67);

 \draw[->] (L56) edge    (L55);
 \draw[->] (L56) edge    (L54);
 \draw[->] (L59) edge    (L54);
 \draw[->] (L60) edge    (L54);

 \draw[->] (L54) edge    (L63);
 
 \draw[->] (L55) edge    (L70);
 \draw[->] (L55) edge    (L71);

 \draw[->] (L66) edge    (L65);
 \draw[->] (L65) edge    (L57);
 \draw[->] (L65) edge    (L61);
 \draw[->] (L65) edge    (L62);

 \tikzset{black dotted/.style={draw=black, line width=1pt,
                               dash pattern=on 1pt off 1pt on 1pt off 1pt,
                                inner xsep=1mm, inner ysep=1mm, rectangle, rounded corners}};
 \node (first dotted box) [black dotted, 
                            fit = (L54) (L59) (L60)] {};
 \node (second dotted box) [black dotted, 
                            fit = (L56)] {};
 \node (third dotted box) [black dotted, 
                            fit = (L57)  (L67) (L61) (L62)  (L65) (L66)] {};
 \node (fourth dotted box) [black dotted, 
                            fit = (L63)] {};
 \node (fifth dotted box) [black dotted, 
                            fit = (L55) (L70) (L71)] {};     

\node (sixth dotted box) [black dotted, 
                             fit = (L64) (L68)]{};

 \node at (third dotted box.south) [below, inner sep=2mm] {$\substack{\textit{Functional characterization} \\ \textit{(Section~\ref{S-sec:prf-functional})}}$};
 \node at (first dotted box.south) [below, inner sep=2mm] {$\substack{\textit{Transformational characterization} \\ \textit{(Section~\ref{S-sec:prf-covered})}}$};
  \node at (second dotted box.north) [above, inner sep=1mm] {$\substack{\textit{Causal ordering} \\ \textit{characterization} \\ \textit{(Section~\ref{S-sec:prf-ordering})}}$};
  \node at (fourth dotted box.south) [below, inner sep=1mm] {$\substack{\textit{Graphical} \\ \textit{representation} \\ \textit{(Section~\ref{S-sec:prf-graphical})}}$};
  \node at (fifth dotted box.north) [above, inner sep=1mm] {$\substack{\textit{Nonlinearity \& faithfulness} \\ \textit{(Section~\ref{S-sec:prf-interplay})}}$};
  \node at (sixth dotted box.north) [above, inner sep=1mm] {$\substack{\textit{Model misspecification} \\ \textit{(Section~\ref{S-sec:proofs-model-missp})}}$};
    \end{tikzpicture}
    \vspace{-0.5cm}
   \caption{Proof structure for the characterization results in Section~\ref{M-sec:char}. The proofs for Section~\ref{M-sec:estim} are given in Section~\ref{S-ssec:prf-consistency} (not depicted). }
    \label{S-figproof}
\end{figure}
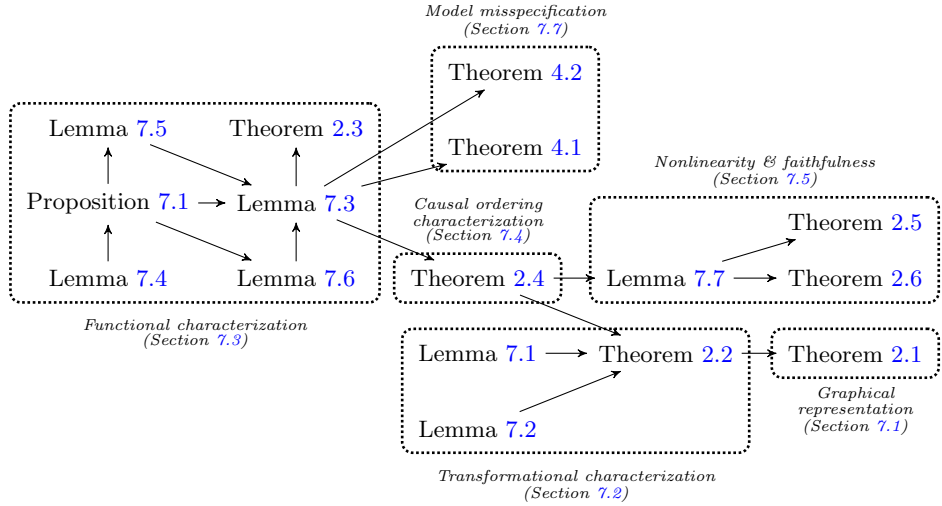

\subsection{Proof of the graphical characterization (Theorem~\ref{M-thm:pdag})} \label{S-sec:prf-graphical}

\begin{proof}
By definition, $\DPX$ is a subset of the set of all consistent DAG extensions of $G_{\DPX}$. It remains to show, that the set of all consistent DAG extensions of $G_{\DPX}$ is a subset of $\DPX$. Suppose there is a consistent DAG extension $\tilde{D}$ of $G_{\DPX}$ such that $\tilde{D} \not\in \DPX$. Let $D \in \DPX$. As both, $D$ and $\tilde{D}$ are consistent DAG extensions of $G_{\DPX}$, they have the same skeleton and $v$-structures and are Markov equivalent. 
Hence, there exists a sequence of distinct covered edge reversals transforming $D$ into $\tilde{D}$~\citep[Theorem~2]{chickering95}. 
Let us denote the sequence of traversed DAGs by $D = D_1,...,D_m = \tilde{D}$.
If all covered edge reversals are linear, $\tilde{D} \in \DPX$ by Theorem~\ref{M-thm:covered-reversals}~(a), which contradicts the assumption. Therefore, there is at least one covered nonlinear edge reversal in this sequence. Without loss of generality, for $1 \leq r \leq m-1$, let the edge reversal of $i \rightarrow j$ to $i \leftarrow j$ between $D_r$ and $D_{r+1}$ be the first covered nonlinear edge reversal in the above sequence. First note that as the sequence of covered edge reversals is distinct, $i \rightarrow j$ in $D$ and $i \leftarrow j$ in $\tilde{D}$.  
Moreover, as $D_r$ is obtained from $D$ by a sequence of covered linear edge reversals, $D_r \in \DPX$ by Theorem~\ref{M-thm:covered-reversals}~(a). Again, by Theorem~\ref{M-thm:covered-reversals}~(a), as $D_r \in \DPX$ and $i \rightarrow j$ is covered and nonlinear in $D_r$, 
$i \rightarrow j$ for all DAGs $D' \in \DPX$. 
Therefore, by Definition~\ref{M-def:PDAGrepr}, $i \rightarrow j$ in $G_{\DPX}$ which contradicts the assumption that $\tilde{D}$ is a consistent DAG extension of $G_{\DPX}$.
\end{proof}

\subsection{Proof of transformational characterization (Theorem~\ref{M-thm:covered-reversals})} \label{S-sec:prf-covered}

\begin{proof}

\textit{Part (a):}
By Lemma~\ref{S-le:uniquePLSEM} there exists a unique PLSEM with DAG $D$ that generates $\PX$. Let $\F$ denote the function that corresponds to this PLSEM as defined in Section~\ref{M-ssec:funcchar}. Without loss of generality let us assume that $\D\F$ is lower triangular. Furthermore, as $i \rightarrow j$ is covered in $D$, no other child of $i$ is an ancestor of $j$ and we can assume that $j=i+1$. The differential $\D\F$ is of the form 
\begin{align*}
  \begin{pmatrix}
    \mbox{Var}(\eps_{1})^{-1/2} & 0 & \ldots & \ldots & \ldots   & 0 \\
    \partial_{1} \F_{2}           & \ddots & \ddots & \ddots & \ddots   & \vdots \\ 
  \vdots & \ddots &  \mbox{Var}(\eps_{i})^{-1/2} & 0 & \ddots   & \vdots \\
  \vdots & \ddots & \partial_{i} \F_{i+1} &  \mbox{Var}(\eps_{i+1})^{-1/2}   & \ddots & \vdots \\
\vdots & \ddots &\ddots &\ddots &\ddots &\vdots \\
\vdots & \ddots &\ddots &\ddots &\ddots  & 0 \\
\partial_{1} \F_{p} & \ldots & \ldots    & \ldots & \ldots  & \mbox{Var}(\eps_{p})^{-1/2}
  \end{pmatrix}.
\end{align*}
Let us write $ v = (\D \F)^{-1} \partial_{i}^{2} \F$, i.e. $ \partial_{i}^{2} \F = \D \F v$.  As $\D\F$ is lower triangular with $\left(\Var(\eps_{i})^{-1/2}\right)_{i=1,\ldots,p}$ on the diagonal we get $v_{1},\ldots,v_{i} =0$ and $v_{i+1} =  \Var(\eps_{i+1})^{1/2} \partial_{i}^{2} \F_{i+1}$. Hence,
\begin{equation*}
  e_{i+1}^{t} (\D\F)^{-1} \partial_{i}^{2} \F =  \Var(\eps_{i+1})^{1/2} \partial_{i}^{2} \F_{i+1}.
\end{equation*}
Now recall that by definition of $\F$,
\begin{equation*}
  \partial_{i}^{2} \F_{i+1} = - \frac{1}{\mbox{Var}(\eps_{i+1})^{1/2}} \partial_{i}^{2} f_{i+1,i}(x_{i}).
\end{equation*}
By combining these two equations,
\begin{equation}\label{S-eq:7}
   e_{i+1}^{t} (\D\F)^{-1} \partial_{i}^{2} \F = - \partial_{i}^{2} f_{i+1,i}(x_{i}).
\end{equation}
By Lemma~\ref{S-lem:trans-char-via-V}, the edge can be reversed if and only if $(i , i+1) \not \in \mathcal{V}$, which by definition of $\mathcal{V}$ is the case if and only if
\begin{align*}
   e_{i+1}^{t} (\D\F)^{-1} \partial_{i}^{2} \F \equiv 0.
\end{align*}
By equation~\eqref{S-eq:7} this is the case if and only if $ \partial_{i}^{2} f_{i+1,i}(x_{i}) \equiv 0$. Hence the edge can be reversed if and only if the edge is linear. This concludes the proof of the ``if and only if'' statement. \\

If the edge $i \rightarrow i+1$ is nonlinear, we can argue analogously as above that  $(i , i+1) \in \mathcal{V}$. By 
Theorem~\ref{M-thm:char-via-order}, all causal orderings of PLSEMs that generate $\PX$ satisfy $\sigma(i) < \sigma(i+1)$. As, by definition, $\PX$ is faithful to all DAGs in $\DPX$, they all have the same skeleton. Hence, $i \rightarrow i+1$ in all DAGs in $\DPX$. \\ 

\textit{Part (b):}
As $D,D' \in \DPX$, $D'$ is Markov equivalent to $D$. Hence, there
exists a sequence of distinct covered edge reversals transforming $D$
into $D'$~\citep[Theorem~2]{chickering95}. Let us denote the sequence
of traversed DAGs by $D=D_{1},\ldots,D_{m}=D'$. By part~(a), we are
done if we can show that each DAG $D_r$ in this sequence lies in $\DPX$. 
We prove this by induction. So let us assume $D_{r} \in \DPX$ with $r < m$. Then $D_{r+1}$ only differs from $D_{r}$ by the reversal of a covered edge, w.l.o.g. $i \rightarrow j$ in $D_r$ and $j \rightarrow i$ in $D_{r+1}$. By construction, all covered edge reversals are distinct, hence, $j \rightarrow i$ in $D'$. Define the set $\mathcal{V}$ as in Theorem~\ref{M-thm:char-via-order}. As $D, D' \in  \DPX$, by Theorem~\ref{M-thm:char-via-order}, $(i , j) \not \in \mathcal{V}$. Hence by Lemma~\ref{S-lem:trans-char-via-V} we immediately get that $D_{r+1} \in \DPX$. Moreover, by Theorem~\ref{M-thm:covered-reversals}~(a), $i \rightarrow j$ is linear. This concludes the proof.
\end{proof}

\begin{lemm}\label{S-lem:trans-char-via-V}
Let $D \in \DPX$. Let $i \rightarrow j$ be a covered edge in $D$. Let $D'$ be a DAG that differs from $D$ only by reversing $i \rightarrow j$. Let $\F$ be a PLSEM-function of $\PX$ and define $\mathcal{V}$ as in equation~\eqref{M-eq:defV}. Then $D' \in \DPX$ if and only if $(i, j) \not \in \mathcal{V}$.
\end{lemm}

\begin{proof}''$\Rightarrow$'': Let $D' \in \DPX$ and  $(i , j) \in \mathcal{V}$. Consider a causal ordering $\sigma$ of $D'$. As $j \rightarrow i$ in $D'$, $\sigma(j) < \sigma(i)$. By Theorem~\ref{M-thm:char-via-order} this leads to a contradiction. Hence if $D' \in \DPX$, then $(i , j) \not \in \mathcal{V}$.\\

``$\Leftarrow$'':  Let  $(i , j) \not \in \mathcal{V}$. Let $\sigma$ be a causal ordering of $D$. As $i \rightarrow j$ is covered in $D$, no other child of $i$ is an ancestor of $j$ in $D$. Hence without loss of generality we can assume that $\sigma(j) = \sigma(i)+1$. Define $\sigma'$ as the permutation with $i$ and $j$ switched, i.e.
\begin{equation*}
  \sigma'(k) = \begin{cases}
  \sigma(k) & k \not \in \{i,j\} \\
  \sigma(j) & k = i \\
  \sigma(i) & k = j.
\end{cases}
\end{equation*}
Note that as  $\sigma(j) = \sigma(i)+1$, the causal orderings of other pairs of variables are unaffected, i.e. 
\begin{equation}\label{S-eq:15}
\sigma(k) < \sigma(l) \iff \sigma'(k) < \sigma'(l) \text{ for all $k,l$ with $\{k,l \} \neq \{i,j\}$. }
\end{equation}
As $\sigma$ is a causal ordering of a PLSEM that generates $\PX$, by Theorem~\ref{M-thm:char-via-order},
\begin{equation}\label{S-eq:17}
\sigma(k) < \sigma(l) \text{ for all $(k , l) \in \mathcal{V}$.}
\end{equation} 
We want to show that the same holds for $\sigma'$. Let $(k , l) \in \mathcal{V}$. As $(i , j)\not \in \mathcal{V}$,  $(k,l) \neq (i,j)$. Hence, by equations~\eqref{S-eq:15} and~\eqref{S-eq:17}, $\sigma'(k) < \sigma'(l)$. This proves~that
\begin{equation*}
\sigma'(k) < \sigma'(l) \text{ for all $(k , l) \in \mathcal{V}$.}
\end{equation*} 
By Theorem~\ref{M-thm:char-via-order}, $\sigma'$ is a causal ordering of
a PLSEM that generates $\PX$. Consider the DAG $\tilde D$ of this
PLSEM. Then $\PX$ is Markov with respect to $\tilde D$ and by
Proposition 17 of \cite{petersetal13}, $\PX$ satisfies causal
minimality with respect to $\tilde D$. By
\citep[Lemma~1]{chickering95}, $\PX$ is Markov and faithful with
respect to $D'$ and we know that $\sigma'$ is a causal ordering of both $\tilde D$ and $D'$. Now we want to show that this implies $\tilde D = D'$.
Without loss of generality assume $\sigma' = \text{Id}$. First, we want to show that  $ \text{pa}_{\tilde D}(l) \supseteq \text{pa}_{D'}(l)$ for all $l$. Fix $l$. Consider the parental set $\text{pa}_{\tilde D}(l)$ of $l$ in $\tilde D$ and let $k$ be a parent of $l$ in $ D'$ but not in $\tilde D$. As $\sigma' = \text{Id}$ is a causal ordering of $D'$, $k < l$, and as $\sigma' = \text{Id}$ is a causal ordering of $\tilde D$ as well, $k$ is not a descendant of~$l$ in $ \tilde D$. As $\PX$ is Markov with respect to $ \tilde D$,
\begin{equation*}
  X_{l} \independent X_k | X_{\text{pa}_{\tilde D}(l)}.
\end{equation*}
 Hence, as $\PX$ is faithful to $D'$, $l$ and $k$ are d-separated by $\text{pa}_{\tilde D}(l)$ in $D'$. But $k$ is a parent of $l$ in $D'$, contradiction. Hence $ \text{pa}_{\tilde D}(l) \supseteq \text{pa}_{D'}(l)$ for all $l$. $\PX$ satisfies causal minimality with respect to $\tilde D$, hence $ \text{pa}_{\tilde D}(l) =\text{pa}_{D'}(l)$ for all $l$. This proves $\tilde D = D'$. Therefore, there exists a PLSEM with DAG $D'$ that generates~$\PX$ and $\PX$ is faithful with respect to $D'$. By definition, $D' \in \DPX$.
\end{proof}

\begin{lemm}\label{S-le:uniquePLSEM}
  Let $\PX$ be generated by a PLSEM. Let $D \in \DPX$. Then there exists a unique PLSEM  (unique set of intercepts, edge functions and Gaussian error variances) with DAG $D$ that generates $\PX$.
\end{lemm}

\begin{proof} 
  By definition of the distribution equivalence class $\DPX$ there exists such a PLSEM with DAG $D$ that generates $\PX$. Now we will show that this PLSEM is unique. Consider another PLSEM with DAG $D$ that generates $\PX$. For a given node $j$ we have
\begin{align*}
 \mu_j + \sum_{i \in \pa_D(j)} f_{j,i}(X_i) &= \mathbb{E}[X_{j} | X_{\pa_{D}(j)}] =  \tilde \mu_j + \sum\limits_{i \in \pa_D(j)} \tilde f_{j,i}(X_i).
\end{align*}
By definition of PLSEMs, the expectations of the $f_{j,i}(X_{i})$ and $\tilde f_{j,i}(X_{i})$ are zero, hence we have $\mu_{j} = \tilde \mu_{j}$. As $\sigma_{k} >0$ for all $k \in \{1,\ldots, p\}$, the density of $X$ is positive on $\mathbb{R}^{p}$. Recall that by definition,  $f_{j,i}$ and $\tilde f_{j,i}$ are continuous.  Hence, for all $x \in \mathbb{R}^{p}$,
\begin{align*}
\sum_{i \in \pa_D(j)} f_{j,i}(x_i) &= \lim_{\delta \rightarrow 0}  E \left[\sum_{i \in \pa_D(j)} f_{j,i}(X_i) | X \in B_{\delta}(x) \right] \\
&= \lim_{\delta \rightarrow 0}  E \left[\sum_{i \in \pa_D(j)} \tilde f_{j,i}(X_i) | X \in B_{\delta}(x) \right] = \sum_{i \in \pa_D(j)} \tilde f_{j,i}(x_i),
\end{align*}
where $B_{\delta}(x)$ denotes the closed ball around $x$ with radius $\delta$. Take an arbitrary $i \in \pa_{D}(j)$. By taking the derivative with respect to $x_{i}$ on both sides of the equation we obtain
\begin{align*}
  f_{j,i}'(x_{i}) = \tilde f_{j,i}'(x_{i}).
\end{align*}
Hence there exists a constant $c$ such that
\begin{align*}
  f_{j,i}(x_{i}) = c + \tilde f_{j,i}(x_{i}).
\end{align*}
By definition of PLSEMs, we have $ \mathbb{E} [f_{j,i}(X_{i})] = 0$ and  $ \mathbb{E} [\tilde f_{j,i}(X_{i})] = 0$. Hence, $c=0$ and $f_{j,i} = \tilde f_{j,i}$ for all $i \in \pa_{D}(j)$. We just showed that $\mu_{j} = \tilde \mu_{j}$ and $f_{j,i} = \tilde f_{j,i}$. It remains to show that $\sigma_{j} = \tilde \sigma_{j}$:
\begin{align*}
  \sigma_{j}^{2} = \Var \! \left(\! X_{j} -\mu_j - \!\!\! \sum_{i \in \pa_D(j)} \!\!\! f_{j,i}(X_i)\right) =  \Var\!\left(\! X_{j} -\tilde \mu_j - \!\!\! \sum_{i \in \pa_D(j)} \!\!\! \tilde{f}_{j,i}(X_i)\right) = \tilde \sigma_{j}^{2}.
\end{align*}
Hence, the intercepts, edge functions and Gaussian error variances of both PLSEMs are equal, which concludes the proof.
\end{proof}

\subsection{Proof of functional characterization (Theorem~\ref{M-thm:funct-char})} \label{S-sec:prf-functional}

In the following, let $\PX$ be generated by a PLSEM.
\begin{defn}[PLSEM-function]\label{S-def:PLSEM-function}
  We call $\F : \mathbb{R}^{p} \rightarrow \mathbb{R}^{p}$ a PLSEM-function of $\PX$ if there exists a PLSEM that generates $\PX$ such that $\F$ can be written as in equation~\eqref{M-eq:4}.
\end{defn}
\begin{rem}\label{S-rem:plsem-function}
  For a PLSEM-function $\F$ of $\PX$ we can retrieve the unique corresponding PLSEM (i.e. the unique DAG, unique set of intercepts, edge functions and Gaussian error
variances) through  equations~\eqref{M-eq:23}-\eqref{M-eq:29}.
\end{rem}

Recall generalized PLSEMs, in particular Definition~\eqref{M-defn:gPLSEM}.
\begin{defn}[generalized PLSEM-function]\label{S-def:generalized-plsem-function}
  We call $\F : \mathbb{R}^{p} \rightarrow \mathbb{R}^{p}$ a generalized PLSEM-function of $\mathring{\mathbb{P}}$ if there exists a generalized PLSEM that generates $\mathring{\mathbb{P}}$ such that $\F$ can be written as in  equation~\eqref{M-eq:gPLSEM}.
\end{defn}

\begin{prop}\label{S-prop:PLSEM2} 
A function  $\F : \mathbb{R}^{p} \rightarrow \mathbb{R}^{p}$ is a PLSEM-function of $\PX$ if and only if
\begin{enumerate}
\item $\F$ is twice continuously differentiable,
\item $\partial_{k} \partial_{l} \F \equiv 0$ for all $k \neq l$,
\item there exists a permutation $\sigma$ such that $(\D\F_{i \sigma^{-1}(j)})_{ij}$ is lower triangular with constant positive entries on the diagonal. 
\item If $X \sim \PX$, then $\F(X) \sim \mathcal{N} (0 ,\text{Id}_{p})$.
\end{enumerate}
Furthermore, a function $F: \mathbb{R}^{p} \rightarrow \mathbb{R}^{p}$ is a  generalized PLSEM-function of $\mathring{ \mathbb{P}}$ if and only if $(1)-(3)$ and $(4')$ holds.
\begin{enumerate}
  \item[4']   If $X \sim \mathring{\PX}$, then $\F(X)_{i}$, $i=1,\ldots,p$ are independent and centered with  variance one.
\end{enumerate}

\end{prop}
\begin{rem}
  We call a permutation $\sigma$ that satisfies (3) a causal ordering of the (generalized) PLSEM-function $\F$. Define the directed graph $D$, the functions $f_{j,i}$, $\sigma_{j}^{2} = \mathrm{Var}(\eps_{j})$ and $\mu_{j}$  through equations~\eqref{M-eq:23} -- \eqref{M-eq:29}. The first condition reflects that the functions $f_{j,i}$ are twice continuously differentiable. The second condition reflects that the functions $f_{j,i}$ depend on $x_{i}$ only. The third condition ensures that the directed graph $D$ is acyclic and that the variances of all $\eps_{j}$ are strictly positive. The last condition ensures that the distribution generated by this PLSEM or generalized PLSEM is $\PX$ or $\mathring{\PX}$, respectively. 
\end{rem}

\begin{proof}
We only prove this result for PLSEM-functions. The proof for generalized PLSEM-functions is analogous. 
``$\Rightarrow$'' By definition of a PLSEM and equation~\eqref{M-eq:4}. \\
``$\Leftarrow$'': Without loss of generality let us assume that the indices are ordered such that $\sigma = \text{Id}$, hence by (3), $\D \F$ is lower triangular with constant positive entries on the diagonal. Let $Z \sim \mathcal{N}(0,\text{Id}_{p})$ and define $X := \F^{-1}(Z)$. Using (4), we obtain $ X \sim \PX$.  Use (1) and (2) and Lemma~\ref{S-lemma:decomposability} for each component of $\F_{j}$, i.e. decompose $\F_{j}(x) = \tilde \mu_{j} + \sum_{i} \tilde f_{j,i}(x_{i})$ with twice continuously differentiable functions $\tilde f_{j,i}$. Here we choose $\tilde \mu_{j}$ and the $\tilde f_{j,i}$ (i.e. the constants) such that $\mathbb{E}[\tilde f_{j,i}(X_{i})] \equiv 0$ for all $j \neq i$ and $ \tilde f_{j,j}$ such that $\tilde f_{j,j}(0)=0$. We define the parental sets $\pa(j) := \{ i \neq j : \tilde f_{j,i} \not \equiv 0\}$. As $\D \F$ is lower triangular, $\pa(j) \subseteq \{1,\ldots,j-1\}$, hence the directed graph $D$ defined by these parental sets is acyclic.
As  $\D \F$ has constant positive entries on the diagonal, $\partial_{j} \F_{j}$ is constant, and we can define the error variances $ \sigma_{j}^{2} := 1/\left(\partial_{j} \F_{j} \right)^{2} >0 $. Furthermore, define functions $f_{j,i}(x_{i}) := -\sigma_{j} \tilde f_{j,i}(x_{i}) $ that only depend on $x_{i}$ and the constants $\mu_{j} := -\sigma_{j} \tilde \mu_{j}$. To sum it up, we have the following relations:
\begin{align*}
  \F_{j}(x) = \frac{1}{\sigma_{j}} \left(  x_{j} - \mu_{j}  - \sum_{i \in \pa_{D}(j)} f_{j,i}(x_{i}) \right),
\end{align*}
with DAG $D$, $f_{j,i} \not \equiv 0$, $\mathbb{E} [f_{j,i}(X_{i})] = 0$ for all $i \in \pa_{D}(j)$. Using that $\F(X) = Z \sim \mathcal{N}(0,\text{Id}_{p})$,
\begin{align*}
  X_{j} = \mu_{j}+ \sum_{i \in \pa_{D}(j)} f_{j,i}(X_{i})   + \sigma_{j} Z_{j}
\end{align*}
By defining the Gaussian errors $\eps_{j} := \sigma_{j} Z_{j}$, it is immediate to see that $\sigma_{j},f_{j,i},D$ define a PLSEM that generates $\PX$.
\end{proof}


\begin{lemm}[Functional characterization]\label{S-thm:func-char}
  Let $\PX$ be generated by a PLSEM. Let $\F$ be a PLSEM-function of $\PX$. Let $\sigma$ be a permutation. Define  $\Pi_{i+1}^{\sigma}$ as the linear projection on the space  $\langle \partial_{\sigma^{-1}(i+1)} \F, \ldots, \partial_{\sigma^{-1}(p)} \F \rangle$ and $\Pi_{p+1}^{\sigma} := 0 \in \mathbb{R}^{p \times p}$. Let $\G : \mathbb{R}^{p} \rightarrow \mathbb{R}^{p} $. Then $\G$ is a generalized PLSEM-function of $\PX$ with causal ordering $\sigma$ if and only if 
 \begin{equation*}
  \left( \text{Id} - \Pi_{i+1}^{\sigma} \right) \partial_{\sigma^{-1}(i)}^{2} \F \equiv 0, \qquad i=1,\ldots,p,
 \end{equation*}
and
\begin{equation}\label{S-eq:14}
  	\G_{i} =  \left( \frac{(\mbox{Id}- \Pi_{i+1}^{\sigma}) \partial_{\sigma^{-1}(i)} \F}{ \| (\mbox{Id}- \Pi_{i+1}^{\sigma}) \partial_{\sigma^{-1}(i)} \F \|_{2}} \right)^{t} \F.
\end{equation} 
In that case, the matrices $ \Pi_{i+1}^{\sigma}$ and the vectors $(\mbox{Id}- \Pi_{i+1}^{\sigma}) \partial_{\sigma^{-1}(i)} \F$ are constant and $\G(X)\sim \mathcal{N}(0,\mathrm{Id}_{p})$, i.e. $G$ is a PLSEM-function.
\end{lemm}
\begin{rem}\label{S-rem:construct-PLSEM}
  Lemma~\ref{S-thm:func-char} tells us that every potential causal ordering satisfies  $ \left( \text{Id} - \Pi_{i+1}^{\sigma} \right) \partial_{\sigma^{-1}(i)}^{2} \F \equiv 0$, $i=1,\ldots,p$ and contains a concrete formula to compute the unique PLSEM-function for this given causal ordering. Furthermore, every causal ordering that satisfies that condition gives rise to a corresponding PLSEM-function by equation~\eqref{S-eq:14}. As $\F$ is a PLSEM-function, $\D \F$ is invertible. This implies $\| (\mbox{Id}- \Pi_{i+1}^{\sigma}) \partial_{\sigma^{-1}(i)} \F \|_{2} >0$ and hence equation~\eqref{S-eq:14} is well-defined. Given the PLSEM-function, we can retrieve the unique corresponding PLSEM (i.e. the unique DAG, unique set of intercepts, edge functions and Gaussian error
variances) through equations~\eqref{M-eq:23} -- \eqref{M-eq:29}.
\end{rem}
\begin{rem}\label{S-rem:proof-func}
If $\G$ is a (generalized) PLSEM-function, from this theorem it follows that the vectors
\begin{equation*}
  \left(\frac{(\mbox{Id}- \Pi_{i+1}^{\sigma}) \partial_{\sigma^{-1}(i)} \F}{ \| (\mbox{Id}- \Pi_{i+1}^{\sigma}) \partial_{\sigma^{-1}(i)} \F \|_{2}} \right)^{t} \quad i=1,\ldots,p,
\end{equation*}
are constant in $x$, have norm $1$ and are orthogonal for $i=1,\ldots,p$. Hence the row-wise concatenation of these vectors for $i=1,\ldots,p$ forms an orthogonal matrix $O$ and by equation~\eqref{S-eq:14}, $  \G = O \F$.
\end{rem}

\begin{proof} 
``$\Rightarrow$''. Let $\G$ be a generalized PLSEM-function of $\PX$ with causal ordering $\sigma$. Without loss of generality let $\sigma = \text{Id}$, i.e. without loss of generality we assume that $\D \G$  is lower triangular. We write $\Pi_{i+1}$ instead of $\Pi_{i+1}^{\sigma}$ for brevity. 
As both PLSEMs generate the same distribution and the probability density function of $\mathbb{P}$ is positive on $\mathbb{R}^{p}$ their log-densities are well-defined and agree for all $x \in \mathbb{R}^{p}$. Hence for all $x \in \mathbb{R}^{p}$,
\begin{equation*}
-\log \det \mathrm{D} G  - \sum_{i=1}^{p}  \log (g_{i} (G_{i}(x))) = - \log \det \mathrm{D} F  + \frac{p}{2} \log(2 \pi ) +\sum_{i=1}^{p} \frac{ F_{i}(x)^{2}}{2},
\end{equation*}
where $g_{i}$ denotes the probability density function of $G_{i}(X)$, $i=1,\ldots,p$. Note that $ \det \mathrm{D} F$ and $ \det \mathrm{D} G$ are constant as for both $F$ and $G$ there exists a permutation of the indices such that the differential is lower triangular with constant diagonal.\\
Define $h :\mathbb{R}^{p} \rightarrow \mathbb{R}^{p}$ by $h_{i}(x) :=  - \frac{g_{i}'(G_{i}(x))}{g_{i}(G_{i}(x))} $ for $i=1,\ldots,p$. As the density of $\mathbb{P}$ is positive on $\mathbb{R}^{p}$, the density of $g_{i}$ is positive on $\mathbb{R}$, i.e. $h$ is well-defined.
By differentiating on both sides,
\begin{equation*}
	h^{t} \D\G = \F^t \D\F.
\end{equation*}
We assumed without loss of generality that $\sigma = \text{Id}$, hence by Proposition~\ref{S-prop:PLSEM2} the differential $\D \G$ is lower triangular and the diagonal entries $c_i := \partial_i \G_i$ are positive. Hence we can recursively solve for $i=1,\ldots,p$ and obtain
\begin{equation}\label{S-eq:31}
	c_{i} h_{i}  =   \F^t \partial_i \F - \sum_{j>i}  h_j \partial_i \G_j.
\end{equation}
Using induction, we will show that 
\begin{equation}\label{S-eq:13}
c_{i} h_i  = \F^t(\mbox{Id}-\Pi_{i+1}) \partial_i \F,
\end{equation}
that $c_i = \| (\mbox{Id}- \Pi_{i+1}) \partial_i \F \|_{2}$, that the matrix $\Pi_{i+1}$ is constant, that the vectors $(\mbox{Id}- \Pi_{i+1}) \partial_i^{2} \F \equiv 0$  and that $h_{i} = G_{i}$ for $i=1,\ldots,p$. 
By using equation~\eqref{S-eq:31} we obtain equation~\eqref{S-eq:13} for $i=p$ and by Lemma \ref{S-le:mainlemma} we obtain $(\mbox{Id}- \Pi_{p+1}) \partial_p^{2} \F = \partial_p^{2} \F \equiv 0$. \\
Now we want to show that $h_{p} = G_{p}$. To this end, note that as $ ( \mbox{Id}- \Pi_{p+1}) \partial_{p}^{2} F \equiv 0$,
\begin{equation*}
  \partial_{p} h_{p} = \frac{ \partial_{p} F^{t} (\mbox{Id}- \Pi_{p+1}) \partial_{p} F}{c_{p}} \text{ is constant.} 
\end{equation*}
As $ \partial_{p} h_{p} =  \partial_{p}^{2} \left( -\log(g_{p} (G_{p}(x))) \right)/c_{p}$,
\begin{equation*}
   \partial_{p}^{2} \left( -\log g_{p} (G_{p}(x)) \right) \text{ is constant.}
\end{equation*} 
Using that $\partial_{p} G_{p}(x) =c_{p} \neq 0$ is constant, $z \mapsto \log(g_{p}(z))$ is a polynomial of degree two. As $g_{p}$ is the density of a centered distribution with variance one and this density is positive on $\mathbb{R}$, it is the density of a centered Gaussian distribution with variance one. For a centered Gaussian distribution with variance one, $ (- \log g_{p})'(z)= z$. Hence,  $h_{p}(x) = - \frac{g_{p}'(G_{p}(x))}{g_{p}(G_{p}(x))} = (- \log(g_{p}))' (G_{p}(x)) = G_{p}(x)$. This proves that $h_{p} = G_{p}$.\\
Hence $\partial_{p} h_{p} = \partial_{p} G_{p} = c_{p}$ and using equation~\eqref{S-eq:13} for $i=p$
\begin{equation*}
   	c_p^2 = c_p \partial_p h_p = (\partial_p  \F)^t \partial_p \F =\|(\mbox{Id} - \Pi_{p+1}) \partial_p \F \|_{2}^2.
\end{equation*}
Furthermore, as $(\mbox{Id}- \Pi_{p+1}) \partial_p^{2} \F  \equiv 0$, $(\mbox{Id}- \Pi_{p+1}) \partial_p \F $ is a constant vector and hence by definition the matrix $\Pi_{p}$ is constant.  This finishes the proof for $i=p$.

Now let us assume  $c_j h_j = \F^t(\mbox{Id}-\Pi_{j+1}) \partial_j \F$, $h_{j} = G_{j}$, $c_j = \| (\mbox{Id}- \Pi_{j+1}) \partial_j \F \|_{2}$, that the matrix $\Pi_{j}$ is constant and that the vectors $(\mbox{Id}- \Pi_{j+1}) \partial_j^{2} \F \equiv 0$  for all $p \ge j > i \ge 1$. We want to prove these statements for $j=i$. By using equation~\eqref{S-eq:31} and the induction assumptions we can rewrite~$c_i h_i$,
\begin{align*}
	c_i h_i &=  \F^t \partial_i \F - \sum_{j>i}  h_j \partial_i \G_j \\ 
        &= \F^t \partial_i \F - \sum_{j>i}  \G_j \partial_i \G_j  \\
	&= \F^t \partial_i \F -  \sum_{j>i}  \frac{\F^t(\mbox{Id}-\Pi_{j+1}) \partial_j \F}{c_j} \frac{\partial_i \F^t(\mbox{Id}-\Pi_{j+1}) \partial_j \F}{c_j} \\
	&= \F^t \left( \mbox{Id} - \sum_{j>i} \frac{(\mbox{Id}-\Pi_{j+1}) \partial_j \F}{\| (\mbox{Id}- \Pi_{j+1}) \partial_j \F \|_{2}} \frac{\left((\mbox{Id}-\Pi_{j+1}) \partial_j \F \right)^t}{\| (\mbox{Id}- \Pi_{j+1}) \partial_j \F \|_{2}} \right) \partial_i \F \\
	&= \F^t \left( \mbox{Id}- \Pi_{i+1} \right) \partial_i \F .
\end{align*}
By Lemma \ref{S-le:mainlemma} we get $\left( \mbox{Id}- \Pi_{i+1} \right) \partial_i^{2} \F \equiv 0$.\\
Now we want to show that $h_{i} = G_{i}$. To this end, note that as $ ( \mbox{Id}- \Pi_{i+1}) \partial_{i}^{2} F \equiv 0$,
\begin{equation*}
  \partial_{i} h_{i} = \frac{ \partial_{i} F^{t} (\mbox{Id}- \Pi_{i+1}) \partial_{i} F}{c_{i}} \text{ is constant.} 
\end{equation*}
As $ \partial_{i} h_{i} =  \partial_{i}^{2} \left( -\log(g_{i} (G_{i}(x))) \right)/c_{i}$,
\begin{equation*}
   \partial_{i}^{2} \left( -\log g_{i} (G_{i}(x)) \right) \text{ is constant.}
\end{equation*} 
Using that $\partial_{i} G_{i}(x) =c_{i} \neq 0$ is constant, $z \mapsto \log(g_{i}(z))$ is a polynomial of degree two. As $g_{i}$ is the density of a centered distribution with variance one and this density is positive on $\mathbb{R}$, it is the density of a centered Gaussian distribution with variance one. For a centered Gaussian distribution with variance one, $ (- \log g_{i})'(z)= z$. Hence,  $h_{i}(x) = - \frac{g_{i}'(G_{i}(x))}{g_{i}(G_{i}(x))} = (- \log(g_{i}))' (G_{i}(x)) = G_{i}(x)$. This proves that $h_{i} = G_{i}$.\\
As $\left( \mbox{Id}- \Pi_{i+1} \right) \partial_i^{2} \F \equiv 0$,
\begin{align*}
c_i^2 &= c_i \partial_i \G_i \\
&= c_{i} \partial_{i} h_{i} \\
 &= \partial_i \left( \F^t \left( \mbox{Id}- \Pi_{i+1} \right) \partial_i \F \right) \\
&= \partial_i \F^t \left( \mbox{Id}- \Pi_{i+1} \right) \partial_i \F + 0  \\
&=\|(\mbox{Id} - \Pi_{i+1}) \partial_i \F \|_{2}^2.
\end{align*}
It remains to show that $\Pi_{i}$ is constant. We proved that 
$\left( \mbox{Id}- \Pi_{i+1} \right) \partial_i^{2} \F \equiv 0$. $\Pi_{i+1}$ is constant by induction assumption. Thus, $\partial_{i} \left( \left( \mbox{Id}- \Pi_{i+1} \right) \partial_i \F \right) \equiv 0$. By Proposition~\ref{S-prop:PLSEM2} (2), $\partial_{i} \F$ depends only on $x_{i}$. Thus, for $j=i$ the vector  $\left( \mbox{Id}- \Pi_{j+1} \right) \partial_j \F$ is constant. By induction assumption we also know that this is true for all $j >i$. By definition, we know that
\begin{align*}
   \Pi_{i} =  \sum_{j \ge i} \frac{(\mbox{Id}-\Pi_{j+1}) \partial_j \F}{\| (\mbox{Id}- \Pi_{j+1}) \partial_j \F \|_{2}} \frac{\left((\mbox{Id}-\Pi_{j+1}) \partial_j \F \right)^t}{\| (\mbox{Id}- \Pi_{j+1}) \partial_j \F \|_{2}}.
\end{align*}
As shown, the quantities on the right-hand side are constant. This concludes the proof by induction.
``$\Leftarrow$''  We will show $1)-4)$ of Proposition~\ref{S-prop:PLSEM2} to prove that $\G$ is a PLSEM-function of $\PX$. By Lemma~\ref{S-lem:constant-equivalence},  the vectors $  (\mbox{Id}- \Pi_{i+1}^\sigma) \partial_{\sigma^{-1}(i)} \F$ are constant. As $\F$ is twice continuously differentiable, $\G$ is twice differentiable as well. This proves $1)$. By part $2)$ of  Proposition~\ref{S-prop:PLSEM2}, $\partial_{k} \partial_{l} \F = 0$ for all $k \neq l$. Recall that the vector $(\mbox{Id}- \Pi_{i+1}^\sigma) \partial_{\sigma^{-1}(i)} \F$ is constant. Let $k \neq l$. Hence $\partial_k \partial_l \G_i = (\partial_k \partial_l \F )^{t} (\mbox{Id}- \Pi_{i+1}^\sigma) \partial_{\sigma^{-1}(i)} \F / \| (\mbox{Id}- \Pi_{i+1}^\sigma) \partial_{\sigma^{-1}(i)} \F \|_{2} = 0 $.
This proves that for all $k \neq l$, $\partial_k \partial_l \G = 0$, i.e. part $2)$ of  Proposition~\ref{S-prop:PLSEM2}. Now we want to show that $(\D\G_{i \sigma^{-1}(j)})_{ij}$ is lower triangular. By construction, $\D\G_{i \sigma^{-1}(j)}=\partial_{\sigma^{-1}(j)} \G_i = \partial_{\sigma^{-1}(j)} \F^t (\mbox{Id}- \Pi_{i+1}^\sigma) \partial_{\sigma^{-1}(i)} \F = 0$ for all $j > i$ as by definition $\partial_{\sigma^{-1}(j)} \F^t (\mbox{Id}- \Pi_{i+1}^\sigma) =0$. Now we want to show that  $(\D\G_{i \sigma^{-1}(j)})_{ij}$ has positive constant entries on the diagonal. Recall that by assumption $  (\mbox{Id}- \Pi_{i+1}^\sigma) \partial_{\sigma^{-1}(i)}^{2} \F \equiv 0$ for $i=1,\ldots,p$. Recall that the vector $  (\mbox{Id}- \Pi_{i+1}^\sigma) \partial_{\sigma^{-1}(i)} \F$ is constant. The vector is nonzero as $\D \F$ is invertible. Hence,
\begin{align*}
  \D\G_{i \sigma^{-1}(i)} &=  \frac{ \partial_{\sigma^{-1}(i)} \F^{t} (\mbox{Id}- \Pi_{i+1}^\sigma) \partial_{\sigma^{-1}(i)} \F }{\| (\mbox{Id}- \Pi_{i+1}^\sigma) \partial_{\sigma^{-1}(i)} \F\|_{2}} =\frac{\| (\mbox{Id}- \Pi_{i+1}^\sigma) \partial_{\sigma^{-1}(i)} \F \|_{2}^{2}}{\| (\mbox{Id}- \Pi_{i+1}^\sigma) \partial_{\sigma^{-1}(i)} \F\|_{2}}.
\end{align*}
Thus $\D\G_{i \sigma^{-1}(i)}$ is constant. This proves $3)$.
Let $X \sim \PX$. Now it remains to show that $\G(X) \sim \mathcal{N}(0,\mbox{Id})$.
To this end, note that by definition of  $\Pi_{j+1}^\sigma$, the vectors
\begin{equation*}
   (\mbox{Id}- \Pi_{j+1}^\sigma) \partial_{\sigma^{-1}(j)} \F, j=1,\ldots,p,
\end{equation*}
are orthogonal. As shown above, these vectors are constant and nonzero, therefore $(\mbox{Id}- \Pi_{j+1}^\sigma) \partial_{\sigma^{-1}(j)} \F, j=1,\ldots,p$ is an orthogonal basis of $\mathbb{R}^{p}$. Therefore, $\| \F(x) \|_2^2 = \| \G(x)\|_2^2$ for all $x \in \mathbb{R}^{p}$. As $\det \D\G$ is constant, we have by the change of variables formula that $|\det \D\G |= |\det \D\F |$ (probability densities integrate to one). Hence, again by the change of variables formula, $\G(X) \sim \mathcal{N}(0, \mbox{Id})$, which is $4)$. This concludes the proof of the ``if and only if'' statement. 

Lemma~\ref{S-lem:constant-equivalence} proves that in that case, the matrices $\Pi_{i+1}^{\sigma}$ and the vectors $(\text{Id}- \Pi_{i+1}^{\sigma}) \partial_{\sigma^{-1}(i)} \F$ are constant. This concludes the proof.
\end{proof}

\begin{lemm}\label{S-lemma:decomposability}
  Let $F : \mathbb{R}^{p} \mapsto \mathbb{R}$ be twice continuously differentiable. If $\partial_{k} \partial_{l} F \equiv 0$ for all $l \neq k$, then $F$ can be written in the form
\begin{equation}\label{S-eq:30}
  F(x) = c + g_{1}(x_{1}) + \ldots + g_{p}(x_{p}).
\end{equation}
In this case, the functions $g_{i}(x_{i})$ are unique up to constants and twice continuously differentiable.
\end{lemm}

\begin{proof}
Fix an arbitrary $y \in \mathbb{R}^{p}$. We use Taylor:
\begin{align*}
  F(x) - F(y) &= \int_{y_{1}}^{x_{1}} \partial_{1} F(z,x_{2},...,x_{p}) \mathrm{d}z + ... + \int_{y_{p}}^{x_{p}} \partial_{p} F(y_{1},y_{2},...,y_{p-1},z) \mathrm{d}z \\
&= \int_{y_{1}}^{x_{1}} \partial_{1} F(z,0,...,0) \mathrm{d}z + ... + \int_{y_{p}}^{x_{p}} \partial_{p} F(0,0,...,0,z) \mathrm{d}z
\end{align*}
In the second line we used that $\partial_{k} \partial_{l} F \equiv 0$ for all $l \neq k$. Now we can define 
\begin{equation*}
  g_{i}(x_{i}) = \int_{y_{i}}^{x_{i}} \partial_{i} F(0,...,0,z,0,...,0) \mathrm{d}z,
\end{equation*}
which proves equation~\eqref{S-eq:30} with constant $c= F(y)$. Furthermore, as
\begin{equation*}
  \partial_{i} F = \partial_{i} g_{i}(x_{i}),
\end{equation*}
the $g_{i}$ are unique up to constants. This completes the proof.
\end{proof}

\begin{lemm}\label{S-le:mainlemma}
Let $F$ be a PLSEM-function. Without loss of generality let us assume that the variables are ordered such that $\mathrm{D} F$ is lower triangular. Fix $l \in \{1,\ldots,p\}$ and consider a constant projection matrix $A \in \mathbb{R}^{p \times p}$.
 Let $S \subset \{1,\ldots,p\} \setminus \{l\} $ and let ${u}, {u}_{s}  : \mathbb{R} \rightarrow \mathbb{R}$, $s \in S$ be twice continuously differentiable functions. For notational convenience, we assume that for each $s \in S$ there exists $x_{s} \in \mathbb{R}$ such that ${u}'_{s}(x_{s}) \neq 0$. If
\begin{equation}\label{S-eqS:general}
  {u} \left(x_{l} + \sum_{s \in S} {u}_{s}(x_{s})\right)=  F^{t} (\mathrm{Id}- A) \partial_{l}   F, 
\end{equation}
then $(\mathrm{Id}-A) \partial_{l}^{2} F \equiv 0$.
\begin{proof}
\textbf{Part 1:} By assumption and as $F$ is a PLSEM-function, $\mathrm{D} F$ is lower triangular with positive constant entries on the diagonal. In particular, $\partial_{l}^{2} F \in \langle \partial_{k} F,  l < k \le p \rangle$. Hence to show $(\mathrm{Id}-A) \partial_{l}^{2} F \equiv 0$, as $A$ is a projection matrix, it suffices to show that
\begin{equation*}
  \partial_{k} F^{t} (\mathrm{Id}-A) \partial_{l}^{2} F \equiv 0 \text{ for } l < k \le p.
\end{equation*}
\textbf{Part 2:} 
Consider the case $k \in \{1,\ldots,p\} \setminus \{l\}$, $k \not \in S$.
By taking the derivative with respect to $x_{k}$ and $x_{l}$ in equation~\eqref{S-eqS:general},
\begin{equation*}
   \partial_{k} F^{t} (\mathrm{Id}- A) \partial_{l}^{2} F \equiv 0.
\end{equation*}
Here we used that as $F$ is a PLSEM-function, $\partial_{k} \partial_{l} F \equiv 0$.\\
\textbf{Part 3:} Consider the case in which $k \in S$ and there exists $k' \in S$ such that $k' \neq k$.
Choose $x_{k'}$ such that ${u}_{k'}'(x_{k'}) \neq 0$. Then,
\begin{align*}
\partial_{l}  F^{t} (\mathrm{Id} -A) \partial_{l}  F +  F^{t} (\mathrm{Id} -A)\partial_{l}^{2}  F &= \partial_{l}  \left({u} \left(x_{l} + \sum_{s \in S} {u}_{s}(x_{s})\right) \right) \\
= \frac{ \partial_{k'} \left(  {u} \left(x_{l} + \sum_{s \in S} {u}_{s}(x_{s})\right)  \right)}{{u}'_{k'}(x_{k'})}&=  \frac{\partial_{k'}  F^{t} (\mathrm{Id}-A) \partial_{l}  F }{{u}'_{k'}(x_{k'})}.
\end{align*}
In the last line we used that as $F$ is a PLSEM-function, $\partial_{k'} \partial_{l} F \equiv 0$. By rearranging and taking the derivative with respect to $x_{k}$,
\begin{align*}
\begin{split}
  \partial_{k} F^{t}(\mathrm{Id}-A) \partial_{l}^{2} F &=    \partial_{k} \left( F^{t}(\mathrm{Id}-A) \partial_{l}^{2} F  \right)\\ 
&=\partial_{k} \left(\frac{\partial_{k'}  F^{t} (\mathrm{Id}-A) \partial_{l}  F }{{u}'_{k'}(x_{k'})}- \partial_{l}  F^{t} (\mathrm{Id} -A) \partial_{l}  F \right) \\ 
&= 0.
\end{split}
\end{align*}
In the first line we used that as $F$ is a PLSEM-function, $\partial_{k} \partial_{l}^{2} F \equiv 0$. In the second line we used that $\partial_{k} \partial_{k'} F \equiv 0$ and $\partial_{k} \partial_{l} F \equiv 0$. Thus, $ \partial_{k} F^{t} (\mathrm{Id}-A) \partial_{l}^{2} F \equiv 0$. \\
\textbf{Part 4:}
Consider the case $S = \{k\}$, i.e.
\begin{equation}\label{S-eqS:generalsmall}
  {u} \left(x_{l} + {u}_{k}(x_{k})\right)= F^{t} (\mathrm{Id}- A) \partial_{l}  F.
\end{equation}
\textbf{Case 1:} If ${u}_{k}'(x_{k}) = 0$, then by equation~\eqref{S-eqS:general}
\begin{equation*}
0=  {u}_{k}'(x_{k}) {u}'  \left(x_{l} + {u}_{k}(x_{k})\right)= \partial_{k} \left( {u}  \left(x_{l} + {u}_{k}(x_{k})\right)\right) = \partial_{k} F^{t} (\mathrm{Id}- A) \partial_{l}  F,
\end{equation*}
and in particular $  \partial_{k} F^{t} (\mathrm{Id}- A) \partial_{l}^{2}  F =0$. Here we used that as $F$ is a PLSEM-function, $\partial_{k} \partial_{l} F \equiv 0$. Hence in this case,  $\partial_{k} F^{t} (\mathrm{Id}- A) \partial_{l}^{2}  F =0$ and we are done. \\
\textbf{Case 2:} Thus in the following we can assume that  ${u}_{k}'(x_{k})\neq 0$. By taking derivatives in equation~\eqref{S-eqS:generalsmall} with respect to $x_{j} , j \in \{1,\ldots,p\} \setminus \{l,k\}$, we obtain
\begin{equation}\label{S-eqS:almost}
  \partial_{j}  F^{t} (\mathrm{Id}- A) \partial_{l}  F \equiv 0  \text{ for } j \in \{1,\ldots,p\} \setminus  \{l,k\}.
\end{equation}
Here we used that as $F$ is a PLSEM-function, $\partial_{j} \partial_{l} F \equiv 0$. 
Analogously as in Part 3 one can show that
\begin{align*}
    F^{t}(\mathrm{Id}-A) \partial_{l}^{2} F  =\frac{\partial_{k}  F^{t} (\mathrm{Id}-A) \partial_{l}  F }{{u}'_{k}(x_{k})}- \partial_{l}  F^{t} (\mathrm{Id} -A) \partial_{l}  F.
\end{align*}
By taking the derivative with respect to $x_{k}$ on both sides,
\begin{align}\label{S-eqS:almostt}
 \partial_{k} F^{t}(\mathrm{Id}-A) \partial_{l}^{2} F   = \partial_{k} \left(  \frac{\partial_{k}  F^{t}}{{u}'_{k}(x_{k})}  \right) (\mathrm{Id}-A) \partial_{l} F.
\end{align}
Here we used that as $F$ is a PLSEM-function, $\partial_{k} \partial_{l} F \equiv 0$. As $F$ is a PLSEM-function and $\mathrm{D} F$ is lower triangular, $\mathrm{D} F$ has positive constant entries on its diagonal. Hence $\langle\partial_{j} F, k< j \le p \rangle = \langle e_{j}, k< j \le p \rangle $, where $e_{j} \in \mathbb{R}^{p}$ denotes the $j$-th unit vector. As discussed in Part 1, we only have to show $ \partial_{k} F^{t} (\mathrm{Id}-A) \partial_{l}^{2} F \equiv 0$ for $k > l$. In that case, by equation~\eqref{S-eqS:almost},
\begin{equation*}
  e_{j}^{t} (\mathrm{Id}- A) \partial_{l}  F \equiv 0  \text{ for } k < j \le p.
\end{equation*}
As $\partial_{k} F_{k} = c \neq 0$ is constant, and as $\mathrm{D} F$ is lower triangular,
\begin{equation*}
  \partial_{k}F^{t}  (\mathrm{Id}- A) \partial_{l}  F = c e_{k}^{t} (\mathrm{Id}- A) \partial_{l}  F.
\end{equation*}
By taking the derivative with respect to $x_{l}$,
\begin{equation}\label{S-eqS:almostttt}
  \partial_{k}F^{t}  (\mathrm{Id}- A) \partial_{l}^{2}  F = c e_{k}^{t} (\mathrm{Id}- A) \partial_{l}^{2}  F.
\end{equation}
Here we used that as $F$ is a PLSEM-function, $\partial_{k} \partial_{l} F \equiv 0$. Analogously it follows that
\begin{equation}\label{S-eqS:almosttt}
\partial_{k} \left(  \frac{\partial_{k}  F^{t}}{{u}'_{k}(x_{k})}  \right) (\mathrm{Id}-A) \partial_{l} F = \partial_{k} \left(\frac{1}{{u}'_{k}(x_{k})}  \right) c e_{k}^{t} (\mathrm{Id}-A) \partial_{l} F.
\end{equation}
By combining equation~\eqref{S-eqS:almostttt} and equation~\eqref{S-eqS:almosttt}  with equation~\eqref{S-eqS:almostt},
\begin{equation}\label{S-eqS:almosttttt}
  \partial_{k}F^{t}  (\mathrm{Id}- A) \partial_{l}^{2}  F = c e_{k}^{t} (\mathrm{Id}- A) \partial_{l}^{2}  F = \partial_{k} \left(\frac{1}{{u}'_{k}(x_{k})}  \right) c e_{k}^{t} (\mathrm{Id}-A) \partial_{l} F.
\end{equation}
If $ \partial_{k}F^{t}  (\mathrm{Id}- A) \partial_{l}^{2}  F =0$ for all $x_{l}$ and $x_{k}$, then the proof is finished. If not, then  by equation~\eqref{S-eqS:almosttttt} for some $x_{l}$ we have that $e_{k}^{t} (\mathrm{Id}- A) \partial_{l}  F \neq 0$ and  $e_{k}^{t} (\mathrm{Id}- A) \partial_{l}^{2}  F \neq 0$. Hence for all $x_{k}$ with ${u}'_{k}(x_{k}) \neq 0$
\begin{equation*}
   \partial_{k} \left(\frac{1}{{u}'_{k}(x_{k})} \right) = \frac{e_{k}^{t} (\mathrm{Id}-A) \partial_{l}^{2} F}{ e_{k}^{t} (\mathrm{Id}-A) \partial_{l} F} \neq 0,
\end{equation*}
which is constant in $x_{k}$. However, there exists no nonzero continuously differentiable function ${u}_{k} : \mathbb{R} \rightarrow \mathbb{R}$ that satisfies this, contradiction. Hence  $\partial_{k}F^{t} (\mathrm{Id}-A) \partial_{l}^{2} F \equiv 0$.\\
\textbf{Part 5:} By Part~2-4, $ \partial_{k}F^{t} (\mathrm{Id}-A) \partial_{l}^{2} F \equiv 0  $ for all $k \in \{ l+1,\ldots,p \}$. By Part~1, this finishes the proof.
\end{proof}

\end{lemm}

\begin{lemm}\label{S-lem:constant-equivalence}
  Let $\F$ be a PLSEM-function and $\sigma$ be a permutation on $\{1,\ldots,p\}$. Let
\begin{equation}\label{S-eq:12}
 \left( \text{Id} - \Pi_{i+1}^{\sigma} \right) \partial_{\sigma^{-1}(i)}^{2} \F \equiv 0, \text{ for }i=1,\ldots,p,
\end{equation}
where $\Pi_{i+1}^{\sigma}$ denotes the linear projection on $\langle  \partial_{\sigma^{-1}(i+1)} \F, \ldots,  \partial_{\sigma^{-1}(p)} \F \rangle$ and $\Pi_{p+1}^{\sigma} =0 \in \mathbb{R}^{p \times p}$. Then the matrices $ \Pi_{i+1}^{\sigma}$ and vectors $(\mbox{Id}- \Pi_{i+1}^{\sigma}) \partial_{\sigma^{-1}(i)} \F$ are constant for $i=1,\ldots,p$.
\end{lemm}

\begin{proof}
Let us first show that the projection matrices  $\Pi_{i+1}^{\sigma}$ are constant. For $i=p$ the claim is trivial as $\Pi_{p+1}^{\sigma} \equiv 0$ and hence by equation~\eqref{S-eq:12}, $\partial_{\sigma^{-1}(p) }^{2} \F \equiv 0$. For arbitrary $i$, equation~\eqref{S-eq:12} implies that $ \partial_{\sigma^{-1}(i)}^{2} \F $ is an element of the space $ \langle   \partial_{\sigma^{-1}(i+1)} \F, \ldots,  \partial_{\sigma^{-1}(p)} \F \rangle$. Furthermore, as $\F$ is a PLSEM-function, by Proposition~\ref{S-prop:PLSEM2},  $  \partial_{\sigma^{-1}(i)} \F$ only depends on $x_{\sigma^{-1}(i)}$. Using these two facts it follows inductively for $i=p-1,\ldots,1$ that the linear spaces  $ \langle   \partial_{\sigma^{-1}(i+1)} \F, \ldots,  \partial_{\sigma^{-1}(p)} \F \rangle$, $i=p-1,\ldots,1$ are constant in $x$. Hence the linear projections on these spaces are constant matrices. This proves that the matrices  $ \Pi_{i+1}^{\sigma}$, $i=1,\ldots,p$ are constant. Now we want to show that the vectors $(\mbox{Id}- \Pi_{i+1}^{\sigma}) \partial_{\sigma^{-1}(i)} \F$, $i=1,\ldots,p$ are constant. Recall that  $  \partial_{\sigma^{-1}(i)} \F$ only depends on $x_{\sigma^{-1}(i)}$. Hence $(\mbox{Id}- \Pi_{i+1}^{\sigma}) \partial_{\sigma^{-1}(i)} \F$ can only depend on  $x_{\sigma^{-1}(i)}$. Hence~it suffices to show that $\partial_{\sigma^{-1}(i)}(\mbox{Id}- \Pi_{i+1}^{\sigma}) \partial_{\sigma^{-1}(i)} \F = 0$. By equation~\eqref{S-eq:12} and as the matrix  $ \Pi_{i+1}^{\sigma}$ is constant, $\partial_{\sigma^{-1}(i)}(\mbox{Id}- \Pi_{i+1}^{\sigma}) \partial_{\sigma^{-1}(i)} \F =(\mbox{Id}- \Pi_{i+1}^{\sigma}) \partial_{\sigma^{-1}(i)}^{2} \F=0$. This concludes the proof.
\end{proof}

\subsection{Proof of characterization via causal orderings (Theorem \ref{M-thm:char-via-order})} \label{S-sec:prf-ordering}



\begin{rem} \label{S-rem:VindepofF}
Slight abuse of notation. A priori, $\mathcal{V}$ might depend on the
concrete choice of $\F$. However by Theorem~\ref{M-thm:char-via-order},
the set of permutations $\{ \sigma : \sigma(i) < \sigma(k) \text{ for
  all } (i,k) \in \mathcal{V} \}$ does only depend on $\PX$. 
\end{rem}

\begin{proof}
 Let  $\Pi_{i+1}^{\sigma}$ be the linear projection on $\langle \partial_{\sigma^{-1}(i+1)} \F, \ldots, \partial_{\sigma^{-1}(p)} \F \rangle$ and $\Pi_{p+1}^{\sigma} := 0 \in \mathbb{R}^{p \times p}$. 
By some algebra,
\begin{align*}
  &(\text{Id} - \Pi_{i+1}^{\sigma}) \partial_{\sigma^{-1}(i)}^{2} \F \equiv 0  \text{ for all } i \\
\iff &  \partial_{\sigma^{-1}(i)}^{2} \F(x) \in \langle \partial_{\sigma^{-1}(i+1)} \F(x), \ldots, \partial_{\sigma^{-1}(p)} \F(x) \rangle \text{ for all } i, \text{ for all }x \in \mathbb{R}^{p} \\
\iff & e_{\sigma^{-1}(j)}^{t} (\D \F )^{-1} \partial_{\sigma^{-1}(i)}^{2} \F \equiv 0 \text{ for all } j \le i\\
\iff & j > i \text{ for all } e_{\sigma^{-1}(j)}^{t} (\D \F )^{-1} \partial_{\sigma^{-1}(i)}^{2} \F \not \equiv 0 \\
\iff &  \sigma(j) > \sigma(i) \text{ for all } e_{j}^{t} (\D \F )^{-1} \partial_{i}^{2} \F \not \equiv 0.
\end{align*}
Here, $e_{j}, j=1,\ldots,p$ denotes the standard basis of
$\mathbb{R}^{p}$. By invoking Lemma~\ref{S-thm:func-char} and Remark~\ref{S-rem:construct-PLSEM} the assertion follows.
\end{proof}

\subsection{Interplay of nonlinearity~\&~faithfulness (Theorems~\ref{M-sec:nonl-faithf}~\&~\ref{M-sec:nonl-faithf-2})} \label{S-sec:prf-interplay}

  Let $C_{i} := \{ j \text{ child of } i \text{ in } D : \partial_{i}^{2} f_{j,i} \not \equiv 0 \}$ be the set of nonlinear children of $i$. Define $\mathcal{V}$ as in equation~\eqref{M-eq:defV} and consider the condition
 \begin{equation}\label{S-eq:linind}
\text{The functions  } f_{j,i}, j \in C_{i}  \text{ are linearly independent}.
\end{equation}

\begin{lemm}[Nonlinear descendants are identifiable under minor assumptions] \label{S-thm:nonlin-desc}
  Consider a PLSEM with DAG $D$ that generates $\PX$ and the corresponding PLSEM-function $\F$. Define $\mathcal{V}$ and $C_{i}$ as above. Fix $i \in \{1,\ldots,p\}$. Then:
\begin{enumerate}  
	\item[(a)]  Let equation~\eqref{S-eq:linind} hold and let $j \in C_{i}$.  Then $(i, j)  \in \mathcal{V}$.
        \item[(b)]  Let $(i,j) \in \mathcal{V}$. Then $j$ is a descendant of $i$ in every DAG $D'$ of a PLSEM that generates $\PX$.
	\item[(c)]  Let $\PX$ be faithful to $D$ and let equation~\eqref{S-eq:linind} hold. Consider a descendant $k$ of $i$, $k \neq i$. Then $(i , k) \in \mathcal{V}$ if and only if $k$ is descendant of one of the nonlinear children in $C_{i}$.
         \item[(d)]  Define $ \tilde{\mathcal{V}}$ as the transitive closure of $\mathcal{V}$. Let $k \neq i$. Then $(i,k) \in  \tilde{\mathcal{V}}$ if and only if $k$ is a descendant of $i$ in every DAG $D'$ of a PLSEM that generates~$\PX$.
  \end{enumerate}
\end{lemm}

\begin{rem}\label{S-rem:uniqueness-V}
  We follow the convention that $i$ is a descendant of itself. If equation~\eqref{S-eq:linind} holds for all $i$, then from (c) it follows that $\tilde{\mathcal{V}} = \mathcal{V}$. In particular, by (d), $k$ is descendant of one of the nonlinear children in $C_{i}$ if and only if $k$ is a descendant of $i$ in every DAG $D'$ of a PLSEM that generates $\PX$.
\end{rem}

\begin{proof}
We will first show (a) and (c).
Recall Proposition~\ref{S-prop:PLSEM2}. Without loss of generality let us assume that $\sigma = \text{Id}$, i.e. that $\D\F$ is lower triangular with constant positive entries on the diagonal.  Furthermore, $\partial_{i}^{2} \F_{l} = 0$ for all $l \le i$. Using this we obtain  $(\D\F)^{-1} \partial_{i}^{2} \F = (\D\F)_{\bullet,(i+1):p}^{-1} \partial_{i}^{2} \F_{(i+1):p}$, and
\begin{align}\label{S-eq:19}
&  e_{k}^{t} (\D\F)^{-1} \partial_{i}^{2} \F &&\equiv 0 \\
\iff &   e_{k}^{t} (\D\F)_{\bullet,(i+1):p}^{-1} \partial_{i}^{2} \F_{(i+1):p} &&\equiv 0. \nonumber
\end{align} 
Here the subindex $\bullet$ denotes all rows $1:p$. In the next step, we want to prove that 
\begin{equation}\label{S-eq:16}
  \langle \partial_{i}^{2} \F_{(i+1):p}(x_{i}) \rangle_{x_{i} \in \mathbb{R}} = \langle (e_{j})_{(i+1):p}  \rangle_{j \in C_{i}}.
\end{equation}
Note that as the components $1,\ldots,i$ of $\partial_{i}^{2} \F$ and $e_{j}, j \in C_{i}$ are zero, this is equivalent to showing that
\begin{equation*}
  \langle \partial_{i}^{2} \F(x_{i}) \rangle_{x_{i} \in \mathbb{R}} = \langle e_{j}  \rangle_{j \in C_{i}}.
\end{equation*}
As $\partial_{i}^{2} \F_{l} \equiv 0$ for all $l \not \in C_{i}$, we have
\begin{equation*}
 \langle \partial_{i}^{2} \F(x_{i}) \rangle_{x_{i} \in \mathbb{R}} \subseteq \langle e_{j}  \rangle_{j \in C_{i}}.
\end{equation*}
Let $\gamma \in \langle e_{j}  \rangle_{j \in C_{i}}$, $\gamma \neq 0$.
\begin{align*}
  \left(\partial_{i}^{2} \F(x_{i}) \right)^{t} \gamma = \sum_{j \in C_{i}} \partial_{i}^{2} \F_{j}(x_{i}) \gamma_{j}
\end{align*}
As the nonlinear children in $C_{i}$ are linearly independent, there exists an $x_{i} \in \mathbb{R}$ such that $\sum_{j} \partial_{i}^{2} \F_{j}(x_{i}) \gamma_{j} \neq 0$. Hence there exists no nonzero vector $\gamma$ in $\langle e_{j}  \rangle_{j \in C_{i}}$ that is orthogonal to $\langle \partial_{i}^{2} \F(x_{i}) \rangle_{x_{i} \in \mathbb{R}}$ and hence
\begin{equation*}
  \langle \partial_{i}^{2} \F(x_{i}) \rangle_{x_{i} \in \mathbb{R}} = \langle e_{j}  \rangle_{j \in C_{i}}.
\end{equation*}
As discussed above, this proves equation~\eqref{S-eq:16}.  By Proposition~\ref{S-prop:PLSEM2}, each column $l$ of $\D \F$ is a function of $x_{l}$ (i.e. constant in $x_{1},\ldots,x_{l-1},x_{l+1},\ldots x_{p}$). Hence, $(\D\F)_{\bullet,(i+1):p}^{-1}$ is a function of $x_{i+1},\ldots,x_{p}$.  In particular,  $(\D\F)_{\bullet,(i+1):p}^{-1}$ is constant in $x_{i}$. Now we can continue:
\begin{align}\label{S-eq:20}
   &   e_{k}^{t} (\D\F)_{\bullet,(i+1):p}^{-1} \partial_{i}^{2} \F_{i+1:p} &&\equiv 0\\ 
\iff &  e_{k}^{t} (\D\F)_{\bullet,(i+1):p}^{-1}(x_{i+1},\ldots,x_{p}) \partial_{i}^{2} \F_{i+1:p}(x_{i}) &&= 0 \text{ for all } x \in \mathbb{R}^{p} \nonumber \\ 
  \iff &  e_{k}^{t} (\D\F)_{\bullet,(i+1):p}^{-1}(x_{i+1},\ldots,x_{p}) (e_{j})_{(i+1):p} &&= 0 \text{ for all } j \in C_{i}, x \in \mathbb{R}^{p} \nonumber \\
\iff &  e_{k}^{t} (\D\F)^{-1}(x) e_{j} &&= 0 \text{ for all } j \in C_{i}, x \in \mathbb{R}^{p} \nonumber
\end{align}
As $\D \F$ is lower triangular with positive entries on the diagonal, $(\D\F)^{-1}$ is lower triangular too, with nonzero entries on the diagonal. As a result, $ e_{k}^{t} (\D\F)^{-1}(x) e_{k} \not \equiv 0$. So if $k \in C_{i}$, by equation~\eqref{S-eq:20}, $e_{k}^{t} (\D\F)_{\bullet,(i+1):p}^{-1} \partial_{i}^{2} \F_{i+1:p} \not \equiv 0$. By equation~\eqref{S-eq:19}, $  e_{k}^{t} (\D\F)^{-1} \partial_{i}^{2} \F \not \equiv 0$ and hence by definition of $\mathcal{V}$, $(i, k) \in \mathcal{V}$. This proves $(a)$.

Let $Z \sim \mathcal{N}(0,\text{Id}_{p})$. Let $X = \F^{-1}(Z)$. By Proposition~\ref{S-prop:PLSEM2}, $X \sim \PX$. Note that $X_{k} = e_{k}^{t} \F^{-1} (Z)$. We denote the partial derivative with respect to $z_{j}$ by $\partial_{j}^{z}$. Note that $x_{k}$ is constant in $z_{j}$ if and only if $\partial_{j}^{z} e_{k}^{t} \F^{-1} (z) =e_{k}^{t} (\D\F)^{-1} (\F^{-1}(z)) e_{j} \equiv 0$. Fix $j$. As there are bijective relationships between $x$ and $z$ and between $x_{1:(j-1)}$ and $z_{1:(j-1)}$,
\begin{align}\label{S-eq:21}
 &e_{k}^{t} (\D\F)^{-1}(x) e_{j} = 0 \ \ \text{  for all }x \in \mathbb{R}^{p} \\
 \iff&e_{k}^{t} (\D\F)^{-1} (\F^{-1}(z)) e_{j} = 0 \ \ \text{  for all } z \in \mathbb{R}^{p} \nonumber \\
\iff & x_{k} = e_{k}^{t } \F^{-1}(z) \text{ is constant in } z_{j}    &&& \nonumber \\
\iff & x_{k} = e_{k}^{t } \F^{-1}(z) \text{ is a function of } z_{1},\ldots, z_{j-1},z_{j+1},\ldots,z_{k}  &&& \nonumber \\
\Rightarrow & X_{k} \independent Z_{j} | Z_{1},\ldots, Z_{j-1}  &&& \nonumber \\
\Rightarrow & X_{k} \independent Z_{j} | X_{1},\ldots,X_{j-1} &&& \nonumber \\
\Rightarrow & X_{k} \independent X_{j} | X_{1},\ldots,X_{j-1}  &&& \nonumber \\
\Rightarrow & k \text{ is not a descendant of } j \text{ in } D.  &&& \nonumber
\end{align}
In the second to last line we used that $X_{j} = \sum_{j' \in \pa_{D}(j)} f_{j,j'}(X_{j'}) + \sigma_{j} Z_{j}$. In the last line we used faithfulness. The other direction follows from the definition of a PLSEM with DAG $D$: For all $j \in C_i$ we have that
\begin{align}\label{S-eq:22}
&k \text{ is a non-descendant of } j \text{ in } D \\ &\Rightarrow x_{k} = e_{k}^{t } \F^{-1}(z) \text{ is constant in } z_{j}. \nonumber
\end{align}
By combining equation~\eqref{S-eq:21} and equation~\eqref{S-eq:22},
\begin{equation*}
 e_{k}^{t} (\D\F)^{-1}(x) e_{j} = 0 \iff  k \text{ is a non-descendant of } j \text{ in } D.
\end{equation*}
Hence, by equation~\eqref{S-eq:19} and equation~\eqref{S-eq:20}
\begin{align*}
   &k \text{ is a non-descendant of } j \text{ in } D \text{ for all }  j \in C_{i} \\
\iff & e_{k}^{t} (\D\F)^{-1}_{\bullet,(i+1):p} (e_{j})_{(i+1):p} \equiv 0 \text{ for all } j \in C_{i} \\
\iff & e_{k}^{t} (\D\F)^{-1} \partial_{i}^{2} \F \equiv 0.
\end{align*}
This concludes the proof of $(c)$. Now let us turn to the proof of $(d)$. \\

Fix a DAG $D'$. $k$ is a descendant of $i$ in $D'$ if and only if for all causal orderings $\sigma$ of $D'$ we have $\sigma(i) < \sigma(k)$. Hence,
\begin{align}\label{S-eq:28}
\begin{split}
 & \text{$k$ is a descendant of $i$ in all DAGs $D'$ of a PLSEM that generates $\PX$ } \\
\iff & \text{for all causal orderings $\sigma$ of a DAG $D'$ of a PLSEM}  \\
 & \text{that generates $\PX$: $\sigma(i) < \sigma(k)$}.
 \end{split}
\end{align}
 By Theorem~\ref{M-thm:char-via-order} a permutation  $\sigma$ is a causal ordering of a DAG $ D'$ of a PLSEM that generates $\PX$ if and only if $\sigma(l) < \sigma(m)$ for all $(l,m)\in \mathcal{V}$: 
\begin{align}\label{S-eq:24}
\begin{split}
  & \text{for all causal orderings $\sigma$ of a DAG $D'$ of a PLSEM} \\
 &\text{that generates $\PX$: $\sigma(i) < \sigma(k)$} \\
\iff & \text{for all permutations $\sigma$ with $\sigma(l) < \sigma(m)$ for all $(l,m) \in \mathcal{V}$} \\
&\text{we have $\sigma(i) < \sigma(k)$} \\
\iff & (i,k) \in \tilde{\mathcal{V}}.
\end{split}
\end{align}
Combining equations~\eqref{S-eq:28} and \eqref{S-eq:24} concludes the proof of (d). (b) follows from (d): $(i,j) \in \mathcal{V}$ implies that $(i,j) \in \tilde{\mathcal{V}}$. This concludes the proof of (b).
\end{proof}

\subsection{Proofs of consistency and correctness of estimation procedures} \label{S-ssec:prf-consistency}
\subsubsection{Proof of Theorem~\ref{M-thm:consistencyRecursiveAlgorithm}} \label{S-ssec:proofConsistencyRecursiveAlgorithm}
\begin{proof}
We prove that for $n$ sufficiently large and with high probability, for any DAG $D^0 \in \DPXO$ and $D \in \CDO$ such that (without loss of generality) $D^0$ and $D$ only differ by reversal of a covered edge between nodes $i$ and $j$, 
\begin{equation} \label{S-eq:consistency_score}
\log(\hat{\sigma}_i^D) + \log(\hat{\sigma}_j^D) - \log(\hat{\sigma}_i^{D^0}) - \log(\hat{\sigma}_j^{D^0}) \geq 3 \xi_0 / 4 ,
\end{equation}
where $\hat{\sigma}_j^D$ denotes the unpenalized maximum likelihood estimator of the standard deviation of the residuals at node $j$ in DAG $D$.
Similarly, for $n$ sufficiently large and with high probability, for $D^0, \tilde{D}^0 \in \DPXO$ that only differ by a reversal of a covered linear edge between nodes $i$ and $j$, 
  \begin{equation} \label{S-eq:consistency_score2}
\left| \log(\hat{\sigma}_i^{D^0}) + \log(\hat{\sigma}_j^{D^0}) - \log(\hat{\sigma}_i^{\tilde{D}^0}) - \log(\hat{\sigma}_j^{\tilde{D}^0}) \right| \leq \xi_0 / 4 .
\end{equation}
The uniform bounds in~\eqref{S-eq:consistency_score} and~\eqref{S-eq:consistency_score2} imply that for $\alpha \in (\xi_0/4, 3\xi_0/4)$, each score-based decision whether a covered edge $i \rightarrow j$ is linear or not in step~6 of Algorithm~\ref{M-alg:listAllDAGsPLSEMscorebased} is consistent. The consistency of the estimated distribution equivalence class then follows from the correctness of Algorithm~\ref{M-alg:listAllDAGsPLSEM}, which is justified at the beginning of Section~\ref{M-ssec:estim_recursive}. Obviously, the constants in~\eqref{S-eq:consistency_score} and~\eqref{S-eq:consistency_score2} can be changed allowing for any $\alpha \in (0,\xi_0)$.

\textit{Proof of~\eqref{S-eq:consistency_score}.}
By Assumption~\ref{M-assu:consistency}~(i), all DAGs under consideration have uniformly bounded node degrees. 
It now follows exactly along the lines of Sections~2.1,~2.2,~2.3 and~5 in the supplement to~\citep{pbjoja13} that for 
$D \in \DPX \cup \CDPX$,
\begin{equation} \label{S-eq:cons1}
	( \sigma_k^{D} )^2 \leq ( \hat{\sigma}^D_k )^2 + \Delta_{n,k}^D, 
\end{equation}
and for $D^0 \in \DPXO$,
\begin{equation} \label{S-eq:cons2}
	( \hat{\sigma}_k^{D^0} )^2 \leq ( \sigma^{D^0}_k )^2 + \gamma_{n,k}^{D^0} + \Delta_{n,k}^{D^0} ,
\end{equation}
with
$\Delta_{n,k}^D$, $\gamma_{n,k}^{D^0}$ as defined in Assumptions~\ref{M-assu:consistency}~(iii) and~(iv).


Without loss of generality, let $D^0 \in \DPXO$ and $D \in \CDO$ such that $D^0$ and~$D$ only differ by reversal of a covered nonlinear edge between nodes $i$ and~$j$. Substituting~\eqref{S-eq:cons1} and~\eqref{S-eq:cons2} and then using~\eqref{M-eq:diffScoreCovRev} and~\eqref{M-eq:xip},
\begin{align*}
	\sum\limits_{k\in \{i,j\}} \left( \log(\hat{\sigma}_k^D) - \log(\hat{\sigma}_k^{D^0}) \right) &\geq 
	\sum\limits_{k\in \{i,j\}}\left(\log(\sigma_k^{D}) - \log(\sigma_k^{D^0})\right) + R_{n,D,D^0} \\
	&\geq \xi_p +  R_{n,D,D^0},
\end{align*}
where
\begin{align*}
	R_{n,D,D^0} = \frac{1}{2} \sum\limits_{k\in \{i,j\}} \left( \log\left(1 + \frac{-\Delta_{n,k}^{D}}{(\sigma_k^{D})^2}\right)
	- \log\left(1+ \frac{\gamma_{n,k}^{D^0} + \Delta_{n,k}^{D^0}}{(\sigma_k^{D^0})^2} \right) \right).
\end{align*}
By Assumption~\ref{M-assu:consistency}~(ii), the error variances are bounded away from zero. Using Taylor expansion and Assumptions~\ref{M-assu:consistency}~(iii) and~(iv), $R_{n,D,D^0} = o_P(1)$. 
As $\xi_p$ is uniformly bounded from below by $\xi_0$, we have that 
$$
	\sum\limits_{k\in \{i,j\}} \left( \log(\hat{\sigma}_k^D) - \log(\hat{\sigma}_k^{D^0}) \right) \geq \xi_0 + o_P(1).
$$
Therefore, for $n$ sufficiently large and with high probability,
$$
\sum\limits_{k\in \{i,j\}} \left( \log(\hat{\sigma}_k^D) - \log(\hat{\sigma}_k^{D^0}) \right) \geq 3 \xi_0 / 4.
$$
A completely analogous argument yields~\eqref{S-eq:consistency_score2} for $D^0, \tilde{D}^0 \in \DPXO$. 
\end{proof}

\subsubsection{Proof of Lemma~\ref{M-le:DAGextension}} \label{S-ssec:proofLemAlgorithm}
\begin{proof}
(a) 
For a DAG $D$ with $i \rightarrow j$ in $D$, we denote by $D_{i \leftarrow j}$ the graph that differs from $D$ only by the reversal of $i \rightarrow j$. Let $i \rightarrow j$ in $\mK$. Recall that $G_{P,\mathcal{K}}$ is obtained by imposing all edge orientations in $\mathcal{K}$ on the pattern $P$ and closing orientations under R1-R4 in Figure~\ref{M-fig:orientationRules}. Hence, by construction, $i\rightarrow j$ in $G_{P, \mK}$. Throughout the proof we use that for a background knowledge $\mK'$, by~\citep[Theorems~2\&4]{meek1995}, the set of all consistent DAG extensions of $G_{P,\mK'}$ equals the set of all Markov equivalent DAGs that have pattern $P$ and edge orientations that comply with the background knowledge $\mK'$.
 Then, it holds that
\begin{align*}
	& \text{$\exists$ consistent DAG extension $D$ of $G_{P, \mK}$ in which $i \rightarrow j$ is covered} \\
	\iff &\text{$\exists$ consistent DAG extension $D$ of $G_{P, \mK}$: $D_{i \leftarrow j}$ is Markov} \\
	&\text{equivalent to $D$} \\
	\iff &\text{$\exists$ consistent DAG extension $D$ of $G_{P, \mK\setminus \{i \rightarrow j\}}$: $D_{i \leftarrow j}$ is Markov} \\
	&\text{equivalent to $D$}, 
\end{align*}
where the first equivalence follows from \citep[Lemma~1]{chickering95}. By assumption, $i \rightarrow j$ is not in $P$ and not in the background knowledge. Hence, the soundness of the four orientation rules R1-R4~\citep[Theorem~2]{meek1995} implies that $i \text{ --- } j \text{ in } G_{P, \mK \setminus \{i \rightarrow j\}}$ which is the case if and only if $G_{P, \mK} \neq G_{P, \mK \setminus \{i \rightarrow j\}}$. \\
To show the other implication, by the completeness of the orientation rules in \citep[Theorem~4]{meek1995}, 
\begin{align*}
 & i \text{ --- } j \text{ in } G_{P, \mK \setminus \{i \rightarrow j\}} \\
 \Longrightarrow \ &\text{$\exists$ consistent DAG extensions $D_1,D_2$ of $G_{P, \mK\setminus \{i \rightarrow j\}}$: $i \rightarrow j$ in $D_1$ and } \\
	& \text{$i \leftarrow j$ in $D_2$}  \\
	\Longrightarrow \  &\text{$\exists$ consistent DAG extension $D$ of $G_{P, \mK\setminus \{i \rightarrow j\}}$: $D_{i \leftarrow j}$ is a consistent}  \\
	&\text{DAG extension of } G_{P, \mK \setminus \{i \rightarrow j\}}, 
\end{align*}
where the last implication follows from~\citep[Theorem~2]{chickering95}. Clearly, as both, $D$ and $D_{i \leftarrow j}$ are consistent DAG extension of the same pattern $P$, $D_{i \leftarrow j}$ is Markov equivalent to $D$, which finishes the proof.


(b) 
As $G_{P, \mK} \neq G_{P, \mK \setminus \{i \rightarrow j\}}$, by
  Lemma~\ref{M-le:DAGextension}~(a), there exists a consistent DAG
  extension of $G_{P, \mK}$ in which $i \rightarrow j$ is covered.
  From that, it immediately follows that all $k \in \pa_{G_{P,\mathcal{K}}}(i)$ are adjacent to $j$ in $G_{P,\mK}$. As $G_{P,\mathcal{K}}$ is closed under R1-R4, $k \rightarrow j$ in $G_{P,\mathcal{K}}$ due to R2. Concludingly, $\pa_{G_{P,\mathcal{K}}}(i) \subseteq \pa_{G_{P,\mathcal{K}}}(j)$. Analogously, for $k' \in \pa_{G_{P,\mathcal{K}}}(j)\setminus \{i\}$, either $k' \rightarrow i$ or $k' \text{ --- } i$ in $G_{P,\mathcal{K}}$ as $i\rightarrow j$ can be covered.\vspace{0.1cm}
  \\
  \textsc{Step~1}: Orient $k' \text{ --- } i$ with $k' \in \pa_{G_{P,\mathcal{K}}}(j)$ into $i$. 
  To be precise, define the new background knowledge $\widetilde{\mathcal{K}} := \mathcal{K} \cup (\ \bigcup\limits_{k' \in \pa_{G_{P,\mathcal{K}}}(j)} \{k' \rightarrow i\})$.
  As there is a consistent DAG extension $D$ of $G_{P, \mK}$ in which $i \rightarrow j$ is covered, $k' \rightarrow i$ in $D$ for all $k' \in \pa_{G_{P,\mathcal{K}}}(j)\setminus \{i\}$. Hence, by definition, $\mKt$ is consistent.\vspace{0.1cm} \\
 \textsc{Step~2}: Close orientations under R1-R4 to obtain the maximally oriented PDAG $G_{P,\mKt}$ with respect to the pattern $P$ and background knowledge $\mKt$~\citep{meek1995}.\vspace{0.2cm} \\
  \textit{Claim 1: Let $(x,y,z)$, $z \in \{i,j\}$ be a triple such
  that  $ x - y$ or $y - z$ in $G_{P, \mathcal{K}}$ and $x \rightarrow
  y \rightarrow z$ in $G_{P,\mKt}$.  
  Then $y=i, z=j$ or $x
  \rightarrow y \text{ --- } z$ in  $G_{P, \mathcal{K}}$.} \\
  Suppose that $y \neq i$ and $x - y$ in $G_{P, \mathcal{K}}$. 
  Then,  
  $x \rightarrow y$ in $G_{P,\mKt}$ was oriented in \textsc{Step~2} by applying R1-R4. We will lead this to a contradiction. By 
  Lemma~\ref{S-le:desc-of-y0}, $y$ is a descendant of $i$ in $G_{P,\mKt}$. Moreover, recall that $y \rightarrow z$ in $G_{P,\mKt}$. By construction, there exists a consistent DAG extension $D$ of $G_{P,\mKt}$ in which $i \rightarrow j$ is covered. As $y \neq i$ and $z \in \{i,j\}$, $y \rightarrow i$ and $y \rightarrow j$ in $D$. But $y$ is a descendant of $i$ in $D$, so $D$ contains a cycle. Contradiction.\vspace{0.2cm} \\
  \textit{Claim 2: In \textsc{Step~2}, no additional edges are oriented into $i$ or $j$.} 
  \\ 
  Suppose the contrary and consider the first edge that is oriented into $i$ or $j$ by applying one of R1-R4 in \textsc{Step~2}. We will consider the rules case-by-case. \\
\textit{R1:} Let $x$ be the upper, 
$y$ the 
lower left and $z$ the right node in R1 in Figure~\ref{M-fig:orientationRules}. Suppose $y \rightarrow z$, $z \in \{i,j\}$ is implied by R1 in \textsc{Step 2}. Thus, $x \rightarrow y \rightarrow z$ in $G_{P,\mKt}$ and $y - z$ in $G_{P,\mK}$. As $G_{P,
  \mK}$ is closed under R1-R4, 
  $x - y$ in $G_{P, \mK}$. Then, by \textit{Claim 1},
 $y=i$ and $z = j$. But $i \rightarrow j$ in   $G_{P, \mK}$, contradiction.\\
\textit{R2:} Let $x$ be the left,
$y$ the 
middle and $z$ the right node in R2 in Figure~\ref{M-fig:orientationRules}. Thus, $x \rightarrow y \rightarrow z$ with $x \rightarrow z$ in $G_{P,\mKt}$, and $x \text{ --- } z$ in $G_{P,\mK}$. \\
\textit{Case 1:} $y = i$. As $z \in \{i,j\}$, $z=j$. As $i \rightarrow
j$ and $x - z = j$ in $G_{P, \mK}$ and $G_{P,\mK}$ is closed under R1-R4, $x - y=i$ in $G_{P,\mK}$. Therefore,
$x \rightarrow y=i$ in $G_{P, \mKt}$ was either oriented in 
\textsc{Step~1} or \textsc{Step~2}.
It was not oriented in 
\textsc{Step~1}
as $x - z=j$ (that is, $x \not\in \pa_{G_{P,\mK}}(j)$). Also, it was
not 
oriented in \textsc{Step~2} as by assumption, $x \rightarrow z$ is the first edge
oriented into $i$ or $j$ 
in \textsc{Step~2}. Contradiction. 
\\
\textit{Case 2:} $y \neq i$. By \textit{Claim~1}, $x \rightarrow y - z$ in  $G_{P, \mK}$. By assumption, $x - z$ is the first edge that is oriented into $z \in \{i,j\}$ 
in \textsc{Step~2}. 
Hence, $y - z$ was oriented into $z$ in 
\textsc{Step~1} ($y \in \pa_{G_{P, \mK}}(j)$) and $x - z$ not ($x \not \in \pa_{G_{P, \mK}}(j)$). As $x$ is oriented into $z \in \{i,j\}$ in $G_{P,\mKt}$ and as there is a consistent DAG extension of $G_{P,\mKt}$ in which $i \rightarrow j$ is covered, $x-i$ and $x-j$ in $G_{P, \mK}$. Recall that $x \rightarrow y$ in  $G_{P, \mK}$ and $y \in \pa_{G_{P, \mK}}(j)$. 
As $G_{P, \mK}$ is closed under R1-R4, 
$x \rightarrow j$ in $G_{P, \mK}$ by R2, which contradicts $x \not \in \pa_{G_{P, \mK}}(j)$. \\
\textit{R3:} Does not apply. 
\textsc{Step~1} and \textsc{Step~2} do not create new $v$-structures. \\
\textit{R4:} Let $x$ be the 
upper left, $y$ the 
upper right and $z$ 
the lower right node in R4 in Figure~\ref{M-fig:orientationRules}. We have $x \rightarrow y \rightarrow z$ in $G_{P, \mKt}$ and $x - y$ or $y - z$ in $G_{P,\mK}$. Note that $x$ and $z$ are not adjacent in $G_{P,\mK}$ and $G_{P,\mKt}$. If $y=i$, then $z=j$ and $x$ is a parent of $i$ but not adjacent to $j$ in $G_{P,\mKt}$. This contradicts the fact that 
$i \rightarrow j$ is covered in 
a consistent DAG extension of $G_{P, \mKt}$. Hence, $y \neq i$.
By \textit{Claim 1}, $x \rightarrow y - z$ in $G_{P, \mK}$. As $x$ is not adjacent to $z$ in $G_{P,\mK}$,
$y \rightarrow z$ in 
 $G_{P, \mK}$ by R1, which contradicts the fact that $G_{P, \mK}$ is closed under R1-R4. This concludes the proof of \textit{Claim 2}.\vspace{0.2cm} \\
By \textit{Claim~2}, we do not orient edges into $i$ or $j$ 
in \textsc{Step~2}. Hence, by construction,
$G_{P, \mKt}$ satisfies $\pa_{G_{P, \mKt}}(i) = \pa_{G_{P, \mKt}}(j) \setminus \{i\} = \pa_{G_{P, \mK}}(j) \setminus \{i\}$.
By Lemma~\ref{S-le:oriented-away}, there exists a consistent DAG extension $D$ of $G_{P, \mKt}$ 
in which all undirected edges incident to $i$ and $j$ are oriented out of $i$ and $j$. By construction, $D$ is a consistent DAG extension of $G_{P, \mK}$. Moreover, $\pa_{D}(i) = \pa_{D}(j) \setminus \{i\} = \pa_{G_{P, \mK}}(j) \setminus \{i\}$, which concludes the proof.
\end{proof}

\begin{lemm} \label{S-le:desc-of-y0}
Consider the maximally oriented PDAG $G_{P, \mathcal{K'}}$ (with orientations closed under R1-R4) with respect to the pattern $P$ and consistent background knowledge $\mathcal{K'}$. Let $a_{m} \text{ --- } b$ in $G_{P, \mathcal{K'}}$ for all $1 \le m \le  M$ and assume there exists a consistent DAG extension of $G_{P,\mathcal{K'}}$ in which $a_{m} \rightarrow b$ for all $1 \le m \le M$. We orient $a_{m} \rightarrow b $ for all $m \le M$ and close the orientations under R1-R4. Let us denote the edges we orient $a_{m}' \rightarrow b_{m}'$, $m=1,2,...$. Then $b_{m}', m \ge 0$ are descendants of $b$. 
\end{lemm}

\begin{proof}
  By induction. By definition, $b$ is a descendant of $b$. At each step, apply one of 
  R1-R4 and 
  orient $a_{m}' \rightarrow b_{m}'$. This only 
  occurs if one of the directed edges in one of 
  R1-R4 is actually an edge $a_{k} \rightarrow b$ or  $a_{k}' \rightarrow b_{k}'$ that was 
  oriented at an earlier stage $1 \le k \le M$ (in the first case) or $k < m$ (in the second case). By the induction assumption, $b_{k}'$ is a descendant of $b$ for all $k < m$. By looking at Figure~\ref{M-fig:orientationRules} (i.e. going through the cases R1-R4) we can see that in each case, $b_{m}'$ is a descendant of $b$ (in the first case) or $b_{k}'$ (in the second case). Hence $b_{m}'$ is a descendant of $b$.
\end{proof}

\begin{lemm} \label{S-le:oriented-away}
Consider the maximally oriented PDAG $G_{P, \mathcal{K}}$  (with orientations closed under R1-R4) with respect to the pattern $P$ and consistent background knowledge $\mathcal{K}$. Let $x \rightarrow y$ in $G_{P, \mathcal{K}}$. Then, there exists a consistent DAG extension of $G_{P, \mathcal{K}}$ in which all undirected edges incident to $x$ and $y$ are oriented out of $x$ and $y$.
\end{lemm}

\begin{proof}
 Orient an undirected edge $e$ incident to $y$ out of $y$ and  close the orientations under R1-R4. By~\citep[Theorems~2~\&~4]{meek1995} the resulting PDAG is maximally oriented with respect to the pattern $P$ and consistent background knowledge $K \cup \{ e \}$. By Lemma~\ref{S-le:desc-of-y0}, no edge that is oriented in that process will point into $x$ or $y$. Now repeat, until there is no more undirected edge incident to $y$. Then, analogously orient 
 all undirected edges incident to $x$ out of $x$.
\end{proof}

\subsubsection{Proof of Lemma~\ref{M-le:correct-computeGDPX}}
\label{S-ssec:prf-le-correct-computeGDPX}

\begin{proof}
We first prove that for $k \geq 1$, 
by construction, each $G_{P, \mathcal{K}_{k}}$ is a consistent extension of $\GDPX$. This means that $G_{P,\mathcal{K}_{k}}$ and $G_{\DPX}$ have the same skeleton and $v$-structures and $i \rightarrow j$ in $G_{\DPX} \Rightarrow i \rightarrow j$ in $G_{P,\mathcal{K}_{k}}$. Then, we show that there exists a $k_0 \leq |\mathcal{K}_1^{\text{init}}| + 1$ such that either $\mathcal{K}^{\text{init}}_{k_0}=\emptyset$ or there is no edge in $\mathcal{K}^{\text{init}}_{k_0}$ that is covered in any of the consistent DAG extensions of $G_{P,\mathcal{K}_{k_0}}$. For $k_0$ it holds that $G_{P,\mathcal{K}_{k_0}} = \GDPX$. \\
By construction, $G_{P,\mathcal{K}_1} = D^0$ is a consistent extension of $\GDPX$. By Theorem~\ref{M-thm:covered-reversals}~(b), if $\mathcal{K}_1^{\text{init}}=\emptyset$ or none of the edges in $\mathcal{K}_1^{\text{init}}$ are covered in $D^0$, $\GDPX = D^0 = G_{P,\mathcal{K}_1}$. For a fixed $k \geq 1$, suppose that $G_{P,\mathcal{K}_k}$ is a consistent extension of $\GDPX$ and that there exists $\{i \rightarrow j\} \in \mathcal{K}^{\text{init}}_{k}$ that is covered in a consistent DAG extension of $G_{P,\mathcal{K}_k}$, which we denote by $D$. By assumption, as $G_{P,\mathcal{K}_k}$ is a consistent extension of $G_{\DPX}$, $D$ is a consistent DAG extension of $\GDPX$. Hence, $D \in \DPX$ by Theorem~\ref{M-thm:pdag}. \\ 
\underline{Case 1}: 
As $i \rightarrow j$ is covered and linear in $D \in \DPX$, by Theorem~\ref{M-thm:covered-reversals}~(a), there is a DAG $D' \in \DPX$ with $i \leftarrow j$. Therefore, by Definition~\ref{M-def:PDAGrepr}, $i \text{ --- } j$ in $G_{\DPX}$. \\
By construction, $\mathcal{K}_{k+1} = \mathcal{K}_k \setminus \{i \rightarrow j\}$. Hence, by Lemma~\ref{M-le:DAGextension}~(a), $G_{P,\mathcal{K}_{k+1}}$ equals $G_{P,\mathcal{K}_k}$ except for an undirected edge $i \text{ --- } j$ (all other directed edges in $G_{P,\mathcal{K}_k}$ are either directed in~$P$ or still contained in $\mathcal{K}_{k+1}$, hence they must be directed in $G_{P,\mathcal{K}_{k+1}}$). Therefore, $G_{P,\mathcal{K}_{k+1}}$ is a consistent extension of $G_{\DPX}$. \\
\underline{Case 2}: 
As $i \rightarrow j$ is covered and nonlinear in $D \in \DPX$, by Theorem~\ref{M-thm:covered-reversals}~(a), $i \rightarrow j$ in all DAGs in $\DPX$. Hence, $i \rightarrow j$ in $G_{\DPX}$ by Definition~\ref{M-def:PDAGrepr}. \\
By construction, $\mathcal{K}_{k+1} = \mathcal{K}_k$ and $G_{P,\mathcal{K}_{k+1}} = G_{P,\mathcal{K}_k}$ is a consistent extension of $G_{\DPX}$. Moreover, as $\{i \rightarrow j\} \not \in \mathcal{K}^{\text{init}}_{k+1}$ and $\{i \rightarrow j\} \in \mathcal{K}^{\text{nonl}}_{k+1}$, $i \rightarrow j$ is fixed in all $G_{P,\mathcal{K}_l}$ for $l > k$. 

In both cases, $|\mathcal{K}^{\text{init}}_{k+1}| = |\mathcal{K}^{\text{init}}_{k}| - 1$. Hence, there exists a $k_0 \leq |\mathcal{K}_1^{\text{init}}| + 1$ such that either $\mathcal{K}^{\text{init}}_{k_0} = \emptyset$ or no edge in $\mathcal{K}^{\text{init}}_{k_0}$ is covered in any of the consistent DAG extensions of $G_{P,\mathcal{K}_{k_0}}$. We will now prove that $G_{P,\mathcal{K}_{k_0}} = \GDPX$. If $\mathcal{K}^{\text{init}}_{k_0} = \emptyset$, this immediately follows from \underline{Case 1} and \underline{Case 2}. For $\mathcal{K}^{\text{init}}_{k_0} \neq \emptyset$,
suppose $G_{P,\mathcal{K}_{k_0}} \neq \GDPX$. Then,
there are $M \geq 1$ edges $i_m \text{ --- } j_m, m=1,...,M$ in $\GDPX$ with $i_m \rightarrow j_m$ in $G_{P,\mathcal{K}_{k_0}}$. By construction, $\{i_m \rightarrow j_m\}_{m=1,...,M} \subseteq \mathcal{K}_{k_0}$. From \underline{Case 2} it must hold that $\{i_m \rightarrow j_m\}_{m=1,...,M} \subseteq \mathcal{K}^{\text{init}}_{k_0}$. By Theorem~\ref{M-thm:covered-reversals}~(b), $\DPX$ is connected with respect to covered linear edge reversals. Hence, there exists an $1 \leq m_0 \leq M$ for which $i_{m_0} \rightarrow j_{m_0}$ is covered in a consistent DAG extension of $G_{P,\mathcal{K}_{k_0}}$. But $\{i_{m_0} \rightarrow j_{m_0}\} \in \mathcal{K}^{\text{init}}_{k_0}$, contradiction. We just showed that $\GDPX$ is a consistent extension of $G_{P,\mathcal{K}_{k_0}}$.
As by construction, $G_{P,\mathcal{K}_{k_0}}$ is a consistent extension of $\GDPX$, we conclude that $G_{P,\mathcal{K}_{k_0}} = \GDPX$, which finishes the proof.
\end{proof}


\subsection{Proofs of model misspecification}\label{S-sec:proofs-model-missp}
\subsubsection{Proof of Theorem~\ref{M-thm:lower-loglik}}

\begin{proof}
Recall PLSEM-functions and generalized PLSEM-functions, in particular Definition~\ref{S-def:PLSEM-function}, Definition~\ref{S-def:generalized-plsem-function} and Proposition~\ref{S-prop:PLSEM2}. 

Let $\mathring X$ be generated by a generalized PLSEM with DAG $D^{0}$. Let $X \sim  \PX$ be generated by a  PLSEM with same DAG, same edge functions, same error variances as the generalized PLSEM, but with  Gaussian errors.
Let $  G$ denote the corresponding PLSEM-function. As $D \in \mathscr{D}(\mathbb{P})$, there exists a PLSEM-function $ F$ of $ X$ with DAG $D$. By Lemma~\ref{S-thm:func-char}, there exists a constant orthonormal matrix $O$ such that $ F =O  G$.
Then,
\begin{align}\label{S-eq:3}
\begin{split}
  \left( \mathring \sigma_{j}^{D} \right)^{2} &=  \min_{g_{j,i} \in \mathcal{F}_{i}} \mathbb{E}[( \mathring X_{j} - \sum_{i \in \pa_{D}(i)} g_{j,i}( \mathring X_{i}))^{2}]  \\
&\le \mathbb{E}[( \mathring X_{j} - \sum_{i \in \pa_{D}(i)} f_{j,i}^{D}(  \mathring X_{i}))^{2}] \\
&= \mathbb{E}[   F(\mathring X)_{j}^{2}] \left( \sigma_{j}^{D} \right)^{2}.
\end{split}
\end{align}
Here $\{f_{j,i}^{D}\}_{i \in \pa_{D}(j)}$ and $(\sigma_{j}^{D})^{2}_{j=1,\ldots,p}$ denote the edge functions and the error variances in the PLSEM that corresponds to PLSEM-function $F$ (cf Remark~\ref{S-rem:plsem-function}). $G$ is not only a PLSEM-function that generates $X$ but also a generalized PLSEM function of $\mathring X$. This implies that $G(\mathring X)_{j}$ is centered with variance one, and in particular $ \mathbb{E} [G(\mathring X)_{j}^{2}] = 1$ for $j=1,\ldots,p$. Hence, using equation~\eqref{S-eq:3},
\begin{align*}
\begin{split}
  2 \log \left( \frac{\mathring \sigma_{j}^{D}}{ \sigma_{j}^{D}} \right) &\le \log \mathbb{E}[  F(\mathring X)_{j}^{2}] \\
&= \log \mathbb{E} [( O_{j \bullet}  G(\mathring X) )^{2}] \\
&= \log \left( \sum_{k=1}^{p} O_{jk}^{2} \right) \\
&= 0.
\end{split}
\end{align*}
Hence,
\begin{align}\label{S-eq:misspec1}
  \begin{split}
\sum_{j=1}^{p} \log( \mathring \sigma_{j}^{D}) \le \sum_{j=1}^{p} \log( \sigma_{j}^{D}).
  \end{split}
\end{align}
Under some permutation of the indices, $\mathrm{D} F$ is lower triangular with diagonal $\left(\frac{1}{\sigma_{i}^{D}} \right)_{i=1,\ldots,p}$. Thus,
\begin{equation*}
  \det \mathrm{D} F = \prod_{i=1}^{p} \frac{1}{\sigma_{i}^{D}}.
\end{equation*}
Analogously,
\begin{equation*}
    \det \mathrm{D} G = \prod_{i=1}^{p} \frac{1}{\sigma_{i}^{D^{0}}},
\end{equation*}
where $(\sigma_{j}^{D^{0}})^{2}_{j=1,\ldots,p}$ denote the error variances of the PLSEM corresponding to PLSEM-function $G$. 
As $F = O G$, $\det \mathrm{D} F = \det \mathrm{D} G$. Hence,
\begin{align}\label{S-eq:misspec2}
  \sum_{j=1}^{p} \log( {\sigma}_{j}^{D}) = \sum_{j=1}^{p} \log( {\sigma}_{j}^{D^{0}}).
\end{align}
Furthermore by construction $ \sigma_{j}^{D^{0}} =\mathring \sigma_{j}^{D^{0}}$. Using equation~\eqref{S-eq:misspec1} and equation~\eqref{S-eq:misspec2},
\begin{align*}
\begin{split}
  \sum_{j=1}^{p} \log(\mathring \sigma_{j}^{D}) &\le \sum_{j=1}^{p} \log( {\sigma}_{j}^{D}) \\
&=   \sum_{j=1}^{p} \log(  \sigma_{j}^{D^{0}})  \\
&=    \sum_{j=1}^{p} \log( \mathring \sigma_{j}^{D^{0}}).
\end{split}
\end{align*}
This concludes the proof.
\end{proof}

\subsubsection{Proof of Theorem~\ref{M-thm:close-loglik}}

\begin{proof}
Let  $D$ be Markov equivalent to $ D^{0}$ and $D \not \in \mathscr{D}(\mathbb{P})$. We want to show that $D \not \in \mathring{\mathscr{D}}(\mathring{\mathbb{P}})$. By definition of  $\mathscr{D}(\mathbb{P})$, there exists no (Gaussian) PLSEM with DAG $D$ that generates $X$.  By Lemma~\ref{S-thm:func-char}, there also exists no generalized PLSEM with DAG $D$ that generates $X$. Hence, $ p_{ \theta^{D}}^{D}$ is not the density of $X$. Thus there is a gap in the expected log-likelihood,
\begin{equation*}
\mathbb{E}[  \log  p_{ \theta^{D^{0}}}^{D^{0}}( X)] >  \mathbb{E}[   \log p_{ \theta^{D}}^{D}( X)].
\end{equation*}
Taking the minimum over all $D \not \in \mathscr{D}(\mathbb{P})$, $D \sim D^{0}$,
\begin{equation*}
  \zeta := \frac{1}{2} \min_{D \sim D^{0}, D \not \in \mathscr{D}(\mathbb{P})} \mathbb{E}[  \log  p_{ \theta^{D^{0}}}^{D^{0}}( X)] -  \mathbb{E}[   \log  p_{ \theta^{D}}^{D}(X)] > 0.
\end{equation*}
Using this and  $| \mathbb{E}[ \log p_{\theta^{D}}^{D} (X)] -  \mathbb{E}[  \log \mathring p_{\mathring \theta^{D}}^{D}(\mathring X)] | < \zeta $,
\begin{align*}
\begin{split}
  \mathbb{E}[  \log \mathring p_{ \mathring{\theta}^{D}}^{D}( \mathring X)] &< \mathbb{E}[ \log p_{ \theta^{D}}^{D}(  X)]  + \zeta \\ 
&<  \mathbb{E}[ \log p_{ \theta^{D^{0}}}^{D^{0}}( X)] - \zeta \\
&<   \mathbb{E}[ \log \mathring p_{\mathring{\theta}^{D^{0}}}^{D^{0}}( \mathring X)]. 
\end{split}
\end{align*}
Hence there is also a gap in the expected log-likelihood of the density $\mathring p_{\mathring \theta^{D^{0}}}^{D^{0}}$ and $  \mathring p_{\mathring \theta^{D}}^{D}$. Hence these densities are not equal and $D \not \in \mathring{\mathscr{D}}(\mathring{\mathbb{P}})$.
  
\end{proof}



\bibliographystyle{apalike} 
\bibliography{reference}

%
%
%
%


\end{document}